\documentclass[nosumlimits,twoside]{amsart}
\usepackage{amsfonts, amsmath, amssymb}
\usepackage{url}
\usepackage{graphicx}
\usepackage{float}
\usepackage{srcltx}
\usepackage[all]{xy}
\usepackage{version}
\usepackage[final]{showlabels} % put "finals" to turns off %inline
\usepackage[T1]{fontenc}
\usepackage{enumerate}
\usepackage[normalem]{ulem}
\usepackage{bm}
\usepackage{latexsym}
\usepackage[2emode]{psfrag}
\usepackage{yhmath}
\usepackage{array}
\usepackage{dsfont}
\usepackage{mathrsfs}% for \mathscr
\usepackage{tikz-cd}% diagrams
\usepackage{comment}
\usepackage{xcolor}
\usepackage{aligned-overset}

\newtheorem{theorem}{\sc Theorem}[section]
\newtheorem{proposition}[theorem]{\sc Proposition}

\newtheorem{lemma}[theorem]{\sc Lemma}
\newtheorem{corollary}[theorem]{\sc Corollary}
\theoremstyle{definition}

\theoremstyle{remark}
\newtheorem{remark}[theorem]{\sc Remark}

\newenvironment{invisible}{{\noindent\sc \colorbox{yellow}{Invisible:}\;}\color{gray}}{\medskip}

\usepackage{kantlipsum} % for text filler
\calclayout

\allowdisplaybreaks
\excludeversion{invisible}
\excludeversion{proof?}

\newcommand{\Cc}{\mathcal{C}}

\newcommand{\Mm}{\mathcal{M}}

\newcommand{\Ff}{\mathcal{F}}

\newcommand{\Rr}{\mathcal{R}}

\newcommand{\ot}{\otimes}

\newcommand{\rd}[1]{{\color{blue}{#1}}}
\newcommand{\mg}[1]{{\color{magenta}{#1}}}
\definecolor{darkgreen}{RGB}{10,210,60}
\newcommand{\gr}[1]{{\color{darkgreen}{#1}}}

\newcommand{\op}{\mathrm{op}}

\newenvironment{sistema}%
{
\left\{\begin{aligned}}
{\end{aligned}
\right.
}

\begin{document}

\title[Infinitesimal $\Rr$-matrices for some families of Hopf algebras]{Infinitesimal $\Rr$-matrices for some families of Hopf algebras}

\author{Lucrezia Bottegoni, Fabio Renda, Andrea Sciandra}

\address{\parbox[b]{\linewidth}{University of Turin, Department of Mathematics ``G. Peano'', via
Carlo Alberto 10, I-10123 Torino, Italy}}
\email{lucrezia.bottegoni@edu.unito.it}

\address{\parbox[b]{\linewidth}{University of Ferrara, Department of Mathematics and Computer Science, Via Machiavelli 30, 44121 Ferrara, Italy}}
\email{fabio.renda@unife.it}

\address{\parbox[b]{\linewidth}{University of Turin, Department of Mathematics ``G. Peano'', via
Carlo Alberto 10, I-10123 Torino, Italy}}
\email{andrea.sciandra@unito.it}

%18M05


\begin{abstract}
Given a  bialgebra $H$ such that the associated trivial topological bialgebra $H[[\hbar]]$ admits a quasitriangular structure $\tilde{\mathcal{R}}=\mathcal{R}(1\otimes 1+\hbar\chi+\mathcal{O}(\hbar^2))$, one gets a distinguished element $\chi \in H \otimes H$ which is an infinitesimal $\mathcal{R}$-matrix, according to the definition given in \cite{ABSW}. In this paper we classify infinitesimal $\mathcal{R}$-matrices for some families of well-known Hopf algebras, among which are the generalized Kac--Paljutkin Hopf algebras $H_{2n^2}$, the Radford Hopf algebras $H_{(r,n,q)}$, and the Hopf algebras $E(n)$. 
\end{abstract}

\maketitle

\tableofcontents

\section{Introduction}
Quasitriangular bialgebras, i.e. bialgebras together with a solution $\Rr$ of the quantum Yang–Baxter equation  \cite{Dr87}, have been widely studied in the literature.
Recently, in \cite{ABSW}, 
%the notion of braided \emph{pre-Cartier} category has been introduced as an infinitesimal version of a braided monoidal category, whose corresponding \emph{infinitesimal braiding} is a first-order deformation of the given braiding. As (co)quasitriangular bialgebras, i.e. bialgebras together with a solution of the quantum Yang–Baxter equation, namely the universal $\Rr$-matrix (or universal $\Rr$-form), see \cite{Dr87}, represent the algebraic counterpart of braided monoidal categories, 
\emph{pre-Cartier} quasitriangular bialgebras $(H,\Rr, \chi)$ have been introduced. They are quasitriangular bialgebras $(H,\Rr)$ equipped with a so-called \emph{infinitesimal $\mathcal{R}$-matrix} $\chi\in H\otimes H$, whose defining axioms involve the given quasitriangular structure $\Rr$. The motivating example is provided by the topological bialgebra: given a bialgebra $H$ and a quasitriangular structure $\tilde{\mathcal{R}}=\mathcal{R}(1\otimes 1+\hbar\chi+\mathcal{O}(\hbar^2))$ on the trivial topological bialgebra $H[[\hbar]]$ of formal power series in a parameter $\hbar$, by reading the axioms of $\tilde{\mathcal{R}}$ in the first order of $\hbar$ one gets precisely the defining axioms of the infinitesimal $\mathcal{R}$-matrix $\chi$. %The latter notion can then be defined for any (co)quasitriangular bialgebra $(H,\Rr)$, independently of a (co)quasitriangular structure $\tilde{\mathcal{R}}$ on $H[[\hbar]]$. 
Among their main properties, infinitesimal $\Rr$-matrices are $2$-cocycles in the cohomology for bialgebras as in \cite[XVIII.5]{Ka} and satisfy an infinitesimal quantum Yang–Baxter equation, see \cite[Theorem 2.21]{ABSW}. %\rd{The richness of these structures makes them quite rare as we will see in this paper}. 
As shown in \cite{ABSW}, examples of infinitesimal $\Rr$-matrices arise for instance on Sweedler's Hopf algebra, quantum $2\times 2$-matrices, $GL_q(2)$, and via Drinfel'd twist deformation. In this paper, we classify infinitesimal $\Rr$-matrices for some known families of quasitriangular Hopf algebras, such as the generalized Kac--Paljutkin Hopf algebras $H_{2n^2}$ \cite{Kac,Masuoka,P19}, the Radford Hopf algebras $H_{(r,n,q)}$  \cite{Rad}, and the Hopf algebras $E(n)$ \cite{BDG, CD}.

We begin our investigation by looking at the class of Hopf algebras of dimension $8$. In fact, %other than this class, 
according to the classification proven in \cite[Theorem 3.5]{St2} for Hopf algebras of dimension up to $11$ over an algebraically closed field of characteristic $0$, the only ones admitting a quasitriangular structure are Sweedler's Hopf algebra, %the Taft algebra of dimension $9$ (which have no quasitriangular structure \cite{Ge2}), 
group algebras, and a list of few $8$-dimensional Hopf algebras. %or the dual $(\Bbbk G)^*$ of group algebras with $G$ nonabelian. 
Among the last-mentioned, the only semisimple Hopf algebra is the \emph{Kac--Paljutkin Hopf algebra} $H_8$ defined in \cite{KP}, whose quasitriangular structures are described in \cite{Wa3}. A class of semisimple $2n^2$-dimensional Hopf algebras $H_{2n^2}$ was introduced in \cite{Kac} as a generalization of $H_8$. They are quasitriangular as well, see \cite{ZL}. In Theorem \ref{prop:H2n2} and Theorem \ref{thm:H8} we prove that $(H_{2n^2}, \mathcal{R})$ has no non-zero infinitesimal $\Rr$-matrices for every $n \geq 2$. %In addition we also obtain a complete classification of the Hochschild $2$-cocycles of $H_8$ (see Theorem \ref{2cocyclesH8}).

With respect to the class of quasitriangular non-semisimple Hopf algebras of dimension $8$, we first classify infinitesimal $\Rr$-matrices for the pointed quasitriangular Hopf algebras $((A''_{C_{4}})^{*},\Rr)$ (Proposition \ref{prop:infinitesimalA''C4}) and $(A_{C_{2}\times C_{2}},\Rr_a)$ (Proposition \ref{prop:chiAC2xC2}), and then in Proposition \ref{AC2-n} we extend the classification obtained for $A_{C_{2}\times C_{2}}$ to $A_{C_{2}^n}$. 
%\gr{\sout{We show that $(A''_{C_{4}})^{*}$ has no non-trivial infinitesimal $\Rr$-matrices and the same happens for $H_{(2,2,q)}$, which is a particular case of the Radford Hopf algebras $H_{(r,n,q)}$, \rd{introduced in \cite{Rad}}. \rd{Considering $r=1$ one recovers Taft Hopf algebras, which do not admit a quasitriangular structure as it is shown in \cite{Ge2}}. In Proposition \ref{prop:Rad} \rd{[Mettere Theorem?]} we indeed prove that the Radford Hopf algebras $H_{(r,n,q)}$, with $r\geq 2$, have no non-trivial infinitesimal $\Rr$-matrices.}}
Instead of classifying infinitesimal $\Rr$-matrices just for the remaining non-semisimple $8$-dimensional quasitriangular Hopf algebras $(E(2),R_A)$ and $(H_{(2,2,q)},\Rr_{s,\beta})$, we deal with the general cases and in Subsections \ref{radfordalgebras} and \ref{E(n)algebras} we determine the pre-Cartier structures of $(H_{(r,n,q)},\Rr)$, and $(E(n),R_A)$, respectively. When $r=1$, Radford Hopf algebras $H_{(1,n,q)}$ coincide with Taft Hopf algebras and they are not quasitriangular for $n>2$, see  \cite{Ge2}. On the other hand, when $r \geq 2$, some of these algebras admit a quasitriangular structure \cite{Rad}. In Theorem \ref{prop:Rad} we prove that $\chi=0$ is the only infinitesimal $\Rr$-matrix for any quasitriangular Radford algebra $(H_{(r,n,q)},\Rr)$.

We then consider the family of Hopf algebras $E(n)$ over a field $\Bbbk$ of characteristic $\mathrm{char}(\Bbbk)\neq 2$ \cite{BDG, CD}. This is the most relevant example considered in our paper as it provides us with an entire class of infinitesimal $\Rr$-matrices. In order to classify all these structures for $E(n)$, we first give a classification of its $2$-cocycles. To this aim we make use of some results on cohomology of smash products. In particular, since $E(n)\cong R\#\Bbbk G$ \--- where $R=\Bbbk\langle x_{P}\ |\ P\subseteq\{1,...,n\}\rangle$ and $G=\{1,g\}$ \--- we are able to determine the dimension of $\mathrm{H}^2(E(n), \Bbbk)$ and explicitly give a complement of $\mathrm{B}^2(E(n),\Bbbk)$ in $\mathrm{Z}^2(E(n),\Bbbk)$ (Theorem \ref{prop:2cocycle}). %Radford--Majid bosonization and smash products from \cite{MPSW, St}.
In Theorem \ref{prop:class-chi}, we prove that the $E(n)$'s admit an exhaustive $n^2$-dimensional space of infinitesimal $\Rr$-matrices given by
	\[
		\chi=\sum_{p,q=1}^n \gamma_{pq}gx_p \otimes x_q,
\]
and that they can be regarded as Cartier bialgebras when equipped with infinitesimal $\Rr$-matrices whose associated square matrix of coefficients is anti-symmetric. Moreover, in Proposition \ref{prop:2cob} we show that $E(n)$ is Cartier if, and only if, its infinitesimal $\Rr$-matrix is a $2$-coboundary. Finally, we give an affirmative answer to the quantization problem highlighted in \cite[Question 2.10]{ABSW}  for $E(n)$ and $A_{C_{2}^n}$.
\par

\medskip

The paper is organized as follows. In Section \ref{sub:Hoch} we recall some basics on pre-Cartier bialgebras and 2-cocycles for the cohomology described in \cite[XVIII.5]{Ka}. In Section \ref{sec:theory} we further investigate properties of infinitesimal $\Rr$-matrices in the Hopf algebra case. In particular, we point out that the 2-cocycle condition can replace either axiom \eqref{cqtr2} or axiom \eqref{cqtr3} in the definition of an infinitesimal $\Rr$-matrix (Proposition \ref{prop:newcharacterization}) and we prove that the Cartier condition is independent from the universal $\Rr$-matrix (Proposition \ref{cor:SSRchi}), which in turn will yield a necessary and sufficient condition for the infinitesimal $\Rr$-matrices of a quasitriangular Hopf algebra which are 2-coboundaries to be Cartier (Corollary \ref{corCartiercob}). Section \ref{sect:infinit} is devoted to the classification of infinitesimal $\Rr$-matrices for the aforementioned families of Hopf algebras. In Section \ref{sect:quant} we treat the quantization problem for the Hopf algebras considered in Section \ref{sect:infinit} that admit a non-zero pre-Cartier structure. 

\medskip

\noindent\textit{Notations and conventions}. All vector spaces are understood to be over a fixed field $\Bbbk$ and by linear maps we mean $\Bbbk$-linear maps. The unadorned tensor product $\otimes$ stands for $\otimes_{\Bbbk}$. We denote by $\tau$ the canonical flip $\tau_{X,Y}: X\otimes Y\to Y\otimes X$, $x\otimes y\mapsto y\otimes x$, for vector spaces $X,Y$. All linear maps whose domain is a tensor product will usually be defined on generators and understood to be extended by linearity. Algebras over $\Bbbk$ will be associative and unital and coalgebras over $\Bbbk$ will be coassociative and counital.
For an algebra the multiplication and the unit are denoted by $m$ and $u$, respectively, while for a coalgebra the comultiplication and the counit are denoted by $\Delta$ and $\varepsilon$, respectively. We write $H$ for a bialgebra over $\Bbbk$. In case $H$ admits an antipode, it is denoted by $S:H\to H$. We use the classical Sweedler’s notation %for calculations with the comultiplication 
and we shall write $\Delta(h)= h_{1}\otimes h_{2}$ for any $h\in H$, where we omit the summation symbol, and $\Delta^{\mathrm{op}}=\tau\Delta$. The cyclic group with $n$ elements will be denoted by $C_n$. The notation $C_n^m$ will be used to denote the direct product $C_n \times \dots \times C_n$ with $m$ factors. For an element $T=\sum_i T^i \otimes T_i \in H \otimes H$, we will also adopt the standard notation $T_{12}=\sum_i T^i  \otimes T_i \otimes 1=T \otimes 1$, $T_{23}=\sum_i 1 \otimes T^i  \otimes T_i =1 \otimes T$, $T_{13}=\sum_i T^i \otimes 1 \otimes T_i$, $T^{\mathrm{op}}=\tau_{H,H}(T)=T_i\otimes T^i$.
%The center of an algebra $A$ is denoted by $\mathscr{Z}(A)$. For any right (left) $H$-module we denote without distinction the action by $\cdot$, even when there are different actions involved. For a right (left) $H$-comodule $M$ we write the coaction as $\rho^r: M\to M\otimes H$, $\rho^r(m)=m_0\otimes m_1$, (resp. $\rho^l:M\to H\otimes M$, $\rho^l(m)= m_{-1}\otimes m_0$), for any $m\in M$.

\section{Preliminaries on pre-Cartier bialgebras}%\label{sect:background}
\label{sub:Hoch}
In this section we remind from \cite{ABSW} some preliminary definitions and results concerning infinitesimal $\Rr$-matrices.\medskip

\noindent\textbf{Pre-Cartier %categories and 
bialgebras.}
%First we recall some definitions from \cite{ABSW}. A \textit{pre-Cartier} category is a pre-additive braided monoidal category $(\Mm,\otimes,\sigma)$ together with a natural transformation $t:\otimes\to\otimes$ such that the identities
%\[
%t_{X,Y\otimes Z}=t_{X,Y}\otimes\mathrm{Id}_{Z}+(\sigma^{-1}_{X,Y}\otimes\mathrm{Id}_{Z})\circ(\mathrm{Id}_{Y}\otimes t_{X,Z})\circ(\sigma_{X,Y}\otimes\mathrm{Id}_{Z}), 
%\]
%\[
%t_{X\otimes Y,Z}=\mathrm{Id}_{X}\otimes t_{Y,Z}+(\mathrm{Id}_{X}\otimes\sigma^{-1}_{Y,Z})\circ(t_{X,Z}\otimes\mathrm{Id}_{Y})\circ(\mathrm{Id}_{X}\otimes\sigma_{Y,Z}),
%\]
%hold true for all objects $X,Y,Z$ in $\Mm$. In this case, $t$ is said to be an \textit{infinitesimal braiding} of $(\Mm,\otimes,\sigma)$. A pre-Cartier category $(\Mm,\otimes,\sigma,t)$ is called \textit{Cartier} if in addition
%\[
%\sigma_{X,Y}\circ t_{X,Y}=t_{Y,X}\circ\sigma_{X,Y}
%\]
%holds for all objects $X,Y$ in $\Mm$. In case the braiding is symmetric we refer to the above as symmetric (pre-)Cartier categories and the notion of symmetric Cartier category coincides with that introduced in \cite{Ca}.
%In \cite{ABSW} we discussed a class of examples of pre-Cartier categories modelled on a braided monoidal category of left $H$-modules, in case $H$ is a quasitriangular bialgebra. Let us recall from \cite{Ka} and \cite{Majid-book} some definitions and preliminary results beforehand.\\
Let us recall (see e.g. \cite[Ch. $2$]{Majid-book}) that a bialgebra $H$ is called \textit{quasitriangular} if there is an invertible element $\Rr=\Rr^i \otimes \Rr_i \in H\otimes H$ \--- called the \textit{universal $\Rr$-matrix} or \textit{quasitriangular structure} \--- such that $H$ is quasi-cocommutative, i.e.
\begin{equation}\label{qtr1}
    \Delta^\mathrm{op}(\cdot)=\Rr\Delta(\cdot)\Rr^{-1},
\end{equation}
and the hexagon equations
\begin{align}
    (\mathrm{Id}_H\otimes\Delta)(\Rr)&=\Rr_{13}\Rr_{12},\label{qtr2}\\
    (\Delta\otimes\mathrm{Id}_H)(\Rr)&=\Rr_{13}\Rr_{23},\label{qtr3}
\end{align}
are satisfied. If in addition $\Rr^{-1}=\Rr^\op$, then $(H,\Rr)$ is called \textit{triangular}.

\begin{comment}
Recall that, by \cite[Theorem VIII.2.4]{Kassel}, a quasitriangular bialgebra $(H,\Rr)$ satisfies the \textit{quantum Yang-Baxter} equations 
\begin{equation}\label{eq:QYB}
\Rr_{12}\Rr_{13}\Rr_{23}=\Rr_{23}\Rr_{13}\Rr_{12},\qquad\Rr^{-1}_{12}\Rr^{-1}_{13}\Rr^{-1}_{23}=\Rr^{-1}_{23}\Rr^{-1}_{13}\Rr^{-1}_{12}
\end{equation}
and $(\varepsilon\otimes\mathrm{Id}_{H})(\Rr^{\pm 1})=1_{H}=(\mathrm{Id}_{H}\otimes\varepsilon)(\Rr^{\pm 1})$. 
Moreover, we have 
\begin{equation}\label{eq:SotimesId(R)}
(S\otimes\mathrm{Id})(\Rr)=\Rr^{-1},\qquad (\mathrm{Id}\otimes S)(\Rr^{-1})=\Rr.
\end{equation}
Thus we have that
\begin{equation}\label{eq:SotimesS(R)}
(S\otimes S)(\Rr)=\Rr,\qquad (S\otimes S)(\Rr^{-1})=\Rr^{-1}.
\end{equation}
\end{comment}

We also recall that $(\varepsilon\ot\mathrm{Id})(\Rr)=1_{H}=(\mathrm{Id}\ot\varepsilon)(\Rr)$. Moreover, if $H$ is a Hopf algebra, the following equalities are satisfied:
\begin{align}
    (S\ot\mathrm{Id})(\Rr)&=\Rr^{-1}\label{eq:antipodeRinv}\\
    (\mathrm{Id}\ot S)(\Rr^{-1})&=\Rr\label{eq:antipodeR}
\end{align}
and hence $(S\ot S)(\Rr)=\Rr$, $(S\ot S)(\Rr^{-1})=\Rr^{-1}$, see e.g. \cite[Lemma 2.1.2]{Majid-book}.

It is known (see \cite[Theorem 9.2.4 and paragraph thereafter]{Majid-book}) that a bialgebra $H$ is quasitriangular if and only if the monoidal category ${}_H\Mm$ of left $H$-modules is braided, with triangular structures corresponding to ${}_{H}\Mm$ being symmetric.
%with braiding determined on objects $M,N\in{}_H\Mm$ by
%\begin{equation*}
    %\sigma^\Rr_{M,N}\colon M\otimes N\rightarrow N\otimes M,\qquad
    %m\otimes n\mapsto\Rr^\op\cdot(n\otimes m)=(\Rr_i\cdot n)\otimes(\Rr^i\cdot m).
%\end{equation*}
%Note that $(\sigma^\Rr_{M,N})^{-1}( n\otimes m)= \Rr^{-1}\cdot (m\otimes n)=(\overline{\Rr}^i\cdot m)\otimes (\overline{\Rr}_{i}\cdot n)$, where $\Rr^{-1}=\overline{\Rr}^{i}\otimes\overline{\Rr}_{i}$.

A (quasi)triangular bialgebra $(H,\Rr)$ is said to be \emph{pre-Cartier} \cite[Definition 2.1]{ABSW} if there is an element $\chi\in H\otimes H$ such that
\begin{align}
    \chi\Delta(\cdot)&=\Delta(\cdot)\chi\label{cqtr1}\\
%    \Rr\chi&=\chi^\op\Rr\label{ctr2}\\
    (\mathrm{Id}_H\otimes\Delta)(\chi)&=\chi_{12}+\Rr^{-1}_{12}\chi_{13}\Rr_{12}\label{cqtr2}\\
    (\Delta\otimes\mathrm{Id}_H)
    (\chi)&=\chi_{23}+\Rr^{-1}_{23}\chi_{13}\Rr_{23}\label{cqtr3}
\end{align}
hold. The element $\chi$ is called an \emph{infinitesimal $\Rr$-matrix}. A (quasi)triangular bialgebra $(H,\Rr)$ is \emph{Cartier} if it is pre-Cartier and the corresponding infinitesimal $\Rr$-matrix $\chi$ satisfies
\begin{equation}\label{eq:ctr2}
    \mathcal{R}\chi=\chi^\mathrm{op}\mathcal{R}
\end{equation}
in addition. By \cite[Remark 2.2, (iii)]{ABSW} the following equalities
    \begin{align}
(\mathrm{Id}\otimes \varepsilon)(\chi)&=0,\label{eq:eps-chi1}\\
%$\chi^i\otimes  \chi_i=\chi^i\otimes\varepsilon(\chi_i) 1_H+\chi^i\otimes \chi_i$, thus 
(\varepsilon\otimes\mathrm{Id})(\chi)&=0,  \label{eq:eps-chi2}
\end{align} hold true for every infinitesimal $\Rr$-matrix $\chi$. In \cite[Theorem 2.6]{ABSW} it is shown that the existence of an infinitesimal $\Rr$-matrix has a categorical interpretation: a quasitriangular bialgebra $(H,\Rr)$ is (pre-)Cartier if, and only if, the category $_{H}\Mm$ is braided (pre-)Cartier in the sense of \cite[Definition 1.1]{ABSW}, i.e. it is a braided monoidal category equipped with an \textit{infinitesimal braiding}. 
The same bijective correspondence holds if one considers (pre-)Cartier triangular bialgebras $H$ and symmetric (pre-)Cartier structures on ${}_H\mathcal{M}$. \medskip
%categories.
 %The corresponding infinitesimal braiding on ${}_H\Mm$ is given for all objects $M,N$ in ${}_H\Mm$ by
%\begin{equation}\label{t-chi}
    %t_{M,N}\colon M\otimes N\rightarrow M\otimes N,\qquad
    %m\otimes n\mapsto\chi\cdot(m\otimes n)=(\chi^i\cdot m)\otimes(\chi_i\cdot n),
%\end{equation}
%where $\chi=\chi^i\otimes\chi_i\in H\otimes H$ is the infinitesimal $\Rr$-matrix for $H$. 

\noindent\textbf{2-cocycles. }%\label{sub:Hoch}
Infinitesimal $\Rr$-matrices also happen to be related to Homology Theory. In \cite[Theorem 2.21]{ABSW} it is proven that any infinitesimal $\Rr$-matrix is also a $2$-cocycle in the cohomology for coalgebras defined as in \cite[XVIII.5]{Ka} (sometimes also called Hochschild cohomology for coalgebras in the finite-dimensional case, see e.g. \cite[page 44]{Doi}). We will use this fact for the classification of infinitesimal $\Rr$-matrices in some of our examples. 

Consider a bialgebra $(H, m, u,\Delta,\varepsilon)$ and regard the base field $\Bbbk $ as an $H$-bicomodule via the unit $u$.
%, i.e., with left and right coactions $k\mapsto1_{H}\otimes k$ and $k\mapsto k\otimes1_{H}$, respectively. 
Then, one can consider the cobar complex of $H$:
$$\xymatrix{\Bbbk\ar[r]^-{b^0}& H\ar[r]^-{b^1}& H\otimes H\ar[r]^-{b^2}& H\otimes H\otimes H\ar[r]^-{b^3}&\cdots }$$
where the differential $b^n:H^{\otimes n}\to H^{\otimes (n+1)}$ is given by $b^n=\sum_{i=0}^{n+1}(-1)^i\delta_n^i$ and $\delta^i_n : H^{\otimes n}\to H^{\otimes (n+1)}$ are the linear maps
\[
\delta^i_n(x_1\otimes \cdots\otimes x_n) = \begin{cases}
1\otimes x_1\otimes\cdots\otimes x_n, \text{ if } i=0\\
x_1\otimes\cdots\otimes x_{i-1}\otimes\Delta(x_i)\otimes x_{i+1}\otimes\cdots\otimes x_n, \text{ if } 1\leq i \leq n\\
x_1\otimes \cdots\otimes x_{n}\otimes 1, \text{ if } i=n+1.
\end{cases}
\]
Here, if $n=0$, we set $H^{\otimes 0}=\Bbbk$ and $\delta^0_0 (1_{\Bbbk})=\delta^1_0(1_{\Bbbk})=1_{H}$. For instance, for small values of $n$, we have $b^0(k)=0$ for $k\in\Bbbk$, $b^1(x)=1\otimes x-\Delta(x)+x\otimes 1$ for $x\in H$, $b^2(x\otimes y)=1\otimes x\otimes y-(\Delta\otimes\mathrm{Id}_H)(x\otimes y)+(\mathrm{Id}_H\otimes \Delta)(x\otimes y)-x\otimes y\otimes 1$, for $x\otimes y\in H\otimes H$, and so on. 

The elements in $\mathrm{Z}^n(H,\Bbbk):=\mathrm{Ker}(b^n)$ are the \textit{$n$-cocycles}, while the elements in $\mathrm{B}^n(H,\Bbbk):=\mathrm{Im}(b^{n-1})$ are the \textit{$n$-coboundaries}. The $n$-th cohomology group is $\mathrm{H}^n( H,\Bbbk)=\frac{\mathrm{Z}^n(H,\Bbbk)}{\mathrm{B}^n(H,\Bbbk)}$, for $n\geq1$. 
Observe that $\mathrm{Z}^{1}(H, \Bbbk)=P(H)$, the space of primitive elements of $H$, while $\mathrm{B}^{1}(H,\Bbbk)=\{0\}$, so $\mathrm{H}^{1}(H,\Bbbk)\cong P(H)$. Moreover, a 2-cocycle is an element $\chi\in H\otimes H$ such that
\begin{equation}\label{eq:Hoch2cocy}
\chi_{12}+(\Delta\otimes\mathrm{Id})(\chi)=\chi_{23}+(\mathrm{Id}\otimes\Delta)(\chi).
\end{equation}

In the next section we provide some results regarding the infinitesimal $\Rr$-matrices of a pre-Cartier quasitriangular Hopf algebra, that will be useful in the following.

\section{Some results on infinitesimal $\Rr$-matrices for Hopf algebras}\label{sec:theory}
We start by proving an equivalent characterization of infinitesimal $\Rr$-matrices for bialgebras, which highlights the fact that they are naturally $2$-cocycles in the cohomology for coalgebras described above.

\begin{proposition}\label{prop:newcharacterization}
An element $\chi \in H \otimes H$ is an infinitesimal $\Rr$-matrix for the quasitriangular bialgebra $(H,\Rr)$ if and only if it is a $2$-cocycle satisfying \eqref{cqtr1} and either \eqref{cqtr2} or \eqref{cqtr3} holds.
\end{proposition}

\begin{proof}
If $(H,\Rr,\chi)$ is a pre-Cartier quasitriangular bialgebra, then \eqref{cqtr1}-\eqref{cqtr3} hold and $\chi$ is a $2$-cocycle, by \cite[Theorem 2.21]{ABSW}.

Now suppose that $\chi \in H \otimes H$ is a $2$-cocycle satisfying \eqref{cqtr1} and \eqref{cqtr2}. We want to prove that \eqref{cqtr3} holds. We have
\begin{align*}
(\Delta\otimes\mathrm{Id}_H)
    (\chi)&\overset{\eqref{eq:Hoch2cocy}}{=}\chi_{23}+(\mathrm{Id}_H \otimes \Delta) (\chi)-\chi_{12}\\
    &\overset{\eqref{qtr1}}{=}\chi_{23}+\Rr_{23}^{-1}(\mathrm{Id}_H \otimes \Delta^{\mathrm{op}}) (\chi)\Rr_{23}-\chi_{12}\\
    &\overset{\eqref{cqtr2}}{=}\chi_{23}+\Rr_{23}^{-1}(\chi_{13}+\Rr^{-1}_{13}\chi_{12}\Rr_{13})\Rr_{23}-\chi_{12}\\
&=\chi_{23}+\Rr_{23}^{-1}\chi_{13}\Rr_{23}+\Rr_{23}^{-1}\Rr^{-1}_{13}\chi_{12}\Rr_{13}\Rr_{23}-\chi_{12}\\
     &\overset{\eqref{qtr3}}{=}\chi_{23}+\Rr_{23}^{-1}\chi_{13}\Rr_{23}+(\Delta \ot \mathrm{Id}_H)(\Rr^{-1})\chi_{12}(\Delta \ot \mathrm{Id}_H)(\Rr)-\chi_{12}\\
    &\overset{\eqref{cqtr1}}{=}\chi_{23}+\Rr_{23}^{-1}\chi_{13}\Rr_{23}+(\Delta \ot \mathrm{Id}_H)(\Rr^{-1})(\Delta \ot \mathrm{Id}_H)(\Rr)\chi_{12}-\chi_{12}\\
&=\chi_{23}+\Rr_{23}^{-1}\chi_{13}\Rr_{23}.
\end{align*}
Similarly, one can show that if $\chi \in H \ot H$ is a $2$-cocycle satisfying \eqref{cqtr1} and \eqref{cqtr3}, then $\chi$ also satisfies \eqref{cqtr2} and therefore $(H,\Rr, \chi)$ is a pre-Cartier quasitriangular bialgebra.
\begin{comment}
\begin{align*}
(\mathrm{Id}_H \otimes \Delta)
    (\chi)&\overset{\eqref{eq:Hoch2cocy}}{=}\chi_{12}+(\Delta \ot \mathrm{Id}_H ) (\chi)-\chi_{23}\\
    &\overset{\eqref{qtr1}}{=}\chi_{12}+\Rr_{12}^{-1}(\Delta^{op} \ot \mathrm{Id}_H ) (\chi)\Rr_{12}-\chi_{23}\\
    &\overset{\eqref{cqtr3}}{=}\chi_{12}+\Rr_{12}^{-1}(\chi_{13}+ \Rr_{13}^{-1}\chi_{23}\Rr_{13})\Rr_{12}-\chi_{23}\\
    &=\chi_{12}+\Rr_{12}^{-1}\chi_{13}\Rr_{12}+ \Rr_{12}^{-1}\Rr_{13}^{-1}\chi_{23}\Rr_{13}\Rr_{12}-\chi_{23}\\
     &\overset{\eqref{qtr2}}{=}\chi_{12}+\Rr_{12}^{-1}\chi_{13}\Rr_{12}+\Rr_{12}^{-1}\Rr^{-1}_{13}\chi_{23}(\mathrm{Id}_H \ot \Delta)(\chi)-\chi_{23}\\
   &\overset{\eqref{cqtr1}}{=}\chi_{12}+\Rr_{12}^{-1}\chi_{13}\Rr_{12}+\Rr_{12}^{-1}\Rr^{-1}_{13}(\mathrm{Id}_H \ot \Delta)(\chi)\chi_{23}-\chi_{23}\\
    &\overset{\eqref{qtr2}}{=}\chi_{12}+\Rr_{12}^{-1}\chi_{13}\Rr_{12}+\Rr_{12}^{-1}\Rr^{-1}_{13}\Rr_{13}\Rr_{12}\chi_{23}-\chi_{23}\\
    &=\chi_{12}+\Rr_{12}^{-1}\chi_{13}\Rr_{12}.
\end{align*}
\end{comment}
\end{proof}

\begin{remark}\label{rmk:Drinfeldtwist}
    %We recall that, given a bialgebra $H$ and a Drinfel'd twist $\mathcal{F}\in H\ot H$ on $H$ in the sense of \cite{Dr87}, one can define a new bialgebra $H_{\Ff}$ by modifying the coproduct \mg{$\Delta$}  of $H$ as $\Delta_{\Ff}(\cdot)=\Ff\Delta(\cdot)\Ff^{-1}$. If $(H,\Rr)$ is a quasitriangular bialgebra then so is $(H_{\Ff},\Rr_{\Ff}:=\Ff^{\mathrm{op}}\Rr\Ff^{-1})$, see e.g. \cite[Theorem 2.3.4]{Majid-book}. 
    By \cite[Theorem 2.17]{ABSW} and \cite[Remark 2.18]{ABSW} 
    %it is shown that, given a quasitriangular bialgebra $(H,\Rr)$ and a Drinfel'd twist $\Ff$ on $H$, $(H,\Rr,\chi)$ is a pre-Cartier quasitriangular bialgebra if and only if $(H_{\Ff},\Rr_{\Ff},\Ff\chi\Ff^{-1})$ is a pre-Cartier quasitriangular bialgebra. Choosing 
    with $\Ff=\Rr$ one obtains that %$(H_{\Ff},\Rr_{\Ff})=(H^{\mathrm{cop}},\Rr^{\mathrm{op}})$, see e.g. \cite[Example 2.3.6]{Majid-book}. Hence, 
    $(H,\Rr,\chi)$ is a (pre-)Cartier quasitriangular bialgebra if and only if $(H^{\mathrm{cop}},\Rr^{\mathrm{op}},\Rr\chi\Rr^{-1})$ is a (pre-)Cartier quasitriangular bialgebra, where $H^{\mathrm{cop}}$ is $H$ with opposite coproduct. Moreover, if $H$ is a Hopf algebra with antipode $S$, which is bijective by \cite[Proposition 1]{Radfordantipode}, then $H^{\mathrm{cop}}$ has antipode $S^{-1}$ (see e.g. \cite[Corollary III.3.5]{Ka}). Hence $(H,\Rr,\chi)$ is a (pre-)Cartier quasitriangular Hopf algebra if and only if $(H^{\mathrm{cop}},\Rr^{\mathrm{op}},\Rr\chi\Rr^{-1})$ is a (pre-)Cartier quasitriangular Hopf algebra.
\end{remark}

Now we focus on the product $\Rr\chi\Rr^{-1}$ in case $(H,\Rr,\chi)$ is a pre-Cartier quasitriangular Hopf algebra.
We will prove that in this case the Cartier property \eqref{eq:ctr2} can be restated without using the quasitriangular structure.

\begin{lemma}
    Let $(H,\Rr,\chi)$ be a pre-Cartier quasitriangular Hopf algebra. Then, the following equalities
\begin{align}
\chi\Rr^{-1}%\rd{\sout{=-S(\Rr^{j}\chi^{i})\ot\Rr_{j}\chi_{i}}}
=-(S \otimes \mathrm{Id}_H)(\Rr\chi)\label{eq:chi}\\
\Rr\chi%\rd{\sout{=-\chi^i\Rr^j\ot\Rr_jS(\chi_i)}}
=-( \mathrm{Id}_H \otimes S)(\chi \Rr^{-1})\label{eq:chi2}
\end{align}
are satisfied.
\end{lemma}

\begin{proof}
    We compute
\[
0\overset{\eqref{eq:eps-chi2}}{=}1_{H}\varepsilon(\chi^{i})\ot\chi_{i}=S(\chi^{i}_{1})\chi^{i}_{2}\ot\chi_{i}\overset{\eqref{cqtr3}}{=}S(1)\chi^i\ot\chi_i+S(\chi^i)\overline{\Rr}^j\Rr^k\ot\overline{\Rr}_j\chi_i\Rr_k=\chi+(S(\chi^i)\overline{\Rr}^j\ot\overline{\Rr}_j\chi_i)\Rr.
\]
Then, we obtain 
\[
\chi\Rr^{-1}=-S(\chi^i)\overline{\Rr}^j\ot\overline{\Rr}_j\chi_i\overset{\eqref{eq:antipodeRinv}}{=}-S(\chi^i)S(\Rr^j)\ot\Rr_j\chi_i=-S(\Rr^{j}\chi^{i})\ot\Rr_{j}\chi_{i}=-(S \otimes \mathrm{Id}_H)(\Rr\chi),
\]
so we get \eqref{eq:chi}. Similarly, we find
\[
0\overset{\eqref{eq:eps-chi1}}{=}\chi^{i}\ot\varepsilon(\chi_{i})1_{H}=\chi^{i}\ot \chi_{i_{1}}S(\chi_{i_{2}})\overset{\eqref{cqtr2}}{=}\chi^i\ot\chi_iS(1)+\overline{\Rr}^j\chi^i\Rr^k \ot\overline{\Rr}_j\Rr_k S(\chi_i)=\chi+\Rr^{-1}(\chi^i\Rr^k\ot\Rr_kS(\chi_i)).
\]
Therefore, we get 
\[
\Rr\chi=-\chi^i\Rr^k\ot\Rr_kS(\chi_i)\overset{\eqref{eq:antipodeR}}{=}-\chi^{i}\overline{\Rr}^{k}\ot S(\overline{\Rr}_{k})S(\chi_{i})=-\chi^{i}\overline{\Rr}^{k}\ot S(\chi_{i}\overline{\Rr}_{k})=-(\mathrm{Id}_{H}\ot S)(\chi\Rr^{-1}),
\]
so  \eqref{eq:chi2} holds true.
\end{proof}
\begin{proposition}\label{cor:SSRchi}
In any pre-Cartier quasitriangular Hopf algebra $(H,\Rr,\chi)$ the following equalities
\begin{align}
(S \otimes S)(\chi)&=\Rr\chi\Rr^{-1}\label{eq:RchiRinv}\\
(S^2 \otimes S^2)(\chi)&=\chi\label{eq:S^2chi}
\end{align}
hold. As a consequence, a pre-Cartier quasitriangular Hopf algebra $(H,\Rr,\chi)$ is Cartier if and only if 
\begin{equation}\label{eq:HopfCartier}
\tau(S \otimes S)(\chi)=\chi.   
\end{equation}
\end{proposition}

\begin{proof}
We compute $(S\ot S)(\Rr\chi)\overset{\eqref{eq:chi}}{=}-(\mathrm{Id}_H\ot S)(\chi\Rr^{-1})\overset{\eqref{eq:chi2}}{=}\Rr\chi$. Therefore, recalling \eqref{eq:antipodeRinv} and \eqref{eq:antipodeR}, we obtain 
\[
\Rr\chi\Rr^{-1}=(S\ot S)(\Rr\chi)\Rr^{-1}=(S\ot S)(\Rr\chi)(S\ot S)(\Rr^{-1})=(S\ot S)(\Rr^{-1}\Rr\chi)=(S\ot S)(\chi),
\]
so we get \eqref{eq:RchiRinv}. Furthermore
\[(S^2 \ot S^2)(\chi)\overset{\eqref{eq:RchiRinv}}{=}(S \ot S)(\Rr\chi\Rr^{-1})=(S \ot S)(\Rr^{-1})(S \ot S)(\chi)(S \ot S)(\Rr) \overset{ \eqref{eq:RchiRinv}}{=}\Rr^{-1}\Rr\chi\Rr^{-1}\Rr=\chi,\]
so \eqref{eq:S^2chi} is satisfied.
\end{proof}

\begin{remark}
Given the \textit{Drinfel'd element} $u:=S(\Rr_i)\Rr^i$ of $(H,\Rr)$, since $S^{2}(a)=uau^{-1}$ for any $a\in H$, by \cite[Theorem 1]{Radfordantipode}, equality \eqref{eq:S^2chi} can also be restated as $(u \otimes u)\chi=\chi(u \otimes u)$.
\end{remark}

As a byproduct of Remark \ref{rmk:Drinfeldtwist} and Proposition \ref{cor:SSRchi}, we obtain the following result:

\begin{corollary}\label{cor:infinitesimalcop}
    Let $(H,\Rr)$ be a quasitriangular Hopf algebra. Then, $(H,\Rr,\chi)$ is (pre-)Cartier if and only if $(H^{\mathrm{cop}},\Rr^{\mathrm{op}},(S\ot S)(\chi))$ is (pre-)Cartier.
\end{corollary}

Proposition \ref{cor:SSRchi} allows us to obtain further information on the \textit{Casimir element} $\gamma:=S(\chi^{i})\chi_{i}$, see \cite[p. 26]{ABSW}. 

\begin{corollary}
 Let $(H,\Rr,\chi)$ be a pre-Cartier quasitriangular Hopf algebra. Then: 
 \begin{enumerate}
     \item[1)] $\gamma=\chi^iS(\chi_i)$,
     \item[2)] $S(\gamma)-\gamma$ is a primitive element. As a consequence, $S^{2}(\gamma)=\gamma$.
 \end{enumerate}
\end{corollary}
\begin{proof}
In view of \eqref{eq:RchiRinv} we have
\[S(\chi^i)\Rr^j \otimes S(\chi_i)\Rr_j = \Rr^j\chi^i \otimes \Rr_j\chi_i.\]
If we apply  $m(\mathrm{Id}_{H}\ot S)$ on both sides we find
\[S(\chi^i)\Rr^j S(\Rr_j)S^2(\chi_i) = \Rr^j\chi^iS(\chi_i)S(\Rr_j).\]
Since $S(u)=\Rr^{j}S(\Rr_{j})$ and $S^{2}(a)=uau^{-1}$ for all $a\in H$, the latter equality is clearly equivalent to $S(\chi^i)S(u)u\chi_i u^{-1} = \Rr^j\gamma'S(\Rr_j)$. The element $\chi^iS(\chi_i)$ is central %(similarly to see 
\cite[p. 26]{ABSW} %, where $\gamma$ is shown to be central) 
 and the element $S(u)u=uS(u)$ is as well (see \cite[Corollary VIII.4.2]{Ka} or \cite[Corollary 2.1.9]{Majid-book}), thus we get that
$\gamma S(u)uu^{-1} = \chi^iS(\chi_i)S(u)$, hence $\gamma=\chi^iS(\chi_i)$.

Moreover, we have
\begin{align*}
    b^1(S(\gamma))&=1_H \otimes S(\gamma)+ S(\gamma) \otimes 1_H-\Delta(S(\gamma))=\tau(S \otimes S)[\gamma \otimes 1_H+1_H \otimes \gamma -\Delta(\gamma)]\\
    &=\tau(S \otimes S)(b^1(\gamma))\overset{(\star)}{=}\tau(S \otimes S)[\chi +\tau(S \otimes S)(\chi)]=\tau(S \otimes S)(\chi) +(S^2 \otimes S^2)(\chi)\\
   \overset{\eqref{eq:S^2chi}}&{=}\tau(S \otimes S)(\chi) +\chi\overset{(\star)}{=}b^1(\gamma),
\end{align*}
where $(\star)$  uses that $b^{1}(\gamma)=\chi+\tau(S\ot S)(\chi)$ by \cite[Proposition 2.29]{ABSW}. As a consequence, $S(\gamma)-\gamma$ is a primitive element and then $S(S(\gamma)-\gamma)=-(S(\gamma)-\gamma)$, so $S^{2}(\gamma)=\gamma$.
\end{proof}

In \cite[Theorem 2.31]{ABSW}, it was shown that, given a pre-Cartier triangular Hopf algebra $(H,\Rr,\chi)$ (and $\mathrm{char}(\Bbbk)\not=2$), if $(H,\Rr,\chi)$ is Cartier then $\chi$ is a 2-coboundary. This was obtained observing that $b^1(\gamma)=\chi+\tau(S\ot S)(\chi)$ by \cite[Proposition 2.29]{ABSW} and $\tau(S\otimes S)(\chi)=\chi$ by \cite[Lemma 2.28]{ABSW}. Proposition \ref{cor:SSRchi} ensures that this is true even in the quasitriangular case. We can further extend this result providing a necessary and sufficient condition for the infinitesimal $\Rr$-matrices of a quasitriangular Hopf algebra which are 2-coboundaries to be Cartier.

\begin{corollary}\label{corCartiercob}
    Let $(H,\Rr,\chi)$ be a pre-Cartier quasitriangular Hopf algebra, where $\chi=b^1(a)$ for $a\in H$. Then $(H,\Rr,\chi)$ is Cartier if and only if $a-S(a)$ is a primitive element.
\end{corollary}

\begin{proof}
Since $(H,\Rr,\chi)$ is a pre-Cartier quasitriangular Hopf algebra, by Proposition \ref{cor:SSRchi}, we know that $(H,\Rr,\chi)$ is Cartier if and only if $\tau(S\ot S)(\chi)=\chi$. Since $\chi=b^1(a)$, we have
\[
\tau(S\ot S)(\chi)=\tau(S\ot S)(1\ot a-\Delta(a)+a\ot 1)=S(a)\ot1-\Delta(S(a))+1\ot S(a)=b^1(S(a)),
\]
hence $(H,\Rr,\chi)$ is Cartier if and only if $b^1(S(a))=b^1(a)$.
\end{proof}

\begin{remark}\label{Cartiercob}
If $\dim_{\Bbbk} H < \infty$ and $\mathrm{char}(\Bbbk)=0$, then any primitive element of $H$ is zero (see, for instance, \cite[p. 294]{Radfordprimitive}). In this case, the pre-Cartier quasitriangular Hopf algebra $(H,\Rr, \chi=b^1(a))$ is Cartier if and only if $a=S(a)$.
\end{remark}

\begin{proposition}
Let $(H,\Rr,\chi)$ be a pre-Cartier quasitriangular Hopf algebra. If  $\chi=b^1(a)$, then $\varepsilon(a)=0$ and $a S(a)=S(a)a$. If $(H,\Rr,\chi)$ is Cartier we also have that $S^2(a)=a$.
\end{proposition}

\begin{proof}
By \eqref{eq:eps-chi1}, we have $\varepsilon(a)=0$. Recalling that $\gamma=S(\chi^{i})\chi_{i}=S(1)a-\varepsilon(a)1+S(a)=a+S(a)$ is central (see \cite[p. 26]{ABSW}), we obtain $a^{2}+S(a)a=(a+S(a))a=a(a+S(a))=a^{2}+aS(a)$. Then $aS(a)=S(a)a$.

If $(H,\Rr,\chi=b^1(a))$ is Cartier, then $a-S(a)$ must be a primitive element by Corollary \ref{corCartiercob}  and therefore $S(a-S(a))=-(a-S(a))$. This leads to $S^2(a)=a$. 
\end{proof}

\begin{remark}
We observe that $S^2(a)=a$ is also equivalent to $au=ua$, since $S^{2}(a)=uau^{-1}$ by \cite[Theorem 1]{Radfordantipode}. Moreover, $S^{2}(a)=a$ is also equivalent to $S^{4}(a)=S^{2}(a)$ and, since $S^{4}(a)=uS(u^{-1})aS(u^{-1})^{-1}u^{-1}$ (see e.g. \cite[Proposition 12.3.2 (c)]{Rad3}), this is equivalent to $S(u^{-1})a=aS(u^{-1})$ (hence to $aS(u)=S(u)a$). In particular, $S^{2}(a)=a$ implies that $a\mathsf{g}=\mathsf{g}a$, where $\mathsf{g}:=uS(u^{-1})$ is a group-like element of $H$ (see e.g. \cite[Proposition 12.3.2 (c)]{Rad3}).
\end{remark}

In the next section we classify infinitesimal $\Rr$-matrices for some relevant families of Hopf algebras.

\section{Infinitesimal $\Rr$-matrices for some well-known Hopf algebras}\label{sect:infinit}
We start with Hopf algebras of low dimension. In \cite[Theorem 3.5]{St2} Hopf algebras of dimension up to 11 over an algebraically closed field of characteristic 0 are completely classified. Let us observe that, if we consider Hopf algebras of dimension different from 8, we only have group algebras and the dual $(\Bbbk G)^*$ of group algebras with $G$ nonabelian,  Sweedler's Hopf algebra of dimension 4 and the Taft algebra of dimension 9. Group algebras $\Bbbk G$, with $G$ abelian, have only infinitesimal $\Rr$-matrix $\chi=0$, as it was proved in \cite[Example 2.3 (2)]{ABSW}. %On the other hand, the existence of a (non trivial) universal $\Rr$-matrix on $\Bbbk G$ forces the group $G$ to be abelian, see e.g. \cite[p. 193]{AM} for the dual case.
%\gr{This is not clear since in the article $G$ is abelian, but here $G$ is not necessarily}. 
For any finite nonabelian group $G$, the dual $(\Bbbk G)^*$ has no quasitriangular
structure, see \cite[Proposition 117]{CMZ02}. Infinitesimal $\Rr$-matrices for Sweedler's Hopf algebra are classified in \cite[Proposition 2.7]{ABSW} while, as it is shown in \cite{Ge2}, Taft Hopf algebras of dimension greater than $4$ do not admit a quasitriangular structure. Therefore we will focus on the class of Hopf algebras of dimension 8. Then, we will move our attention to some families of known Hopf algebras, such as Radford Hopf algebras \cite{Rad} and the Hopf algebras $E(n)$ \cite{BDG, CD}. 

\begin{remark}\label{semisimple8}
Semisimple Hopf algebras of dimension 8 over an algebraically closed field of characteristic 0 are classified in \cite{Ma}: they are, up to isomorphism, $\Bbbk G$, $(\Bbbk D_{8})^{*}$, $(\Bbbk Q_{8})^{*}$ and $H_{8}$, where $G$ is a finite group of order 8, $D_{8}$ is the dihedral group of order 8, $Q_{8}$ is the quaternion group, and $H_{8}$ is the unique noncommutative and noncocommutative semisimple Hopf algebra of dimension 8, known as the \emph{Kac--Paljutkin Hopf algebra} \cite{KP}. 
From what we recalled above, %$\Bbbk G$ admits only the trivial infinitesimal $\Rr$-matrix $\chi=0$, while 
$(\Bbbk D_{8})^{*}$ and  $(\Bbbk Q_{8})^{*}$ have no quasitriangular structures.
In Subsection \ref{subsect:generKP} we will obtain a classification of infinitesimal $\Rr$-matrices for $H_8$, and more generally for the family of generalized Kac--Paljutkin Hopf algebras $H_{2n^2}$. 
\end{remark}

%\begin{remark}
%\mg{Ho integrato questo remark (era subito dopo $H_8$ che ora è più avanti) nell'ultima riga del remark precedente. \sout{We recall that every other semisimple Hopf algebra $H$ of dimension $8$ is either a group algebra or the dual of a group algebra (see Remark \ref{semisimple8}), in the former case $H$ admits only the trivial infinitesimal braiding $\chi=0$, while in the latter $H$ has no quasi-triangular structure.}}
%\end{remark}

We now spend a few words on non-semisimple Hopf algebras of dimension 8. 

\subsection{Non-semisimple Hopf algebras of dimension $8$}
 According to \cite[Theorem 3.5]{St2} the class of non-semisimple Hopf algebras of dimension 8 (over an algebraically closed field of characteristic 0) consists of the following, up to isomorphism: 
\begin{itemize}
    \item A pointed Hopf algebra denoted by $A'_{C_{4}}$ which has no quasitriangular structure \cite[Proposition 2.6 (1)]{Wa2}.
    \item A pointed Hopf algebra denoted by $A''_{C_{4}}$ which again has no quasitriangular structure (ibid.).
    \item Its dual $(A''_{C_{4}})^{*}$, which has a unique quasitriangular structure \cite[Proposition 2.10]{Wa2}.
    \item A pointed Hopf algebra denoted by $A_{C_{2}\times C_{2}}$ which has two 1-parameter families of quasitriangular structures \cite[Proposition 2.7]{Wa2}.
    \item A pointed Hopf algebra denoted by $E(2)$, which is part of the family $E(n)$ and has a 1-parameter family of quasitriangular structures (see \cite{PVO}).
    \item The family of Radford pointed Hopf algebras $H_{(2,2,q)}$, that are also quasitriangular (see \cite{Rad}).
\end{itemize}

We start by classifying infinitesimal $\Rr$-matrices for $(A''_{C_{4}})^{*}$ and $A_{C_{2}\times C_{2}}$.  
We refer to \cite{Wa2} for the description of these algebras and their respective quasitriangular structures. The base field $\Bbbk$ will be assumed to be of characteristic different from 2. %[E' l'unica cosa che ci serve per definire $\Rr$] 
\medskip

\noindent\textbf{The Hopf algebra $(A''_{C_{4}})^{*}$}. Let $\omega\in\Bbbk$ be a primitive 4th root of unity. Then $(A''_{C_{4}})^{*}$ is the algebra generated by $g,x$ modulo the relations
\[
g^{4}=1,\qquad x^{2}=0,\qquad xg=\omega gx.
\]
The comultiplication and counit on $g$ and $x$ are given by
\[
\Delta(g)=g\otimes g-2gx\otimes g^{3}x,\quad \Delta(x)=1\otimes x+x\otimes g^{2},\quad \varepsilon(g)=1_{\Bbbk},\quad \varepsilon(x)=0,
\]
while the antipode is defined by $S(g)=g^{3}$ and $S(x)=-xg^{2}$. By \cite[Theorem 1.4]{Wa2} its unique quasitriangular structure has the form
\[
\Rr=\frac{1}{2}\big(1\otimes 1+g^{2}\otimes1+1\otimes g^{2}-g^{2}\otimes g^{2}\big)-x\otimes x-x\otimes g^{2}x+g^{2}x\otimes x-g^{2}x\otimes g^{2}x.
\]
Let us observe that a basis of $(A''_{C_{4}})^{*}$ is given by $\{1,g,g^{2},g^{3},x,xg,xg^{2},xg^{3}\}$. 

\begin{proposition}\label{prop:infinitesimalA''C4}
    The Hopf algebra $(A''_{C_{4}})^{*}$ admits no non-zero infinitesimal $\Rr$-matrices.
\end{proposition}

\begin{proof}
    Consider the group $G:=\{1,g,g^{2},g^{3}\}$ and the group algebra $A:=\Bbbk G$ (here considered just as an algebra). Then $(A''_{C_{4}})^{*}=A\oplus xA$, so that 
\[
\chi\in(A''_{C_{4}})^{*}\otimes(A''_{C_{4}})^{*}=(A\oplus xA)\otimes(A\oplus xA)=(A\otimes A)\oplus(A\otimes xA)\oplus(xA\otimes A)\oplus(xA\otimes xA).
\]
Thus, we can write
\[
\chi=a\otimes b+c\otimes xd+ xe\otimes f+xp\otimes xq,
\]
for $a,b,c,d,e,f,p,q\in A$. By imposing \eqref{cqtr1} on $g$ %: since $\Delta(g)=g\otimes g-2gx\otimes g^{3}x$ 
we obtain
\[
\begin{split}
    ag\otimes bg+&cg\otimes xdg+xeg\otimes fg+xpg\otimes xqg-2agx\otimes bg^{3}x\\&=ga\otimes gb+gc\otimes gxd+gxe\otimes gf+gxp\otimes gxq-2gxa\otimes g^{3}xb.
\end{split}
\]
By linear independence, considering the terms with $x$ only in the right tensorand, we must have $cg\otimes xdg=gc\otimes gxd$. But, by the defining relations, $cg\otimes xdg=%gc\otimes xdg=
gc\otimes xgd=\omega gc\otimes gxd$. Hence, since $\omega\not=1$, we must have $gc\otimes gxd=0$, i.e. $c\otimes xd=0$. Similarly, searching for terms with $x$ only in the left tensorand, we find that $xeg\otimes fg=gxe\otimes gf$. Once again $xeg\otimes fg=%xeg\otimes gf=
xge\otimes gf=\omega gxe\otimes gf$,
hence $xe\otimes f=0$. Then $\chi=a\otimes b+xp\otimes xq$. 

Now imposing \eqref{cqtr1} on $x$
%: since $\Delta(x)=1\otimes x+x\otimes g^{2}$ 
we obtain
\[
a\otimes bx+ax\otimes bg^{2}=a\otimes xb+xa\otimes g^{2}b.
\]
By linear independence, looking at the terms with no $x$ on the left tensorand, we obtain $a\otimes bx=a\otimes xb$. Hence $a=0$ or $bx=xb$. Suppose $a\not=0$. By writing an arbitrary $b=\eta_{0}1+\eta_{1}g+\eta_{2}g^{2}+\eta_{3}g^{3}\in A$ with $\eta_{0},\eta_{1},\eta_{2},\eta_{3}\in\Bbbk$, we have
\[
bx=\eta_{0}x+\eta_{1}gx+\eta_{2} g^{2}x+\eta_{3} g^{3}x
\]
while
\[
xb=\eta_{0} x+\eta_{1} xg+\eta_{2} xg^{2}+\eta_{3} xg^{3}=\eta_{0} x+\omega\eta_{1} gx-\eta_{2} g^{2}x-\omega\eta_{3} g^{3}x.
\]
By linear independence, since $\omega$ is different from 1, we obtain  $\eta_{1}=\eta_{2}=\eta_{3}=0$, i.e. $b=\eta_{0}1$. Therefore we can write $\chi=\eta_{0}a\otimes1+xp\otimes xq$. But then $0\overset{\eqref{eq:eps-chi1}}{=}(\mathrm{Id}\otimes\varepsilon)(\chi)=\eta_{0}a$, so $\eta_{0}=0$ and we find $\chi=xp\otimes xq$. 
%Finally, we show that the previous $\chi$ does not satisfy $\chi\Delta(g)=\Delta(g)\chi$ if it is not zero. 
Finally,
\[
\chi\Delta(g)=xpg\otimes xqg=xgp\otimes xgq=\omega^{2}gxp\otimes gxq=-\Delta(g)\chi,
\]
so that, by \eqref{cqtr1} on $g$, we get $\chi\Delta(g)=0$ and then $\chi=xp\otimes xq=0$.
\end{proof}

\medskip

\noindent\textbf{The Hopf algebra $A_{C_{2}\times C_{2}}$}. As an algebra $A_{C_{2}\times C_{2}}$ is generated by $g,h,x$ modulo the relations 
\[
g^{2}=h^{2}=1,\quad x^{2}=0,\quad gh=hg,\quad gx=-xg,\quad hx=-xh.
\]
The comultiplication and counit on $g,h,x$ are given by
\[
\Delta(g)=g\otimes g,\quad \Delta(h)=h\otimes h,\quad \Delta(x)=1\otimes x+x\otimes g,\quad \varepsilon(g)=\varepsilon(h)=1_{\Bbbk},\quad \varepsilon(x)=0,
\]
while the antipode is defined by $S(g)=g$, $S(h)=h$ and $S(x)=-xg$. A basis is given by $\{ 1,g, h, gh, x, gx, hx, ghx\}$. By \cite[Theorem 1.4]{Wa2}, $A_{C_{2}\times C_{2}}$ has two 1-parameter families of universal $\Rr$-matrices of the form
\begin{equation}\label{eq:Ra}
\Rr_a:=\Rr_q(1\otimes 1+ ax\otimes gx)=\Rr_q+a\Rr_q(x\otimes gx), \quad \text{with}\, \, a\in\Bbbk, \, q=0, 1,
\end{equation}
where $\Rr_q=\frac{1}{4}\sum_{i,j,k,l=0,1} (-1)^{-ij-kl}g^ih^k\otimes g^{j+q(j+l)}h^{q(j+l)}.$
Explicitly, they are given by
\begin{equation}\label{eq:R0}
\Rr_\lambda=\Rr_0=\frac{1}{2}\big(1\otimes1+1\otimes g+g\otimes 1-g\otimes g)+\frac{\lambda}{2}\big(x\otimes gx+x\otimes x+gx\otimes gx-gx\otimes x)
\end{equation}
for $\lambda\in\Bbbk$, and
\begin{equation}\label{eq:R1}
\begin{split}
	\Rr_\nu&= \Rr_1+ \nu\Rr_1x\otimes gx=\frac{1}{4}\big(1\otimes 1+1\otimes gh+h\otimes 1-h\otimes gh+1\otimes h+1\otimes g+h\otimes h\\&-h\otimes g+g\otimes 1+g\otimes gh+gh\otimes 1-gh\otimes gh-g\otimes h-g\otimes g-gh\otimes h+gh\otimes g\big)\\&+\frac{\nu}{4}\big( x\otimes gx+x\otimes hx+hx\otimes gx -hx\otimes hx+x\otimes ghx+x\otimes x+hx\otimes ghx-hx\otimes x+gx\otimes gx\\&+gx\otimes hx+ghx\otimes gx-ghx\otimes hx-gx\otimes ghx-gx\otimes x-ghx\otimes ghx+ghx\otimes x \big)
\end{split}
\end{equation}
for $\nu\in \Bbbk$. As observed in \cite[Table III]{Wa2}, \eqref{eq:Ra}  are indeed triangular structures for $A_{C_{2}\times C_{2}}$.
 
%It has two 1-parameter families of triangular structure, namely
%\[
%\Rr=\frac{1}{2}\big(1\otimes1+1\otimes g+g\otimes 1-g\otimes g)+\frac{\lambda}{2}\big(x\otimes gx+x\otimes x+gx\otimes gx-gx\otimes x)
%\]
%and
%\[
%\Rr'=(\text{mettere $q=1$ in $R_{q}$})
%\]
We now classify the infinitesimal $\Rr$-matrices for $A_{C_{2}\times C_{2}}$.
\begin{proposition}\label{prop:chiAC2xC2}
	The Hopf algebra $A_{C_{2}\times C_{2}}$ has an exhaustive 1-parameter family of infinitesimal $\Rr$-matrices $\chi_\alpha$ given by $$\chi_\alpha=\alpha x\otimes xg, \quad \text{with}\,\, \alpha\in\Bbbk.$$
\end{proposition}
\begin{proof}
	Denote $H:=A_{C_{2}\times C_{2}}$. If $(H, \Rr_{a})$ is pre-Cartier, then there exists $\chi=\chi^i\otimes\chi_i\in H\otimes H$ fulfilling \eqref{cqtr1}, \eqref{cqtr2}, \eqref{cqtr3}. %We look at \eqref{qtr1} on generators. 
%\begin{invisible}
%From \eqref{cqtr1} on $1$ we do not obtain information. 
%Metterei in \% se sei d'accordo.
%\end{invisible}	
Let $H_{0}:=\Bbbk\langle g, h \rangle$ be the coradical of $H$. Then, $H=H_{0}\oplus xH_{0}$. Equation \eqref{cqtr1} on $g$
%,i.e. $\chi^ig\otimes \chi_i g=g\chi^{i}\otimes g\chi_{i}$, 
is equivalent to $g\chi^i g^{-1}\otimes g\chi_ig^{-1}=\chi^{i}\otimes\chi_{i}$. 
The conjugation by $g$ is trivial on $H_{0}$ and it is the multiplication by $-1$ on $xH_{0}$. 
%\begin{invisible}
%Indeed, for an arbitrary element $\xi=\alpha x+\beta xg+\gamma xh+\delta xgh\in xH_0$, with $\alpha, \beta, \gamma, \delta\in \Bbbk$, one has $g\xi g^{-1} = \alpha gxg+\beta gxgg+\gamma gxhg+\delta gxghg=-(\alpha x+\beta xg+\gamma xh+\delta xgh)=-\xi$. 
%Ok, metterei però in \% se sei d'accordo.
%\end{invisible}
Since $\mathrm{char}(\Bbbk)\neq2$, we have $\chi\in H_{0}\otimes H_{0}+xH_{0}\otimes xH_{0}$.
	%Similarly, equation \eqref{cqtr1} on $h$ 
 %is equivalent to $h\chi^i h^{-1}\otimes h\chi_ih^{-1}=\chi^{i}\otimes\chi_{i}$. The conjugation by $h$ is trivial on $H_{0}$ and it is the multiplication by $-1$ on $xH_{0}$ and we do not get further information. 
 We further observe that the Hopf algebra $H$ has a Hopf algebra projection $\pi:H\to H_0$, given on the basis by $\pi(x^mg^ah^b)=\delta_{m,0}g^a h^b$. %In fact, denoting $\sigma:H_0\to H$ the canonical injection, one has $\pi\circ\sigma = \mathrm{Id}_{H_0}$.
%	\begin{invisible}
\begin{invisible}
We show that $\pi$ is a Hopf algebra map. 
Indeed, $\pi(1)=1$; given arbitrary elements $x^mg^ah^b$ and $x^{m'}g^{a'}h^{b'}$, if $m\neq 0$ or $m'\neq 0$, one has $\pi(x^mg^a h^b x^{m'}g^{a'}h^{b'})=0=\pi(x^mg^ah^b)\pi( x^{m'}g^{a'}h^{b'})$; if $m=0=m'$, then $\pi(g^ah^bg^{a'}h^{b'})=\pi(g^ag^{a'}h^{b}h^{b'})=\pi(g^{a+a'}h^{b+b'})=g^{a+a'}h^{b+b'}=g^ag^{a'}h^bh^{b'}=g^ah^bg^{a'}h^{b'}=\pi(g^ah^b)\pi(g^{a'}h^{b'})$. Moreover, $\varepsilon\pi(x^mg^a h^b)=\varepsilon(\delta_{m,0}g^a h^b)=\delta_{m,0}\varepsilon(g^a h^b)=\varepsilon (x^m)\varepsilon (g^a h^b)=\varepsilon (x^{m}g^a h^b)$; if $m\neq 0$ (we can suppose $m=1$), then $(\pi\otimes \pi)\Delta (xg^a h^b)=(\pi\otimes \pi)\Delta (x)\Delta(g^a h^b)=(\pi\otimes \pi)(g^ah^b\otimes xg^a h^b+xg^a h^b\otimes g^{a+1}h^b)=0=\Delta \pi(xg^ah^b)$, while if $m=0$, then $(\pi\otimes\pi)\Delta (g^a h^b)=(\pi\otimes \pi)(g^ah^b\otimes g^ah^b)=g^a h^b\otimes g^a h^b=\Delta(g^a h^b)=\Delta\pi(g^ah^b)$.
%	\end{invisible}
\end{invisible}
By \cite[Proposition 2.24]{ABSW} (cf. also \cite[Remark 2.25]{ABSW}), since $H_0$ is a commutative Hopf algebra, we have that $(\pi\otimes\pi)(\chi)\in P(H_0)\otimes P(H_0)=0$. Then, $\chi\in xH_0\otimes xH_0$. 
%\begin{invisible}
%If we write $\chi\in H_0\otimes H_0 + xH_0\otimes xH_0$ as
%$\chi=a_01\otimes 1+a_11\otimes g+a_21\otimes h+a_31\otimes gh+a_4g\otimes 1+a_5g\otimes g+a_6g\otimes h+a_7g\otimes gh+a_8 h\otimes 1+a_9 h\otimes g+a_{10}h\otimes h + a_{11} h\otimes gh +a_{12}gh\otimes 1+a_{13} gh\otimes g+a_{14} gh\otimes h+a_{15}gh\otimes gh+\chi'$, where $a_0,\ldots, a_{15}\in\Bbbk$ and $\chi'\in xH_0\otimes xH_0$, then from $0=(\pi\otimes\pi)(\chi)$ we get $a_0=0,\ldots, a_{15}=0$, hence $\chi\in xH_0\otimes xH_0$. 
%
%Lo vedi anche senza esplicitare la base:
%$\chi=a\otimes b+xc\otimes xd$ with $a,b,c,d\in H_{0}$. Then $0=(\pi\otimes\pi)(\chi)=(\pi\otimes\pi)(a\otimes b)+(\pi\otimes\pi)(xc\otimes xd)=a\otimes b$, so $\chi=xc\otimes xd\in xH_{0}\otimes xH_{0}$. Comunque ok e metterei tutto l'invisible in \%.
%\end{invisible}
By \cite[Theorem 2.21]{ABSW} any infinitesimal $\Rr$-matrix $\chi$ is a $2$-cocycle, i.e. $\chi$ fulfills \eqref{eq:Hoch2cocy}. 
\begin{invisible}
Writing $\chi=xl\otimes xl'$, for $l, l'\in H_0$ given on the basis by $l=\alpha_01+\alpha_1g+\alpha_2 h+\alpha_3 gh,l'=\alpha'_01+\alpha'_1g+\alpha'_2 h+\alpha'_3 gh$, then $\chi_{12}+(\Delta\otimes\mathrm{Id})(\chi)=\chi_{23}+(\mathrm{Id}\otimes\Delta)(\chi)$ becomes
\[
\begin{split}
	(\alpha_0x&+\alpha_1xg+\alpha_2 xh+\alpha_3 xgh)\otimes (\alpha'_0x+\alpha'_1xg+\alpha'_2 xh+\alpha'_3 xgh)\otimes 1\\&+(\alpha_01\otimes x+\alpha_1g\otimes xg+\alpha_2 h\otimes xh+\alpha_3 gh\otimes xgh+\alpha_0x\otimes g+\alpha_1xg\otimes 1+\alpha_2 xh\otimes gh\\&+\alpha_3 xgh\otimes h)\otimes (\alpha'_0x+\alpha'_1xg+\alpha'_2 xh+\alpha'_3 xgh)=1\otimes (\alpha_0x+\alpha_1xg+\alpha_2 xh+\alpha_3 xgh)\otimes (\alpha'_0x\\&+\alpha'_1xg+\alpha'_2 xh+\alpha'_3 xgh)+(\alpha_0x+\alpha_1xg+\alpha_2 xh+\alpha_3 xgh)\otimes (\alpha'_01\otimes x+\alpha'_1g\otimes xg\\&+\alpha'_2h\otimes xh+\alpha'_3 gh\otimes xgh
	+\alpha'_0x\otimes g+\alpha'_1xg\otimes 1+\alpha'_2 xh\otimes gh+\alpha'_3 xgh\otimes h).
\end{split}	
\]
Looking at terms with $1$ on the first entry, by linear independence we get $\alpha_1\alpha'_0=\alpha_1\alpha'_1=\alpha_1\alpha'_2=\alpha_1\alpha'_3=0$, $\alpha_2\alpha'_0=\alpha_2\alpha'_1=\alpha_2\alpha'_2=\alpha_2\alpha'_3=0$, $\alpha_3\alpha'_0=\alpha_3\alpha'_1=\alpha_3\alpha'_2=\alpha_3\alpha'_3=0$, from which it follows that $\alpha_1=0$, $\alpha_2=0$, $\alpha_3=0$, respectively. Moreover, looking at terms with $1$ on the third entry, by linear independence we obtain $\alpha_0\alpha'_0=\alpha_1\alpha'_0=\alpha_2\alpha'_0=\alpha_3\alpha'_0=0$, $\alpha_0\alpha'_2=\alpha_1\alpha'_2=\alpha_2\alpha'_2=\alpha_3\alpha'_2=0$,  $\alpha_0\alpha'_3=\alpha_1\alpha'_3=\alpha_2\alpha'_3=\alpha_3\alpha'_3=0$, from which it follows that $\alpha'_0=0$, $\alpha'_2=0$, $\alpha'_3=0$, respectively. 
\end{invisible}
By writing $\chi=xl\otimes xl'$, for $l,l'\in H_{0}$, \eqref{eq:Hoch2cocy} becomes
\begin{equation}\label{eq:xl}
xl\otimes xl'\otimes1+(1\otimes x)\Delta(l)\otimes xl'+(x\otimes g)\Delta(l)\otimes xl'=1\otimes xl\otimes xl'+xl\otimes(1\otimes x)\Delta(l')+xl\otimes(x\otimes g)\Delta(l').
\end{equation}
Then, looking at those terms with no $x$ in the first tensorand, %\gr{(a priori $\Delta(l)$ potrebbe non avere $1$ in prima posizione per nessuno dei suoi termini)}, 
by linear independence we get $(1\otimes x)\Delta(l)\otimes xl'=1\otimes xl\otimes xl'$. Hence, writing $l=\beta 1+\gamma g+\delta h+\eta gh$ in terms of the basis of $H_{0}$, for $\beta,\gamma,\delta,\eta\in\Bbbk$, we get 
\[
\begin{split}
&\beta 1\otimes x\otimes xl'+\gamma g\otimes xg\otimes xl'+\delta  h\otimes xh\otimes xl'+\eta gh\otimes xgh\otimes xl'=1\otimes x(\beta 1+\gamma g+\delta h+\eta gh)\otimes xl'.
\end{split}
\]
Hence, by linear independence, $\gamma=\delta=\eta=0$ and then $l=\beta 1$. Analogously, looking at those terms in \eqref{eq:xl}  with no $x$ in the third tensorand, we must have $x\beta\otimes xl'\otimes 1=x\beta\otimes (x\otimes g)\Delta(l')$. Since $\beta \in \Bbbk$, we can rewrite this as $x\otimes xl'\otimes 1=x\otimes (x\otimes g)\Delta(l')$ by renaming $l'$ the product $\beta l'$. Writing $l'$ in terms of the basis of $H_0$ as done for $l$ before, one gets, by linear independence, that $l'=\alpha g$, for $\alpha\in\Bbbk$.

Thus, we have $\chi=\alpha x\otimes xg$. As in \cite[Remark 2.25]{ABSW}, we did not use the specific form of the universal $\Rr$-matrix. It remains to check that $\chi$ is an infinitesimal $\Rr$-matrix. Axioms \eqref{cqtr1} and \eqref{eq:Hoch2cocy} are trivially satisfied. Then, by Proposition \ref{prop:newcharacterization}, it remains to verify \eqref{cqtr2},
%Clearly, we have $\chi\Delta(1)=\Delta(1)\chi$, $\chi\Delta (g)=\alpha xg\otimes x=\alpha gx\otimes gxg =\Delta(g)\chi$, $\chi\Delta(h)=\alpha xh\otimes xgh=\alpha xh\otimes xhg=\alpha hx\otimes hxg=\Delta(h)\chi$, 
%\begin{invisible}
%$\chi\Delta(gh)=\alpha xgh\otimes xh=\alpha ghx\otimes ghxg=\Delta(gh)\chi$,
%Sempre vera perchè $\Delta(gh)=\Delta(g)\Delta(h)$, metterei in \%,
%\end{invisible}
%$\chi\Delta(x)=0=\Delta(x)\chi$, $\chi\Delta(gx)=0=\Delta(gx)\chi$, $\chi\Delta(hx)=0=\Delta(hx)\chi$, $\chi\Delta(ghx)=0=\Delta(ghx)\chi$.  %with universal $\Rr$-matrices $\Rr_\lambda$, $\Rr_\nu$ as in \eqref{eq:R0}, \eqref{eq:R1}.	
which is equivalent to $\Rr_{12}(\mathrm{Id}\otimes\Delta)(\chi)=\Rr_{12}\chi_{12}+\chi_{13}\Rr_{12}$.  %$\Rr_{a}=\Rr_q(1\otimes 1+ ax\otimes gx)=\Rr_q+a\Rr_q(x\otimes gx)$, $a\in \Bbbk$, $q=0,1$, 
We have 
\[(\Rr_{a})_{12}\chi_{12}=(\Rr_{q})_{12}\chi_{12},\quad \chi_{13}(\Rr_{a})_{12}=\chi_{13}(\Rr_{q})_{12},\quad (\Rr_{a})_{12}(\mathrm{Id}\otimes \Delta)(\chi)=(\Rr_{q})_{12}(\mathrm{Id}\otimes \Delta)(\chi),\] 
%\[(\Rr_{a})_{23}\chi_{23}=(\Rr_{q})_{23}\chi_{23},\quad \chi_{13}(\Rr_{a})_{23}=\chi_{13}(\Rr_{q})_{23},\quad (\Rr_{a})_{23}(\Delta\otimes \mathrm{Id})(\chi)=(\Rr_{q})_{23}(\Delta\otimes \mathrm{Id})(\chi),\] 
so that, since  $\Rr_q^{-1}=\Rr_q^{\mathrm{op}}=\Rr_q$, for $q=0,1$, the axiom \eqref{cqtr2} for the triangular structures \eqref{eq:Ra} is equivalent to $(\mathrm{Id}\otimes \Delta)(\chi)=\chi_{12}+(\Rr_{q})_{12}\chi_{13}(\Rr_{q})_{12}$. 

For the case $q=0$, the universal $\Rr$-matrix is given as in \eqref{eq:R0} and one can easily show that these equalities are satisfied. 
\begin{invisible}
We have $(\mathrm{Id}\otimes\Delta)(\chi)=\alpha(x\otimes g\otimes xg+x\otimes xg\otimes 1)$ and $\chi_{12}+\Rr_{0_{12}}\chi_{13}\Rr_{0_{12}} = \alpha x\otimes xg \otimes 1+\frac{1}{4}\alpha (x\otimes 1\otimes xg +x\otimes g\otimes xg+ xg\otimes 1\otimes xg-xg\otimes g\otimes xg+x\otimes g\otimes xg+x\otimes 1\otimes xg+xg\otimes g\otimes xg-xg\otimes 1\otimes xg-xg\otimes 1\otimes xg-xg\otimes g\otimes xg-x\otimes 1\otimes xg+x\otimes g\otimes xg+xg\otimes g\otimes xg+xg\otimes 1\otimes xg+x\otimes g\otimes xg-x\otimes 1\otimes xg)=\alpha(x\otimes g\otimes xg+x\otimes xg\otimes 1)$, so \eqref{cqtr2} holds true. Moreover, $(\Delta\otimes\mathrm{Id})(\chi)=\alpha(1\otimes x\otimes xg+x\otimes g\otimes xg)$ and $\chi_{23}+\Rr_{0_{23}}\chi_{13}\Rr_{0_{23}}=\alpha 1\otimes x\otimes xg+\frac{1}{4}\alpha (x\otimes 1\otimes xg+x\otimes 1\otimes x+x\otimes g\otimes xg-x\otimes g\otimes x-x\otimes 1\otimes x-x\otimes 1\otimes xg-x\otimes g\otimes x+x\otimes g\otimes xg+x\otimes g\otimes xg+x\otimes g\otimes x+x\otimes 1\otimes xg-x\otimes 1\otimes x+x\otimes g\otimes x+x\otimes g\otimes xg+x\otimes 1\otimes x-x\otimes 1\otimes xg=\alpha(1\otimes x\otimes xg+x\otimes g\otimes xg)$, so \eqref{cqtr3} holds true.
\end{invisible}
One can also observe that 
\[
\begin{split}
(\Rr_{0})_{12}\chi_{13}(\Rr_{0})_{12}&=\alpha\Rr_{0}(x\otimes1)\Rr_{0}\otimes xg\\&=\frac{\alpha}{4}(1\otimes1+1\otimes g+g\otimes1-g\otimes g)(x\otimes1)(1\otimes1+1\otimes g+g\otimes1-g\otimes g)\otimes xg\\&=\frac{\alpha}{4}(x\otimes1)(1\otimes1+1\otimes g-g\otimes1+g\otimes g)(1\otimes1+1\otimes g+g\otimes1-g\otimes g)\otimes xg\\&=\alpha(x\otimes1)(1\otimes g)\otimes xg=\alpha x\otimes g\otimes xg,
\end{split}
\]
hence \eqref{cqtr2} is satisfied.
%\rd{Let $q=0$, so $\Rr_{0}=\frac{1}{2}(1\otimes 1+1\otimes g+g\otimes1-g\otimes g)$. Then, $(\Rr_{0})_{12}\chi_{13}(\Rr_{0})_{12}$ is given by
%\[
%\begin{split}
%\alpha\Rr_{0}(x\otimes1)\Rr_{0}\otimes xg&=\frac{\alpha}{4}(1\otimes1+1\otimes g+g\otimes1-g\otimes g)(x\otimes1)(1\otimes1+1\otimes g+g\otimes1-g\otimes g)\otimes xg\\&=\frac{\alpha}{4}(x\otimes1)(1\otimes1+1\otimes g-g\otimes1+g\otimes g)(1\otimes1+1\otimes g+g\otimes1-g\otimes g)\otimes xg\\&=\alpha(x\otimes1)(1\otimes g)\otimes xg=\alpha x\otimes g\otimes xg
%\end{split}
%\]
%hence \eqref{cqtr2} is satisfied. Similarly, $(\Rr_{0})_{23}\chi_{13}(\Rr_{0})_{23}=\alpha x\otimes g\otimes xg$, hence \eqref{cqtr3} is satisfied.
%}
For the case $q=1$, the universal $\Rr$-matrix is given as in \eqref{eq:R1}, and it is not hard to check that $(\Rr_{1})_{12}(\mathrm{Id}\otimes\Delta)(\chi)=(\Rr_{1})_{12}\chi_{12}+\chi_{13}(\Rr_{1})_{12}$.
\begin{invisible}
Indeed, we have \[
\begin{split}
\Rr_{1_{12}}&(\mathrm{Id}\otimes\Delta)(\chi)=\frac{\alpha}{4}(x\otimes g\otimes xg+x\otimes xg\otimes 1+x\otimes h\otimes xg-x\otimes hx\otimes 1+hx\otimes g\otimes xg\\&+hx\otimes xg\otimes 1-hx\otimes h\otimes xg+hx\otimes hx \otimes 1+x\otimes hg\otimes xg-x\otimes hgx\otimes 1+x\otimes 1 \otimes xg\\&-x\otimes x\otimes 1+hx\otimes hg\otimes xg-hx\otimes hgx\otimes 1- hx\otimes 1\otimes xg+ hx\otimes x\otimes 1+gx\otimes g\otimes xg\\&+gx\otimes xg\otimes 1+gx\otimes h\otimes xg-gx\otimes hx\otimes 1+ghx\otimes g \otimes xg+ghx\otimes xg \otimes 1-ghx\otimes h\otimes xg\\&+ ghx\otimes hx\otimes 1-gx\otimes hg\otimes xg+gx\otimes hgx\otimes 1-gx\otimes 1\otimes xg+gx\otimes x\otimes 1-ghx\otimes hg\otimes xg\\&+ghx\otimes hgx\otimes 1+ghx\otimes 1\otimes xg- ghx\otimes x\otimes 1)
\end{split}
\]
which is equal to \[
\begin{split}
\Rr_{1_{12}}&\chi_{12}+\chi_{13}\Rr_{1_{12}}=\frac{\alpha}{4}(x\otimes xg\otimes 1-x\otimes hx\otimes 1+hx\otimes xg\otimes 1+hx\otimes hx\otimes 1-x\otimes hgx\otimes 1\\&-x\otimes x\otimes 1-hx\otimes hgx\otimes 1+hx\otimes x\otimes 1+gx\otimes xg\otimes 1-gx\otimes hx\otimes 1+ghx\otimes xg\otimes 1\\&+ghx\otimes hx\otimes 1+gx\otimes hgx\otimes 1+gx\otimes x\otimes 1+ ghx\otimes hgx\otimes 1-ghx\otimes x\otimes 1+x\otimes 1\otimes xg\\&+x\otimes gh\otimes xg-hx\otimes 1\otimes xg+hx\otimes gh\otimes xg+x\otimes h\otimes xg+x\otimes g\otimes xg-hx\otimes h\otimes xg\\&+hx\otimes g\otimes xg+xg\otimes 1\otimes xg+xg\otimes gh\otimes xg+ghx\otimes 1\otimes xg-ghx\otimes gh\otimes xg\\&+gx\otimes g\otimes xg-ghx\otimes h\otimes xg+ghx\otimes g\otimes xg).
\end{split}
\]
We now show that $\Rr_{1_{23}}(\Delta\otimes \mathrm{Id})(\chi)=\Rr_{1_{23}}\chi_{23}+\chi_{13}\Rr_{1_{23}}$. We have 
\[
\begin{split}
	\Rr_{1_{23}}&(\Delta\otimes \mathrm{Id})(\chi)=\frac{\alpha}{4}(1\otimes x\otimes xg+x\otimes g\otimes xg-1\otimes x\otimes hx-x\otimes g\otimes hx+1\otimes hx\otimes xg\\&+x\otimes hg\otimes xg+1\otimes hx\otimes hx+x\otimes hg\otimes hx-1\otimes x\otimes hgx-x\otimes g\otimes hgx-1\otimes x\otimes x\\&-x\otimes g\otimes x-1\otimes hx\otimes hgx -x\otimes hg\otimes hgx+1\otimes hx \otimes x+x\otimes hg\otimes x-1\otimes xg\otimes xg\\&+x\otimes 1\otimes xg-1\otimes gx\otimes hx-x\otimes 1\otimes hx+1\otimes ghx\otimes xg+x\otimes h\otimes xg+1\otimes ghx\otimes hx\\&+x\otimes h\otimes hx+1\otimes gx\otimes hgx+x\otimes 1\otimes hgx+1\otimes gx\otimes x+x\otimes 1\otimes x+1\otimes ghx\otimes hgx\\&+x\otimes h\otimes hgx-1\otimes ghx\otimes x-x\otimes h\otimes x)
\end{split}
\]
which is equal to 
\[
\begin{split}
	\Rr_{1_{23}}&\chi_{23}+\chi_{13}\Rr_{1_{23}}=\frac{\alpha}{4}(1\otimes x\otimes xg-1\otimes x\otimes hx	+1\otimes hx\otimes xg+1\otimes hx\otimes hx-1\otimes x\otimes hgx\\&-1\otimes x\otimes x-1\otimes hx\otimes hgx+1\otimes hx\otimes x+1\otimes gx\otimes xg-1\otimes gx\otimes hx+1\otimes ghx\otimes xg\\&+1\otimes ghx\otimes hx+1\otimes gx\otimes hgx+1\otimes gx\otimes x+1\otimes ghx\otimes hgx-1\otimes ghx\otimes x+x\otimes 1\otimes xg\\&-x\otimes 1\otimes hx+x\otimes h\otimes xg+x\otimes h\otimes hx+x\otimes 1\otimes hgx+x\otimes 1\otimes x+x\otimes h\otimes hgx\\&-x\otimes h\otimes x+x\otimes g\otimes xg-x\otimes g\otimes hx+x\otimes gh\otimes xg+x\otimes gh\otimes hx-x\otimes g\otimes hgx\\&-x\otimes g\otimes x-x\otimes gh\otimes hgx+x\otimes gh\otimes x).
\end{split}
\]
\end{invisible}
Thus, $A_{C_{2}\times C_{2}}$ has a 1-parameter family of infinitesimal $\Rr$-matrices $\chi_\alpha=\alpha x\otimes xg$, $\alpha\in\Bbbk$.
\end{proof}

\begin{remark}\label{rmk:chiAC2xC2-Cart}
	Notice that 
\begin{invisible}
	Indeed, for $q=0$ one has $\Rr_a=\Rr_\lambda$ as in \eqref{eq:R0} and
		\[
	\begin{split}
		\Rr_\lambda&\chi_\alpha=\Rr_0\chi_\alpha=\frac{1}{2}\alpha(x\otimes xg-x\otimes x+gx\otimes xg+gx\otimes x)=-\chi_\alpha^{\mathrm{op}}\Rr_0=-\chi_\alpha^{\mathrm{op}}\Rr_\lambda;		
	\end{split}
	\]
	for $q=1$ one has $\Rr_a=\Rr_\nu$ as in \eqref{eq:R1} and
	\[
	\begin{split}
		\Rr_\nu&\chi_\alpha=\Rr_1\chi_\alpha=\frac{1}{4}\alpha(x\otimes xg-x\otimes hx+hx\otimes xg+hx\otimes hx-x\otimes hgx-x\otimes x-hx\otimes hgx\\&+hx\otimes x+gx\otimes xg-gx\otimes hx+ghx\otimes xg+ghx\otimes hx+gx\otimes hgx+gx\otimes x+ghx\otimes hgx\\&-ghx\otimes x)=-\chi_\alpha^{\mathrm{op}}\Rr_1=-\chi_\alpha^{\mathrm{op}}\Rr_\nu.		
	\end{split}
\]
%for $\alpha,\lambda,\nu\in \Bbbk$. 
\end{invisible}
$\tau(S \otimes S)(\chi_{\alpha})= -\alpha x \otimes xg=-\chi_{\alpha}$,
thus, by Proposition \ref{cor:SSRchi},
 $(H,\Rr_a,\chi_{\alpha})$ is Cartier if and only if $\chi_{\alpha}=0$.
Thus, as for Sweedler's Hopf algebra (cf. \cite[Remark 2.8]{ABSW}), $A_{C_{2}\times C_{2}}$ is pre-Cartier through a 1-parameter family of infinitesimal $\Rr$-matrices $\chi_\alpha$, while it is Cartier only  with $\chi=0$.
\end{remark}	

Let us observe that $A_{C_{2}\times C_{2}}$ has coradical given by $\Bbbk(C_{2}\times C_{2})$.  As stated in \cite[p. 559]{CDR}, % see also \cite[p. 768]{BDG}
$A_{C_{2}\times C_{2}}$ is  isomorphic to $H\otimes\Bbbk C_{2}$, where $H$ is Sweedler's Hopf algebra, see also \cite[Proposition 4]{BDG00}. One can ``enlarge'' $A_{C_{2}\times C_{2}}$ by adding another group-like element, hence considering a Hopf algebra of dimension 16 with coradical given by $\Bbbk(C_{2}\times C_{2}\times C_{2})$. Let us observe that the pointed Hopf algebras of dimension 16 over an algebraically closed field of characteristic 0 are classified in \cite{CDR} and the (unique) Hopf algebra with coradical $\Bbbk(C_{2}\times C_{2}\times C_{2})$ is isomorphic to $H\otimes\Bbbk(C_{2}\times C_{2})$, see \cite[Theorem 6.1]{CDR}. %\mg{[In\cite[p. 768]{BDG} scrivono $\Bbbk(C_2^{n-1})$]} 
In general one can build a Hopf algebra which has coradical $\Bbbk C_{2}^n$, dimension $2^{n+1}$ and is isomorphic to $H\otimes\Bbbk C_{2}^{n-1}$ with $2\leq n\in\mathbb{N}$, see \cite{BDG00}. We denote it by $A_{C_{2}^n}$. 
On the other hand $A_{C_{2}^n}$ can also be regarded as $A_{C_{2}\times C_{2}} \otimes \Bbbk C_2^{n-2}$ via the Hopf algebra isomorphism $f:A_{C_{2}\times C_{2}} \otimes \Bbbk C_2^{n-2} \rightarrow A_{C_{2}^n}$ given by
\begin{eqnarray*}
    f(a \otimes 1)&=&a \quad \textrm{for every } a \in A_{C_{2}\times C_{2}},\\
    f(1 \otimes g_i)&=&gg_i \quad \textrm{for every } g_i, \ i=1, \ldots, n-2,
\end{eqnarray*}
where $g$ is one of the grouplike generators of $A_{C_{2}\times C_{2}}$ and $g_{i}$ are the grouplike elements of order 2 generating each factor $C_{2}$ in $\Bbbk C_{2}^{n-2}$.

\begin{remark}
The (quasi)triangular structures $\Rr_a=\Rr^{i}\otimes\Rr_{i}$ of the algebra $A_{C_{2}\times C_{2}}$, are given by \eqref{eq:Ra}. Moreover, $\Bbbk  C_2^{n-2}$ is triangular with $\Rr$-matrix $1\otimes 1$ since it is cocommutative, hence $\Rr^{i}\otimes1\otimes\Rr_{i}\otimes 1$ is a (quasi)triangular structure for $A_{C_{2}\times C_{2}} \otimes \Bbbk C_2^{n-2}$, see e.g. \cite[Theorem 2.2]{Ch98}. From this we can immediately deduce that $A_{C_{2}^n}$ has triangular structures $\Rr_{a}$ given as in \eqref{eq:Ra}, since $f(\Rr^{i}\otimes1)\otimes f(\Rr_{i}\otimes 1)=\Rr^i \otimes \Rr_i=\Rr_a$. %once observed that $H\bowtie^{1\otimes 1} H'=H\otimes H'$ and $[\Rr,\Rr']=\tilde{\Rr}$, see also \cite[Proposition 2.15]{ABSW}};
\end{remark}

\begin{remark}
We observe that there can be other quasitriangular structures on $A_{C_{2}^n}$ different from \eqref{eq:Ra} and therefore that the family of universal $\Rr$-matrices obtained in the last remark might not be exhaustive. Indeed if one, for example, repeats the same reasoning considering the isomorphism $A_{C_{2}\times C_{2}} \cong H \otimes \Bbbk C_{2}$, it becomes clear how to recover \eqref{eq:R0} from the quasitriangular structures of Sweedler's Hopf algebra $H$, although we cannot achieve a similar result for \eqref{eq:R1}.
\end{remark}

Nontheless we obtain a partial result on the infinitesimal $\Rr$-matrices of $A_{C_{2}^n}$.

\begin{proposition}\label{AC2-n}
    The triangular Hopf algebra $(A_{C_{2}^n}, \Rr_a)$ has a 1-parameter family of infinitesimal $\Rr$-matrices given by $\alpha x\otimes xg$, for $\alpha\in\Bbbk$.
\end{proposition}

\begin{proof}
   The Hopf algebra $A_{C_{2}\times C_{2}}$ has infinitesimal $\Rr$-matrices given by $\alpha x\otimes xg$, with $\alpha\in\Bbbk$, while the group Hopf algebra $\Bbbk C_2^{n-2}$ has only the infinitesimal $\Rr$-matrix $\chi=0$. Hence, by \cite[Proposition 2.15]{ABSW} we obtain that $\alpha x\otimes 1\otimes xg\otimes 1$ is an infinitesimal $\Rr$-matrix for $A_{C_{2}\times C_{2}}\otimes\Bbbk C_2^{n-2}$ (see also \cite[Example 2.16]{ABSW}). Thus, using the isomorphism $f$ above, we have that $f(\alpha x\otimes 1)\otimes f(xg\otimes 1)=\alpha x\otimes xg$ is an infinitesimal $\Rr$-matrix for $A_{C_{2}^n}$.
\end{proof}

%\mg{As already mentioned concerning Hopf algebras of dimension $8$, the unique noncommutative and noncocommutative semisimple Hopf algebra of dimension $8$ is the  Kac--Paljutkin Hopf algebra $H_8$, which can be seen as a particular case of the next family of Hopf algebras we consider.}

\subsection{The generalized Kac--Paljutkin Hopf algebras $H_{2n^2}$}\label{subsect:generKP}
%Here we consider the well-known Kac--Paljutkin Hopf algebra $H_8$, which is not pointed. %, and neither commutative nor cocommutative. 
%\gr{\sout{Let us recall the definition.}} 
We recall from \cite{Masuoka, P19} a class of semisimple Hopf
algebras $H_{2n^2}$ of dimension $2n^2$ (they also already appeared in \cite{Kac}, under the name ``ring groups''). These can be viewed as a generalization of the $8$-dimensional \emph{Kac--Paljutkin Hopf algebra} $H_8$ defined in \cite{KP}. They are not pointed, neither commutative, nor cocommutative.

Let $n > 1$ and $q$ be a primitive $n$-th root of unity of the field $\Bbbk$, whose characteristic does not divide $2n^2$. The Hopf algebra $H_{2n^2}$ is the algebra generated by $x$, $y$, and $z$, subject to the following relations: 
\[
x^n=1,\quad y^n=1,\quad xy=yx,\quad zx=yz,\quad zy=xz,\quad z^2=\frac{1}{n}\sum_{i,j=0}^{n-1}q^{-ij}x^iy^j.
\]
The Hopf algebra structure is defined by
\[
\begin{split}
&\Delta(x)=x\otimes x,\qquad \Delta (y)=y\otimes y,\qquad \Delta(z)=\frac{1}{n}\sum_{i,j=0}^{n-1}q^{-ij}(x^i\otimes y^j)(z\otimes z),\\&\varepsilon(x)=1,\quad \varepsilon(y)=1,\quad  \varepsilon (z)=1,\quad S(x)=x^{-1}=x^{n-1},\quad S(y)=y^{-1}=y^{n-1},\quad S(z)=z.
\end{split}
\]
One can check that $\{x^iy^j, x^iy^jz\,\vert\, 0\leq i,j\leq n-1\,\}$ is a basis of $H_{2n^2}$. Moreover, all universal $\Rr$-matrices for $H_{2n^2}$, with $n\geq 3$, are given in \cite[Proposition 3.11]{ZL}.
We prove that the infinitesimal $\Rr$-matrices on $(H_{2n^2}, \Rr)$ with $n \geq 3$ consist only of the zero one. The following proof does not depend on the explicit form of $\Rr$ but only on the fact that $z$ does not appear in the quasitriangular structure. %\eqref{R-generKP}.

\begin{theorem}\label{prop:H2n2}
The Hopf algebra $H_{2n^2}$ with $n\geq 3$, has no non-zero infinitesimal $\Rr$-matrices.
\end{theorem}
\begin{proof}
Set $G:=\{x^iy^j\,\vert\, 0\leq i,j\leq n-1\}$ and $A:=\Bbbk G$. Notice that $G$ is an abelian group (so $A$ is commutative), $H_{2n^2}=A\oplus zA$ and $\Delta(A) \subseteq A \otimes A$. Furthermore, all the quasitriangular structures of $H_{2n^2}$ with $n \geq 3$ are contained in $A \otimes A$ (see \cite[Proposition 3.11]{ZL}). Consider an element $$\chi\in H_{2n^2}\otimes H_{2n^2}=(A\oplus zA)\otimes (A\oplus zA)=(A\otimes A)\oplus (A\otimes zA)\oplus (zA\otimes A)\oplus (zA\otimes zA),$$
 that is,
	\[
	\chi=a\otimes b+c\otimes zd+ ze\otimes f+zp\otimes zr,
	\]
for $a,b,c,d,e,f,p,r \in A$. 
If $\chi$ satisfies \eqref{cqtr2}, then we have 
\[
(\mathrm{Id}\otimes \Delta)(\chi)=a\otimes \Delta(b)+c\otimes \Delta(z)\Delta(d)+ze\otimes\Delta (f)+zp\otimes \Delta(z)\Delta(r) 
\]
and 
\[
\begin{split}
\chi_{12}+\Rr^{-1}_{12}\chi_{13}\Rr_{12}&=a\otimes b\otimes 1+c\otimes zd\otimes 1+ze\otimes f\otimes 1+zp\otimes zr\otimes 1\\&+\Rr^{-1}_{12}(a\otimes 1\otimes b+c\otimes 1\otimes zd+ze\otimes 1\otimes f+zp\otimes 1\otimes zr)\Rr_{12}\\
&=a\otimes b\otimes 1+c\otimes zd\otimes 1+ze\otimes f\otimes 1+zp\otimes zr\otimes 1\\&+a\otimes 1\otimes b+c\otimes 1\otimes zd+\Rr^{-1}_{12}(ze\otimes 1\otimes f+zp\otimes 1\otimes zr)\Rr_{12}
\end{split}\]
are equal:
\[
\begin{split}
   a\otimes \Delta(b)&+c\otimes \Delta(z)\Delta(d)+ze\otimes\Delta (f)+zp\otimes \Delta(z)\Delta(r)=a\otimes b\otimes 1+c\otimes zd\otimes 1\\&+ze\otimes f\otimes 1+zp\otimes zr\otimes 1+a\otimes 1\otimes b+c\otimes 1\otimes zd+\Rr^{-1}_{12}(ze\otimes 1\otimes f+zp\otimes 1\otimes zr)\Rr_{12}.  
\end{split}
\]
Observe that $a \otimes \Delta(b) \in A \otimes A \otimes A$, $c\otimes \Delta(z)\Delta(d) \in A \otimes zA \otimes zA$, $ze\otimes\Delta (f) \in zA \otimes A \otimes A$, $zp\otimes \Delta(z)\Delta(r) \in zA \otimes zA \otimes zA$. Furthermore, $\Rr^{-1}_{12}(ze\otimes 1\otimes f)\Rr_{12} \in zA \otimes A \otimes A$ and $\Rr^{-1}_{12}(zp\otimes 1\otimes zr)\Rr_{12} \in zA \otimes A \otimes zA$, since $\Rr \in A \otimes A$.
By linear independence, we immediately get that $c\otimes\Delta(z)\Delta(d)=0=zp\otimes \Delta(z)\Delta(r)$. By applying $\mathrm{Id}\otimes\varepsilon\otimes\mathrm{Id}$ to these equalities, we find $c\otimes zd=0=zp \otimes zr$.
Thus, we obtain $\chi=a\otimes b+ze\otimes f$.  Given that $\chi$ also satisfies \eqref{cqtr3}, we have
\[
\begin{split}
\Delta(a)\otimes b+\Delta(z)\Delta(e)\otimes f&=(\Delta\otimes \mathrm{Id})(\chi)
=\chi_{23}+\Rr^{-1}_{23}\chi_{13}\Rr_{23}\\&=1\otimes a\otimes b+1\otimes ze\otimes f+\Rr^{-1}_{23}(a\otimes 1\otimes b+ze\otimes 1\otimes f)\Rr_{23}\\
&=1\otimes a\otimes b+1\otimes ze\otimes f+a\otimes 1\otimes b+ ze\ot1\ot f.
\end{split}
\]
Observe that $\Delta(a)\otimes b \in A \otimes A \otimes A$, $\Delta(z)\Delta(e)\otimes f \in zA \otimes zA \otimes A$. By linear independence, we get $\Delta(z)\Delta(e)\otimes f =0$.
By applying $\mathrm{Id}\otimes \varepsilon\otimes\mathrm{Id}$, we find $ze\otimes f=0$, hence $\chi=a\otimes b \in A \otimes A$. Then, since $\Rr \in A \otimes A$ and $A$ is commutative, $\Rr_{12}^{-1}\chi_{13}\Rr_{12}=\chi_{13}=\Rr_{23}^{-1}\chi_{13}\Rr_{23}$. Therefore, \eqref{cqtr2} and \eqref{cqtr3} together become equivalent to $\chi\in P(A)\otimes P(A)$, cf. \cite[Example 2.3 (2)]{ABSW}, but $P(A)=\{0\}$, so $\chi=0$.
\end{proof}

\subsubsection{The Kac--Paljutkin Hopf algebra $H_{8}$} When $n=2$, we retrieve the well-known Kac--Paljutkin Hopf algebra $H_8$. Explicitly, $H_8$ is the algebra generated by $x$, $y$, and $z$, subject to the following relations: 
\[
x^2=1,\quad y^2=1,\quad xy=yx,\quad zx=yz,\quad zy=xz,\quad z^2=\frac{1}{2}(1+x+y-xy).
\] 
The Hopf algebra structure is given by
\[
\begin{split}
&\Delta(x)=x\otimes x,\quad \Delta (y)=y\otimes y,\quad \Delta(z)=\frac{1}{2}(z\otimes z)(1\otimes 1+y\otimes 1+1\otimes x-y\otimes x),\\&\varepsilon(x)=1,\quad \varepsilon(y)=1,\quad  \varepsilon (z)=1,\quad S(x)=x,\quad S(y)=y,\quad S(z)=z.
\end{split}
\]
This Hopf algebra is again quasitriangular, but it admits a richer variety of $\Rr$-matrices. Namely $H_8$ is quasitriangular via two families of universal $\Rr$-matrices listed in \cite[Lemma 5.4]{Wa3} and therein denoted by $\Rr^{H_8}_{pq}$ and  $\Rr^{H_8}_{l}$. The second family of $\Rr$-matrices is well-defined if, and only if, $\Bbbk$ contains a primitive $8$-th root of unity $\omega$. In order to write these quasitriangular structures in a more suitable form for our purposes, let us make use of a different presentation of $H_8$ (see \cite[Example $2.3$]{ZL}). One can define
\[
\begin{split}
e_1:=\frac{1}{4}(1+x+y+xy), \quad  & e_x:=\frac{1}{4}(1-x+y-xy),\\
e_y:=\frac{1}{4}(1+x-y-xy), \quad  & e_{xy}:=\frac{1}{4}(1-x-y+xy),\\
\end{split}
\]
$G:=\{1,x,y,xy\}$ and $A:=\Bbbk G$, so that $H_{8}=A\oplus zA$.
Then one can check that $A$ is generated by $e_1$, $e_x$, $e_y$, $e_{xy}$ and that the $e_g$'s, with $g \in G$, are pairwise orthogonal and idempotent. Moreover, we have that all quasitriangular structures on $H_8$ are given by
\begin{eqnarray*}
\Rr_{\alpha,\beta}&=& e_1 \otimes [e_1 + e_x + e_y + e_{xy}] + e_x \otimes [e_1 + \alpha e_x + \beta e_y + \alpha\beta e_{xy}]+\\
&&e_y \otimes [e_1 - \beta e_x + \alpha e_y - \alpha \beta e_{xy}] + e_{xy} \otimes [e_1 - \alpha\beta e_x + \alpha\beta e_y - e_{xy}]
\end{eqnarray*}
for $\alpha, \beta \in \{-1,1 \}$ and
\begin{eqnarray*}
\Rr_{\omega}&=& [e_1 \otimes e_1 + e_1 \otimes e_{xy}+ e_{xy}\otimes e_1 - e_{xy}\otimes e_{xy}]+\\
&&[e_1 \otimes e_x + e_1 \otimes e_y- \omega^2 e_{xy} \otimes e_x + \omega^2 e_{xy}\otimes e_y](z \otimes 1)+\\
&&[e_x \otimes e_1 + e_y \otimes e_1 + \omega^2 e_x \otimes e_{xy} -\omega^2 e_y \otimes e_{xy}](1 \otimes z)+\\
&&[\omega^{-1}e_x \otimes e_x + \omega e_x \otimes e_y + \omega e_y \otimes e_x + \omega^{-1}e_y \otimes e_y](z \otimes z)
\end{eqnarray*}
for all $\omega$ such that $\omega^4=-1$ (i.e. $\omega$ must be a primitive $8$-th root of unity). 

\begin{remark}
The presentation of $H_8$ given in \cite{Wa3} uses generators $g, h$ and $t$. One can rewrite $\Rr^{H_8}_{pq}$ and $\Rr^{H_8}_{l}$ respectively as $\Rr_{\alpha,\beta}$ and $\Rr_{\omega}$ by identifying:
\[
g \leadsto x,\quad  h \leadsto y,\quad t \leadsto \frac{z}{4}\left[1-\omega^2 +\sqrt{2}+(1+\omega^2 +\sqrt{2}\omega^2)x+(1+\omega^2 -\sqrt{2}\omega^2)y+(1-\omega^2 -\sqrt{2})xy \right].
\]
%\begin{eqnarray*}
    %g &\leadsto& x\\
    %h &\leadsto& y\\
    %t &\leadsto& \frac{z}{4}\left[1+\omega^2 +\sqrt{2}+(1-\omega^2 -\sqrt{2}\omega^2)x+(1-\omega^2 +\sqrt{2}\omega^2)y+(1+\omega^2 -\sqrt{2})xy \right].
%\end{eqnarray*}
This should be understood as a formal substitution since $\Bbbk$ does not need to contain a square root of $2$ in order for $\Rr_{\omega}$ to be well-defined. Notice also that $\Rr_{\alpha,\beta}$ was already explicitly determined in this form in \cite[Remark $3.12$]{ZL}.
\end{remark}

The first family of quasitriangular structures $\Rr_{\alpha,\beta}$ is of the same form of those of $H_{2n^2}$ for $n \geq 3$ and does not depend on the generator $z$ (see \cite[Remark $3.12$]{ZL}), so the proof of Theorem~\ref{prop:H2n2} can be adapted to show that any infinitesimal $\Rr$-matrix $\chi$ for $(H_8, \Rr_{\alpha,\beta})$, must be zero. Unfortunately, to determine the infinitesimal $\Rr$-matrices for $(H_8, \Rr_{\omega})$, we cannot perform the same steps, since $\Rr_{\omega}$ depends on the generator $z$. On account of this, we prove a list of equalities that will be useful for the computations we will carry out in the sequel. 

\begin{lemma}\label{lemma:useful}
The following equalities hold:
\begin{eqnarray}
&&\Rr_{\omega}^{-1}(e_1 \otimes 1)\Rr_{\omega}=e_1 \otimes 1,\label{eq:e1}\\
&&\Rr_{\omega}^{-1}(e_{xy} \otimes 1)\Rr_{\omega}=e_{xy} \otimes 1,\label{eq:exy}\\
&& \Rr_{\omega}^{-1}(e_x \otimes 1)\Rr_{\omega}=e_y \otimes 1+(e_x-e_y) \otimes (e_1+e_{xy})\label{eq:exl},\\
&& \Rr_{\omega}^{-1}(e_y \otimes 1)\Rr_{\omega}=e_x \otimes 1-(e_x-e_y) \otimes (e_1+e_{xy})\label{eq:eyl},\\
&& \Rr_{\omega}^{-1}(z \otimes 1) \Rr_{\omega}=(z \otimes 1) [1 \otimes 1-(e_x+e_y) \otimes (e_x+e_y-2e_{xy})-\omega^2(e_x -e_y) \otimes (e_x-e_y)]. \label{eq:zl}
\end{eqnarray}
\end{lemma}

\begin{proof}
Let us define
\begin{eqnarray*}
\Rr_{\omega}^{00}&:=&e_1 \otimes e_1 + e_1 \otimes e_{xy}+ e_{xy}\otimes e_1 - e_{xy}\otimes e_{xy}\\
\Rr_{\omega}^{10}&:=&e_1 \otimes e_x + e_1 \otimes e_y- \omega^2 e_{xy} \otimes e_x + \omega^2 e_{xy}\otimes e_y\\
\Rr_{\omega}^{01}&:=&e_x \otimes e_1 + e_y \otimes e_1 + \omega^2 e_x \otimes e_{xy} -\omega^2 e_y \otimes e_{xy}\\
\Rr_{\omega}^{11}&:=&\omega^{-1}e_x \otimes e_x + \omega e_x \otimes e_y + \omega e_y \otimes e_x + \omega^{-1}e_y \otimes e_y,
\end{eqnarray*}
so that
\[\Rr_{\omega}=\Rr_{\omega}^{00}+\Rr_{\omega}^{10}(z \otimes 1)+\Rr_{\omega}^{01}(1 \otimes z)+\Rr_{\omega}^{11}(z \otimes z).\]
Clearly each $\Rr_{\omega}^{ij}$ ($i,j=0,1$) is in the center of $A \otimes A$ since it does not contain $z$ and $x$ and $y$ commute. Moreover, the antipode $S$ is the identity on $x$ and $y$, hence $(S \otimes \textrm{Id})(\Rr_{\omega}^{ij})=\Rr_{\omega}^{ij}$. Thus, we can compute
\[\Rr_{\omega}^{-1}=(S \otimes \textrm{Id})(\Rr_{\omega})=\Rr_{\omega}^{00}+(z \otimes 1)\Rr_{\omega}^{10}+\Rr_{\omega}^{01}(1 \otimes z)+(z \otimes 1)\Rr_{\omega}^{11}(1 \otimes z).\]
It is straightforward to check that $e_1z=ze_1$, $e_xz=ze_y$, $e_yz=ze_x$ and $e_{xy}z=ze_{xy}$. Then \eqref{eq:e1} and \eqref{eq:exy} follow immediately.
Next, since the $e_g$'s (with $g\in G$), are pairwise orthogonal, we have
\begin{eqnarray*}
\Rr_{\omega}^{-1}(e_x \otimes 1)\Rr_{\omega}-e_y \otimes 1&=&\Rr_{\omega}^{-1}(e_x \otimes 1)(\Rr_{\omega}^{00}+\Rr_{\omega}^{10}(z \otimes 1)+\Rr_{\omega}^{01}(1 \otimes z)+\Rr_{\omega}^{11}(z \otimes z))-e_y \otimes 1\\
&=&\Rr_{\omega}^{-1}(\Rr_{\omega}^{00}(e_x \otimes 1)+\Rr_{\omega}^{10}(e_xz \otimes 1)+\Rr_{\omega}^{01}(e_x \otimes z)+\Rr_{\omega}^{11}(e_x z \otimes z))-e_y \otimes 1\\
&=&\Rr_{\omega}^{-1}(\Rr_{\omega}^{00}(e_y \otimes 1)+\Rr_{\omega}^{10}(ze_y \otimes 1)+\Rr_{\omega}^{01}(e_x \otimes z)+\Rr_{\omega}^{11}( ze_y \otimes z))-e_y \otimes 1\\
&=&\Rr_{\omega}^{-1}\Rr_{\omega}(e_y \otimes 1) -\Rr_{\omega}^{-1}\Rr_{\omega}^{01}(e_y \otimes z)+\Rr_{\omega}^{-1}\Rr_{\omega}^{01}(e_x \otimes z)-e_y \otimes 1\\
&=&\Rr_{\omega}^{-1}\Rr_{\omega}^{01}[(e_x-e_y) \otimes z]\\
&=&\Rr_{\omega}^{-1}[(e_x - e_y) \otimes e_1 + \omega^2 (e_x + e_y) \otimes e_{xy}](1 \otimes z)\\
&=&[\Rr_{\omega}^{01}+(z \otimes 1)\Rr_{\omega}^{11}][(e_x - e_y) \otimes e_1 + \omega^2 (e_x + e_y) \otimes e_{xy}](1 \otimes z^2)\\
&=&\Rr_{\omega}^{01}[(e_x - e_y) \otimes e_1 + \omega^2 (e_x + e_y) \otimes e_{xy}](1 \otimes z^2)\\
&=&(e_x - e_y) \otimes (e_1 + \omega^4 e_{xy})z^2\\
&=&(e_x - e_y) \otimes (e_1 - e_{xy})(e_1+e_x+e_y-e_{xy})\\
&=&(e_x - e_y) \otimes (e_1 + e_{xy}),\\
\end{eqnarray*}
using that $z^{2}=e_{1}+e_{x}+e_{y}-e_{xy}$, so \eqref{eq:exl} is proven. Equation \eqref{eq:eyl} can be obtained with similar calculations. Finally, in view of the fact that also the $\Rr^{ij}_{\omega}$'s are pairwise orthogonal and $z^{4}=1$, we have 
\begin{eqnarray*}
\Rr_{\omega}^{-1}(z \otimes 1) \Rr_{\omega}&=&[\Rr_{\omega}^{00}(z \otimes 1)+(z \otimes 1)\Rr_{\omega}^{10}(z \otimes 1)+\Rr_{\omega}^{01}(z \otimes z)+(z \otimes 1)\Rr_{\omega}^{11}(z \otimes z)]\Rr_{\omega}\\
&=&[(z \otimes 1)\Rr_{\omega}^{00}+(z^2 \otimes 1)\Rr_{\omega}^{10}+(z \otimes 1)[(e_x+e_y) \otimes e_1 + \omega^2 (e_y-e_x) \otimes e_{xy}](1 \otimes z)+\\
&+&(z^2 \otimes 1)[(\omega^{-1}e_y+\omega e_x) \otimes e_x + (\omega e_y   + \omega^{-1}e_x) \otimes e_y](1 \otimes z)]\Rr_{\omega}\\
&=&\left[(z \otimes 1)\Rr_{\omega}^{00}+(z^2 \otimes 1)\Rr_{\omega}^{10}+(z \otimes z)[(e_x+e_y) \otimes e_1 + \omega^2 (e_y-e_x) \otimes e_{xy}]+ \right.\\
&+&\left. (z^2 \otimes z)[(\omega^{-1}e_y+\omega e_x) \otimes e_y + (\omega e_y   + \omega^{-1}e_x) \otimes e_x]\right]\Rr_{\omega}\\
&=&(z \otimes 1)(\Rr_{\omega}^{00})^2+(z^2 \otimes 1)(\Rr_{\omega}^{10})^2(z \otimes 1)+(z \otimes z)[(e_x+e_y) \otimes (e_1 +e_{xy})](1 \otimes z)+ \\
&+& (z^2 \otimes z)[(\omega^{-2}e_y+\omega^2 e_x) \otimes e_y + (\omega^2 e_y   + \omega^{-2}e_x) \otimes e_x](z \otimes z)\\
&=&(z \otimes 1)[(e_1+e_{xy})\otimes(e_1+e_{xy})]+(z^2 \otimes 1)[(e_1-e_{xy})\otimes(e_x+e_y)](z \otimes 1)+\\
&+&(z \otimes z)[(e_x+e_y) \otimes (e_1 +e_{xy})](1 \otimes z)+\omega^2 (z^2 \otimes z)[(e_x- e_y) \otimes (e_y-e_x)](z \otimes z)\\
&=&(z \otimes 1)[(e_1+e_{xy})\otimes(e_1+e_{xy})]+(z \otimes 1)[(e_1+e_{xy})\otimes(e_x+e_y)]+\\
&+&(z \otimes 1)[(e_x+e_y) \otimes (e_1 -e_{xy})]+\omega^2 (z^3 \otimes z^2)[(e_x- e_y) \otimes (e_y-e_x)]\\
&=&(z \otimes 1)[(e_1+e_{xy})\otimes(e_1+e_{xy})]+(z \otimes 1)[(e_1+e_{xy})\otimes(e_x+e_y)]+\\
&+&(z \otimes 1)[(e_x+e_y) \otimes (e_1 -e_{xy})]+\omega^2 (z \otimes 1)[(e_x- e_y) \otimes (e_y-e_x)]\\
&=&(z \otimes 1)[1\otimes 1-(e_x+e_y) \otimes (e_x+e_y+2e_{xy})-\omega^2 (e_x- e_y) \otimes (e_x-e_y)]\\
\end{eqnarray*}
and therefore \eqref{eq:zl} holds true.
\end{proof}

\begin{remark}\label{conjugacy}
An immediate consequence of the previous lemma is that $\Rr_{\omega}^{-1}(a \otimes 1) \Rr_{\omega} \in A \otimes A$ and $\Rr_{\omega}^{-1}(za \otimes 1) \Rr_{\omega} \in zA \otimes A$ for any $a \in A$.
For a more detailed result one can  further observe that
\[\Rr_{\omega}^{-1}(e_x \otimes 1)\Rr_{\omega}=e_y \otimes 1+(e_x-e_y) \otimes (e_1+e_{xy})=\frac{1}{2}(e_x+e_y) \otimes 1+ \frac{1}{2}(e_x-e_y) \otimes xy\]
and
\[\Rr_{\omega}^{-1}(e_y \otimes 1)\Rr_{\omega}=e_x \otimes 1-(e_x-e_y) \otimes (e_1+e_{xy})=\frac{1}{2}(e_x+e_y)\otimes 1-\frac{1}{2}(e_x-e_y) \otimes xy.\]
Then it becomes clear that $\Rr_{\omega}^{-1}(a \otimes 1) \Rr_{\omega} \in A \otimes (\Bbbk1 \oplus \Bbbk xy)$ for any $a \in A$.
\end{remark}
In order to show that every infinitesimal $\Rr$-matrix $\chi$ for $(H_8,\Rr_{\omega})$ is zero, we can assume $\chi$ to be a 2-cocycle of $H_8$ for the cohomology defined in Section \ref{sub:Hoch}. Indeed, we recall that an infinitesimal $\Rr$-matrix $\chi$ for $H_8$ is an element of $\mathrm{Z}^2(H_8,\Bbbk)$ %see \cite[Theorem 2.21]{ABSW}} 
that satisfies \eqref{cqtr1} and \eqref{cqtr2} (see Proposition \ref{prop:newcharacterization}). Since the base field contains the 4th root of the unity $\omega^2$, by \cite[Remark 2.14 (3)]{Ma} $H_8$ is cosemisimple, and then coseparable \cite[Theorem, p. 265]{Lar73}. By \cite[Theorem 3]{Doi}, 
%\sout{Furthermore, by \cite[Corollary 3.7]{St} $H_{8}$ is separable, since it is semisimple, and therefore} 
it follows that $\mathrm{H}^{2}(H_{8},\Bbbk)=\{0\}$, i.e. $\mathrm{Z}^{2}(H_{8},\Bbbk)=\mathrm{B}^{2}(H_{8},\Bbbk)$.

Hence, to prove that $\chi=0$ is the only infinitesimal $\Rr$-matrix for $(H_{8},\Rr_{\omega})$, we now impose axiom \eqref{cqtr2} on elements of the form $b^1(u+zv)$, with $u, v \in A$. %We will denote with $u_{ij}$ and $v_{ij}$ the coordinates of $u$ and $v$ on the usual basis of $A$.

\begin{proposition}
The Hopf algebra $(H_8, \Rr_{\omega})$ has no non-zero infinitesimal $\Rr$-matrices.
\end{proposition}

\begin{proof}
Consider an infinitesimal $\Rr$-matrix $\chi$ for $(H_8, \Rr_{\omega})$. %By \cite[Theorem 2.21]{ABSW} $\chi$ is a Hochschild $2$-cocycle of $H_8$ and thus, by Theorem \ref{2cocyclesH8}, 
As already observed, $\chi$ must be a $2$-coboundary of $H_8$. Therefore $\chi=b^1(u+zv)=1 \otimes u + 1 \otimes zv +u \otimes 1 +zv \otimes 1-\Delta(u)-\Delta(zv)$ for some $u, v \in A$.
We have 
\begin{eqnarray*}
(\textrm{Id} \otimes \Delta)(\chi)&=&1 \otimes \Delta(u) + 1 \otimes \Delta(zv) +u \otimes 1 \otimes 1 +zv \otimes 1 \otimes 1-(\textrm{Id} \otimes \Delta)\Delta(u)-(\textrm{Id} \otimes \Delta)\Delta(zv),\\
\chi_{12}&=&1 \otimes u \otimes 1 + 1 \otimes zv \otimes 1 +u \otimes 1 \otimes 1 +zv \otimes 1 \otimes 1-\Delta(u) \otimes 1-\Delta(zv) \otimes 1,\\
(\Rr_{\omega}^{-1})_{12}\chi_{13}(\Rr_{\omega})_{12}&=&1 \otimes 1 \otimes  u + 1 \otimes 1 \otimes zv +\Rr_{\omega}^{-1}(u \otimes 1)\Rr_{\omega}\otimes 1 +\Rr_{\omega}^{-1}(zv \otimes 1)\Rr_{\omega} \otimes 1\\
&-&\Rr_{\omega}^{-1}(u_1 \otimes 1)\Rr_{\omega} \otimes u_2-\Rr_{\omega}^{-1}(z_1v_1 \otimes 1)\Rr_{\omega} \otimes z_2v_2,
\end{eqnarray*}
hence \eqref{cqtr2} is equivalent to
\begin{equation}\label{eq:ax5}
\begin{split}
   1 \otimes &\Delta(u) + 1 \otimes \Delta(zv) -(\textrm{Id} \otimes \Delta)\Delta(u)-(\textrm{Id} \otimes \Delta)\Delta(zv)=1 \otimes u \otimes 1 + 1 \otimes zv \otimes 1 -\Delta(u) \otimes 1\\&-\Delta(zv) \otimes 1 +1 \otimes 1 \otimes  u + 1 \otimes 1 \otimes zv +\Rr_{\omega}^{-1}(u \otimes 1)\Rr_{\omega}\otimes 1 +\Rr_{\omega}^{-1}(zv \otimes 1)\Rr_{\omega} \otimes 1\\
   &-\Rr_{\omega}^{-1}(u_1 \otimes 1)\Rr_{\omega} \otimes u_2-\Rr_{\omega}^{-1}(z_1v_1 \otimes 1)\Rr_{\omega} \otimes z_2v_2.
   \end{split}
\end{equation}
By Lemma~\ref{lemma:useful} we have that $\Rr_{\omega}^{-1}(u \otimes 1)\Rr_{\omega} \otimes 1 \in A \otimes A \otimes A$, $\Rr_{\omega}^{-1}(u_1 \otimes 1)\Rr_{\omega} \otimes u_2 \in A \otimes A \otimes A$, $\Rr_{\omega}^{-1}(zv \otimes 1)\Rr_{\omega} \otimes 1 \in zA \otimes A \otimes A$ and $\Rr_{\omega}^{-1}(z_1v_1 \otimes 1)\Rr_{\omega} \otimes z_2v_2 \in zA \otimes A \otimes zA$ (see Remark~\ref{conjugacy}). Moreover $\Delta(u) \in A \otimes A$ and $\Delta(zv) \in zA \otimes zA$, so it becomes clear that, by linear independence,  $1 \otimes zv \otimes 1=0$, i.e. $zv=0$. Then \eqref{eq:ax5} boils down to
\begin{equation}\label{eq:ax5'}
1 \otimes \Delta(u)  -(\textrm{Id} \otimes \Delta)\Delta(u)=1 \otimes u \otimes 1  -\Delta(u) \otimes 1 +1 \otimes 1 \otimes  u  +\Rr_{\omega}^{-1}(u \otimes 1)\Rr_{\omega}\otimes 1 -\Rr_{\omega}^{-1}(u_1 \otimes 1)\Rr_{\omega} \otimes u_2.
\end{equation}
Again, thanks to Lemma~\ref{lemma:useful}, we can easily see that the only term appearing in this equation that has $x$ in every tensorand is $-u_{10}(x \otimes x \otimes x)$ and the only one with $y$ in every tensorand is $-u_{01}(y \otimes y \otimes y)$ (see again Remark~\ref{conjugacy}). From this we immediately get $u_{10}=u_{01}=0$, so that $u=u_{00}1+u_{11}xy$. Moreover, since $zxy=xyz$, \eqref{eq:ax5'} reduces to
\[
\begin{split}
u_{11}1\otimes xy\otimes xy-u_{11}xy\otimes xy\otimes xy&=u_{00}1\otimes1\otimes1+u_{11}1\otimes xy\otimes 1-u_{11}xy\otimes xy\otimes1\\&+u_{11}1\otimes1\otimes xy+u_{11}xy\otimes1\otimes1-u_{11}xy\otimes 1\otimes xy
\end{split}
\]
and then, by linear independence, it immediately follows that $u_{00}=u_{11}=0$.
\end{proof}

We have previously observed that also $(H_8, \Rr_{\alpha,\beta})$ has no non-zero infinitesimal $\Rr$-matrices, therefore we can state the following:

\begin{theorem}\label{thm:H8}
The Hopf algebra $H_8$ has no non-zero infinitesimal $\Rr$-matrices.
\end{theorem}

\begin{remark}
    We point out that Theorem \ref{prop:H2n2} and Theorem \ref{thm:H8} can be recovered from general results, in case $\Bbbk$ is algebraically closed of characteristic 0. This observation was made by M. Faitg, A. Gainutdinov and C. Schweigert, after our paper appeared online. In \cite[Definition 4.1]{FGS} a notion of infinitesimal braiding \textit{tangent} to a braiding $c$ is introduced. As it is said in \cite[Remark 4.2]{FGS}, $t$ is an infinitesimal braiding in the sense of \cite[Definition 1.1]{ABSW} if and only if $ct$ is a braiding tangent to $c$. Considered the braided monoidal category $_{H}\Mm$ of modules over a quasitriangular bialgebra $(H,\Rr)$, we have that $\chi$ is an infinitesimal $\Rr$-matrix if and only if $\Rr\chi$ satisfies equation (128) in \cite[p. 46]{FGS}, see \cite[Remark 5.2]{FGS}. 
    As it is said in \cite[p. 5]{FGS}, using \cite[Theorem 2]{FGS} and \textit{Ocneanu rigidity} \cite[§7]{ENO}, \cite[§3.5]{GHS}, one obtains that there are no non-zero infinitesimal braidings tangent to a braiding $c$ for a semisimple finite category $\Cc$, if the base field $\Bbbk$ is algebraically closed of characteristic 0. Therefore, there are no non-zero infinitesimal braidings for $_{H}\Mm$ under such assumptions. Hence, if the base field $\Bbbk$ is algebraically closed of characteristic 0, there are no non-zero infinitesimal $\Rr$-matrices for a semisimple quasitriangular Hopf algebra $(H,\Rr)$.
\end{remark}

\subsection{The Radford Hopf algebras $H_{(r,n,q)}$}\label{radfordalgebras}
We recall from \cite[p. 477]{Rad3} the construction of the so-called Radford pointed Hopf algebras introduced in \cite{Rad}. Let $1\leq r,n\in\mathbb{N}$, $M=rn$ and suppose that $\Bbbk$ is a field of characteristic $\mathrm{char}(\Bbbk)\neq 2$ with a primitive $M$-th root of unity $q$. Let $H=H_{(r,n,q)}$ be the Hopf algebra over $\Bbbk$ described as follows. As an algebra $H$ is generated by $g$ and $x$ subject to the relations
\[
g^{M}=1,\quad x^{n}=0,\quad xg=qgx
\]
and the coalgebra structure of $H$ is given by 
\[
\Delta(g)=g\otimes g,\quad \Delta(x)=1\otimes x+x\otimes g^{r},\quad \varepsilon(g)=1_{\Bbbk},\quad \varepsilon(x)=0. 
\]
The antipode is given by $S(g)=g^{M-1}$, $S(x)=-xg^{M-r}$. As a $\Bbbk$-vector space $H_{(r,n,q)}$ has dimension $rn^{2}$ with basis $\{g^{l}x^{m}\ |\ 0\leq l<M,\,0\leq m<n\}$. Let us recall that 
\[
\Delta(g^{l}x^{m})=\sum_{u=0}^{m}{\binom{m}{u}_{Q}g^{l}x^{m-u}\otimes g^{l+r(m-u)}x^{u}}
\]
for all $0\leq l<M$ and $0\leq m<n$, where $Q=q^{r}$ and $\left( \ \right)_Q$ is used to denote a $Q$-binomial symbol (or gaussian coefficient). We have that $G:=G(H)=\{g^{i}\ |\ 0\leq i< M\}$. 

Let us observe that if $r=1$ we recover Taft Hopf algebras, which do not admit a quasitriangular structure for $n>2$, as it is shown in \cite{Ge2}. On the other hand, quasitriangular structures for some of the Radford Hopf algebras with $r \geq 2$ were classified in \cite{Rad}. %\rd{We point out that $H_{(r,n,q)}$ has quasitriangular structure if and only if $n=2$ and $r$ is odd, see \cite[Corollary 3]{Rad}.} 
However, we prove that $\chi=0$ is the only infinitesimal $\Rr$-matrix for any Radford Hopf algebra $H_{(r,n,q)}$ which is quasitriangular. This does not depend on the specific form of the quasitriangular structure.

\begin{theorem}\label{prop:Rad}
    The Radford Hopf algebra $H:=H_{(r,n,q)}$, with $r\geq2$, has no non-zero infinitesimal $\Rr$-matrices.
\end{theorem}

\begin{proof}
Set $G:=\{1,g,\ldots,g^{M-1}\}$ and $A:=\Bbbk G$, so that $H=A\oplus Ax\oplus Ax^{2}\oplus\cdots\oplus Ax^{n-1}$, hence
\[
H\otimes H=(A\otimes A)\oplus\cdots\oplus(A\otimes Ax^{n-1})\oplus\cdots\oplus(Ax^{n-1}\otimes A)\oplus\cdots\oplus(Ax^{n-1}\otimes Ax^{n-1}).
\]
%\[
%\rd{H\ot H=(A\ot A)\oplus(A\ot Ax)\oplus(Ax\ot A)\oplus(Ax\ot Ax)}
%\]
The axiom \eqref{cqtr1} on $g$ is equivalent to $g\chi^{i}g^{-1}\otimes g\chi_{i}g^{-1}=\chi^{i}\otimes\chi_{i}$. But the conjugation on every summand (except $A\otimes A$) produces a scalar $q^{s}$ with $1\leq s\leq 2n-2$. Since $2\leq r$, we have $2n-2\leq rn-2=M-2<M$, hence $q^{s}$ is always different from 1 so that, by linear independence, we must have $\chi\in A\otimes A$. Writing $\chi=a\otimes b$ with $a,b\in A$ we impose \eqref{cqtr1} on $x$, i.e. 
\[
a\otimes xb+xa\otimes g^{r}b=\Delta(x)\chi=\chi\Delta(x)=a\otimes bx+ax\otimes bg^{r}. 
\]
By linear independence, looking at those terms with $x$ only in the right tensorand, we obtain $a\otimes xb=a\otimes bx$ and then $a=0$ or $xb=bx$. Suppose $a\not=0$, so $xb=bx$. Now, writing $b$ in terms of the basis of $A$, one deduces that $b=1$. Indeed, suppose $b$ contains an element $\alpha g^{i}$, with $0\not=i<M$ and $\alpha\in\Bbbk$, then, since $xg^{i}=q^{i}g^{i}x$, by linear independence one obtains $\alpha q^{i}g^{i}x=\alpha g^{i}x$, hence $\alpha=0$ and then $b=1$. Similarly, looking at those elements with $x$ only in the left tensorand, by linear independence we must have $xa\otimes g^{r}=ax\otimes g^{r}$. Applying $\varepsilon$ on the right tensorand we get $ax=xa$. By writing $a$ in terms of the basis of $A$, again by linear independence we obtain $a=1$. Then $\chi=1\otimes 1$ and \eqref{cqtr2} and \eqref{cqtr3} are not satisfied. Hence $a=0$, i.e. $H$ has only the infinitesimal $\Rr$-matrix $\chi=0$.
\end{proof}

\subsection{The Hopf algebras $E(n)$}\label{E(n)algebras}

We now classify infinitesimal $\Rr$-matrices for the family of Hopf algebras $E(n)$, which was obtained in \cite{BDG} by Ore extensions.  We first recall from \cite{CD} the definition and some properties of $E(n)$. 

Let us fix a field $\Bbbk$ of characteristic $\mathrm{char}(\Bbbk)\neq 2$ and a natural number $n\geq1$. The $2^{n+1}$-dimensional Hopf algebra $E(n)$ is generated as an algebra by $g,x_i$, $i=1, \ldots, n$, and relations 
\[
g^2=1,\qquad x_i^2=0, \qquad gx_i=-x_ig, \qquad x_ix_j=-x_jx_i,
\]
for every $i,j=1, \ldots, n$.
The Hopf algebra structure is defined by
\[\Delta(g)=g \otimes g, \qquad \Delta(x_i)=x_i \otimes 1 +g \otimes x_i,\qquad \varepsilon(g)=1_{\Bbbk},\qquad \varepsilon(x_i)=0,
\]
for all $i=1, \ldots, n$. Moreover, the antipode is given by $S(g)=g$ and $S(x_i)=-gx_i$, for all $i=1, \ldots, n$. For $P=\lbrace i_1,i_2,\ldots,i_s \rbrace \subseteq \lbrace 1,2, \ldots,n \rbrace$ such that $i_1 < i_2 < \cdots < i_s$, we denote $x_P = x_{i_1}x_{i_2}\cdots x_{i_s}$. If $P=\emptyset$ then $x_{\emptyset} = 1$. The set $\lbrace g^jx_P \ | \ P\subseteq \lbrace 1, \ldots ,n \rbrace, j \in \lbrace 0,1 \rbrace \rbrace$ is a basis of $E(n)$.
Let $F=\lbrace i_{j_1}, i_{j_2}, \ldots, i_{j_r} \rbrace$ be a subset of $P$ and define
\[S(F,P)=(j_1+\cdots+j_r)-\frac{r(r+1)}{2} \textrm{ and } S(\emptyset,P)=0.\]
Then, one can show that
\begin{equation}\label{deltagx}
\Delta(g^jx_P)=\sum_{F \subseteq P} (-1)^{S(F,P)} g^{|F|+j}x_{P \setminus F} \otimes g^jx_F,
\end{equation}
\begin{equation}\label{antipode}
S(g^jx_P) = (-1)^{|P|(j+1)}g^{|P|+j}x_P.
\end{equation}

Let us recall some facts about the $E(n)$:

\begin{itemize}
\item The Hopf algebras $E(n)$ are generalizations of well-known cases in literature. For example $E(1)$ is the Sweedler's Hopf algebra, and $E(2)$ is the 8-dimensional unimodular ribbon Hopf algebra introduced in \cite{Rad}. On the other hand, the $E(n)$ are a particular case of the Hopf algebras $A_{m,n}$, for positive integers $m,n$, defined by Radford in \cite{Rad4}. More precisely, we have $E(n) = A_{1,n}$. Since $A_{m,n}$ do not admit any quasitriangular structure if $m\geq 2$, see \cite[Theorem 3.2 and Remark 3.3]{N}, then they cannot be pre-Cartier for $m\geq 2$.
\item The $E(n)$ are pointed  and their coradical is $\Bbbk C_2$, cf. \cite[Proposition 1.16 (i)]{BDG}.
\item If $\Bbbk$ is algebraically closed and $\mathrm{char}(\Bbbk)=0$, then $E(n)$ is, up to isomorphism, the only $2^{n+1}$-dimensional pointed Hopf algebra with coradical $\Bbbk C_{2}$, see \cite[Theorem 2]{CD}.
\item The $E(n)$ are self-dual, i.e. $E(n) \cong E(n)^*$ as Hopf algebras, see \cite[Proposition 1]{PVO}.
\item $(E(n),\Rr)$ is a quasitriangular Hopf algebra if and only if there
exists $A \in M_n(\Bbbk)$ such that $\Rr=R_A$, where 
\begin{eqnarray}\label{RstructureE(n)}
R_A:=\frac{1}{2}(1 \otimes 1 + 1 \otimes g + g \otimes 1-g \otimes g)+\frac{1}{2}\Big[\sum_{|F|=|P|}(-1)^{\frac{|P|(|P|-1)}{2}}\det(P,F) \times \label{tristruct}\\ 
\times (g^{|P|}x_F \otimes x_P + g^{|P|}x_F \otimes gx_P +g^{|P|+1}x_F     \otimes x_P -  g^{|P|+1}x_F \otimes gx_P)\Big],  \nonumber
\end{eqnarray}
where $P$ and $F$ are two subsets of $\lbrace 1, \ldots, n \rbrace$ with $|F|=|P|$ and $\det(P,F)$ denotes the determinant of the matrix obtained at the intersection of the rows indexed by $P$, and the columns indexed by $F$, of the matrix $A$, see \cite[Proposition 2]{PVO}.
\item $(E(n),\Rr)$ is a triangular Hopf algebra if, and only if, $\Rr=R_A$ for some symmetric $A \in M_n(\Bbbk)$, see \cite[Proposition 2.1]{CC}, \cite[Proposition 7]{PVO}.
\end{itemize}

In order to classify infinitesimal $\Rr$-matrices for $E(n)$ we will use results concerning bosonization and cohomology of smash products. First we are going to show that $E(n)$ can be realized as the smash biproduct of the algebra generated by the skew-primitive elements $x_i$ and the group algebra $\Bbbk C_2$. We will use this decomposition to further prove that in the cohomology for coalgebras (recalled in Section \ref{sub:Hoch}) the group  $\mathrm{H}^{2}(E(n),\Bbbk)$ has dimension $\frac{n(n+1)}{2}$. In turn, this will help us to determine a decomposition of the space $\mathrm{Z}^{2}(E(n),\Bbbk)$ of 2-cocycles as a direct sum of the space $\mathrm{B}^{2}(E(n),\Bbbk)$ of $2$-coboundaries, with a suitable vector space. We will also show that this decomposition is inherited when we restrict our attention to infinitesimal $\Rr$-matrices: the ones who give $E(n)$ a Cartier structure are exactly those who are $2$-coboundaries.

\medskip

\subsubsection{Radford-Majid bosonization}
Let us recall that a \textit{Hopf algebra with a projection} \cite{Rad2} is a quadruple $(A,H,\gamma,\pi)$ where $A$ and $H$ are Hopf algebras and $\gamma:H\to A$ and $\pi:A\to H$ are Hopf algebra maps such that $\pi\circ\gamma=\mathrm{Id}_{H}$. Given a Hopf algebra with a projection $(A,H,\gamma,\pi)$, $A$ becomes an object in $\mathfrak{M}^{H}_{H}$, i.e. a right $H$-module, a right $H$-comodule such that the action is right $H$-colinear or, equivalently, the coaction is right $H$-linear. In particular, the $H$-action and the $H$-coaction are given by
\[
\mu:A\otimes H\to A,\ a\otimes h\mapsto a\gamma(h),\qquad \rho:A\to A\otimes H,\ a\mapsto a_{1}\otimes\pi(a_{2}). 
\]
Thus, denoting $R:=\{a\in A\,\vert\, \rho(a)=a\otimes 1_{H}\}$, by applying the Structure Theorem for Hopf modules, one obtains that
\[
\theta: R\otimes H\to A,\ a\otimes h\mapsto a\gamma(h)
\]
is an isomorphism in $\mathfrak{M}^{H}_{H}$, where  $R\otimes H$ is in $\mathfrak{M}^{H}_{H}$ with action $\mathrm{Id}_{R}\otimes m_{H}$ and coaction $\mathrm{Id}_{R}\otimes\Delta_{H}$. %with inverse
%\[
%\theta^{-1}:A\to A^{\mathrm{co}H}\otimes H,\ a\mapsto a_{1}\gamma(S(\pi(a_{2})))\otimes\pi(a_{3}).
%\]
%and $B:=R\otimes H$. 
Since $\theta:R\otimes H\to A$ is bijective, one can transfer to $R\otimes H$ the Hopf algebra structure of $A$ in
a unique way such that $\theta$ becomes a Hopf algebra map. One has $(R,\mu_{R},\rho_{R})\in{}^{H}_{H}\mathcal{YD}$ with
\[
\mu_{R}:H\otimes R\to R,\ h\otimes r\mapsto\gamma(h_{1})r\gamma(S(h_{2}))\ (\textit{adjoint action}),\quad \rho_{R}:R\to H\otimes R,\ r\mapsto \pi(r_{1})\otimes r_{2},
\]
and also $(R,m_{R},u_{R},\Delta_{R},\varepsilon_{R},S_{R})\in\mathrm{Hopf}(^{H}_{H}\mathcal{YD})$ with $m_{R}=m_{A}|_{R\otimes R}$ and $1_{R}=1_{A}$, while
\[
\Delta_{R}:R\to R\otimes R,\ r\mapsto r_{1}\gamma(S(\pi(r_{2})))\otimes r_{3},\quad \varepsilon_{R}=\varepsilon_{A}|_{R}, \quad S_{R}:R\to R,\ r\mapsto\gamma\pi(r_{1})S(r_{2}).
\]
Moreover $\theta$ is an isomorphism of Hopf algebras, where $R\otimes H$ has the smash product and smash coproduct structure and it is denoted by $R\#H$ (see \cite{Ma2}).

Let us give explicitly the bosonization in the case of $E(n)$. By setting $G:=\{1, g \}$, we have a Hopf algebra with a projection $(E(n),\Bbbk G,\gamma,\pi)$ given by the Hopf algebra projection $\pi:E(n)\to\Bbbk G$, $g^{j}x_{P}\mapsto\delta_{P,\emptyset}g^{j}$ and the canonical inclusion $\gamma:\Bbbk G \to E(n)$. Then, $E(n)\cong E(n)^{\mathrm{co}\Bbbk G}\#\Bbbk G$, where $E(n)^{\mathrm{co}\Bbbk G}:=\{x\in E(n)\ |\ x_{1}\otimes\pi(x_{2})=x\otimes1\}$.
Given $g^{j}x_{P}\in E(n)$, we compute
\[
\begin{split}
(\mathrm{Id}\otimes\pi)\Delta(g^{j}x_{P})&=%(\mathrm{Id}\otimes\pi)\big(\sum_{F\subseteq P}{(-1)^{S(F,P)}g^{|F|+j}x_{P\setminus F}\otimes g^{j}x_{F}}\big)\\&=
\sum_{F\subseteq P}{(-1)^{S(F,P)}g^{|F|+j}x_{P\setminus F}\otimes\pi(g^{j}x_{F})}=(-1)^{S(\emptyset,P)}g^{j}x_{P}\otimes g^{j}=g^{j}x_{P}\otimes g^{j}
\end{split}
\]
and then $(\mathrm{Id}\otimes\pi)\Delta(g^{j}x_{P})=g^{j}x_{P}\otimes1$ if and only if $g^{j}x_{P}\otimes g^{j}=g^{j}x_{P}\otimes1$ if and only if $j=0$. Thus the vector space $R= E(n)^{\mathrm{co}\Bbbk G}$ of coinvariant elements is generated by the $x_{i}$'s. Moreover, it has the same algebra structure of $E(n)$ and then we can write it as $R=\Bbbk\langle x_{P}\ |\ P\subseteq\{1,...,n\}\rangle$. Note that $R$ is clearly not a coalgebra with the structure induced by $E(n)$ since $\Delta(x_{i})=x_{i}\otimes1+g\otimes x_{i}$ for all $i=1,...,n$. Indeed, $R$ has $\varepsilon_{R}=\varepsilon_{E(n)}|_{R}$ but a different coproduct and a different antipode. We make these structures explicit. First notice that
\[
\begin{split}
(\mathrm{Id}\otimes\Delta)\Delta(x_{P})&=(\mathrm{Id}\otimes\Delta)\big(\sum_{F\subseteq P}{(-1)^{S(F,P)}g^{|F|}x_{P\setminus F}\otimes x_{F}}\big)\\&=\sum_{F\subseteq P}{(-1)^{S(F,P)}\sum_{F'\subseteq F}{(-1)^{S(F',F)}g^{|F|}x_{P\setminus F}\otimes g^{|F'|}x_{F\setminus F'}\otimes x_{F'}}}.
\end{split}
\]
So the comultiplication is given by
\[
\begin{split}
    \Delta_{R}(x_{P})&=\sum_{F'\subseteq F\subseteq P}{(-1)^{S(F,P)}(-1)^{S(F',F)}(g^{|F|}x_{P\setminus F})\gamma(S(\pi(g^{|F'|}x_{F\setminus F'})))\otimes x_{F'}}\\&=\sum_{F\subseteq P}{(-1)^{S(F,P)}(g^{|F|}x_{P\setminus F})S(g^{|F|})\otimes x_{F}}=\sum_{F\subseteq P}{(-1)^{S(F,P)}g^{|F|}x_{P\setminus F}g^{|F|}\otimes x_{F}}\\&=\sum_{F\subseteq P}{(-1)^{S(F,P)}(-1)^{|F||P\setminus F|}g^{2|F|}x_{P\setminus F}\otimes x_{F}}=\sum_{F\subseteq P}{(-1)^{(S(F,P)+|F||P\setminus F|)}x_{P\setminus F}\otimes x_{F}}
\end{split}
\]
and the antipode by
\[
\begin{split}
S_{R}(x_{P})&=\sum_{F\subseteq P}{(-1)^{S(F,P)}\gamma\pi(g^{|F|}x_{P\setminus F})S(x_{F})}=g^{|P|}S(x_{P})=(-1)^{|P|}g^{|P|}g^{|P|}x_{P}=(-1)^{|P|}x_{P}.
\end{split}
\]
We know that $R=\Bbbk\langle x_{P}\ |\ P\subseteq\{1,...,n\}\rangle\in\mathrm{Hopf}(^{\Bbbk G}_{\Bbbk G}\mathcal{YD})$ and $E(n)\cong R\#\Bbbk G$, where $G:=\{1,g\}$. 

Now we can use this decomposition to determine a basis of the vector space $\mathrm{H}^{2}(E(n),\Bbbk)$ for the cohomology of coalgebras (as in Section \ref{sub:Hoch}).
This will lead to a description of $\mathrm{Z}^{2}(E(n),\Bbbk)$ that will help us in finding the infinitesimal $\Rr$-matrices of $E(n)$.

\subsubsection{2-cocycles for $E(n)$} Notice that the $R=\Bbbk\langle x_{P}\ |\ P\subseteq\{1,\ldots,n\}\rangle$ is the exterior algebra on an $n$-dimensional vector space, and it is a special case of the algebra $S$ given in \cite[Section 4]{MPSW}. %taking $\theta=n$, $N_i=2$ for all $1\leq i\leq n$ and $q_{ij}=-1$ for $1\leq i<j\leq n$.    
%Thus, our purpose is to study Hochschild 2-cocycles for a pointed Hopf algebra $A$ with a projection onto its coradical $\Bbbk G$, which is isomorphic to the bosonization $A^{\mathrm{co}H}\#\Bbbk G$ (see \cite{Bes97,BD98}), in case $A^{\mathrm{co}H}$ is described as the $S$ in \cite[Section 4]{MPSW} and $G$ is a finite group.
Let us recall its definition. Given a positive integer $\theta$ and an integer $N_i>1$, for each $1\leq i\leq \theta$, let $q_{ij}\in \Bbbk\setminus\{0\}$, and $S$ be the algebra generated by $x_1,\ldots, x_\theta$ subject to the
following relations:
\[x_ix_j = q_{ij}x_jx_i\, \quad\text{for all }\, i < j \quad\text{and}\quad x^{N_i}_i = 0 \quad\text{for all}\ i,\] 
where $q_{ji} = q^{-1}_{ij}$ for $i < j$. Thus, $R$ coincides with $S$ when $\theta=n$, $N_i=2$ for all $1\leq i\leq n$ and $q_{ij}=-1$ for $1\leq i<j\leq n$.

Given a $\Bbbk$-algebra $A$, we denote by $\mathrm{H}^{m}_{\mathrm{Hoch}}(A,\Bbbk)$ the $m$-th Hochschild cohomology group of the Hochschild cohomology of $A$ with coefficients in the trivial $A$-bimodule $\Bbbk$ (see e.g. \cite[Definition 1.1.13]{Wi}).
It is known that 
\begin{equation}\label{isoH2}
\mathrm{H}^{m}_{\mathrm{Hoch}}(S\#\Bbbk G, \Bbbk)\cong \mathrm{H}^{m}_{\mathrm{Hoch}}(S,\Bbbk)^{\Bbbk G},
\end{equation}
for all $m\geq1$, see e.g. \cite[Corollary 3.4]{St}, \cite[Corollary 9.6.6]{Wi}, where the $\Bbbk G$-action on $\mathrm{H}^{*}_{\mathrm{Hoch}}(S,\Bbbk)$ is explicitly given on generators $\xi_{i},\eta_{i}$ of the cohomology ring in the following way (see \cite[p. 393]{MPSW}): 
\[
g\cdot\xi_i:=f_{i}(g)^{-N_i}\xi_{i}, \qquad g\cdot\eta_{i}=f_{i}(g)^{-1}\eta_{i},
\]
with $f_{i}:\Bbbk G\to \Bbbk$ the character on $G$ for which $g\cdot x_{i}=f_{i}(g)x_{i}$, for each $1\leq i\leq\theta$. As it is observed in \cite[p. 392]{MPSW}, the $\eta_i$'s are identified with elements in $\mathrm{H}^1_{\mathrm{Hoch}}(S,\Bbbk)$, while the $\xi_i$'s with elements in $\mathrm{H}^2_{\mathrm{Hoch}}(S,\Bbbk)$ for all $i=1,\ldots,\theta$. Moreover,  $\mathrm{H}^m_{\mathrm{Hoch}}(S,\Bbbk)$ has dimension ${m+\theta -1 \choose \theta-1}$ (see \cite[p. 393]{MPSW}), so, in particular, $\mathrm{H}^{2}_{\mathrm{Hoch}}(S,\Bbbk)$ has dimension $\frac{\theta(\theta+1)}{2}$.  

%\begin{remark}
%\end{remark}

%From Example \ref{es:bosE(n)} 
%We know that $E(n)\cong R\#\Bbbk G$, where $R=\Bbbk\langle x_{P}\ |\ P\subseteq\{1,...,n\}\rangle\in\mathrm{Hopf}(^{\Bbbk\langle g\rangle}_{\Bbbk\langle g\rangle}\mathcal{YD})$ and $G=\Bbbk\langle g\rangle=\Bbbk\{1,g\}$. %Notice that $R$ is the exterior algebra on an $n$-dimensional vector space over $\Bbbk$.
Since
\[
\mathrm{H}^{m}_{\mathrm{Hoch}}(E(n),\Bbbk)\cong\mathrm{H}^{m}_{\mathrm{Hoch}}(R\#\Bbbk G, \Bbbk),
\]
we can apply \eqref{isoH2} with $S=R$, obtaining
%\cite[Corollary 4.3]{AAM} (see also \cite[Corollary 3.4]{St} and \cite[(2.4.1) page 1404]{MW}) and obtain
\begin{equation}\label{eq:iso-En}
\mathrm{H}^{m}_{\mathrm{Hoch}}(E(n),\Bbbk)\cong\mathrm{H}^{m}_{\mathrm{Hoch}}(R\#\Bbbk G, \Bbbk)\cong \mathrm{H}^{m}_{\mathrm{Hoch}}(R,\Bbbk)^{\Bbbk G},
\end{equation}
for all $m\geq1$. 

Let us consider $\mathrm{H}^{2}_{\mathrm{Hoch}}(R,\Bbbk)^{\Bbbk G}$. %In \cite[Theorem 4.1]{MPSW} the algebra $\mathrm{H}^{*}(R,\Bbbk)$ is represented in terms of generators and relations. 
%As it is observed in \cite[p. 392]{MPSW}, $\eta_i$ can be identified with elements in $\mathrm{H}^1(R,\Bbbk)$, while $\xi_i$ with elements in $\mathrm{H}^2(R,\Bbbk)$ for all $i=1,\ldots,n$. 
%Then, as it said in \cite[after Remark 4.2]{MPSW}, the $\Bbbk G$-action on $\mathrm{H}^{*}(R,\Bbbk)$ is defined through characters $f_i$, with $i=1,\ldots,n$, on the group $G$, which are defined such that %the following equality holds 
%the equality $\mu_R(g\otimes x_i)=gx_ig=f_i(g)x_{i}$ holds, 
In this case $g\cdot x_{i}=gx_ig=-x_{i}$,
so that $f_i(g)=-1$ for all $i=1,\ldots,n$. %Observe also that $R$ is a special case of the algebras $S$ given in \cite[Section 4]{MPSW}: take $\theta=n$, $N_i=2$ for all $1\leq i\leq n$ and $q_{ij}=-1$ for $1\leq i<j\leq n$. 
The action on $\mathrm{H}^{2}_{\mathrm{Hoch}}(R,\Bbbk)$ is explicitly given by: %(see \cite[p. 392]{MPSW})
\[
g\cdot\xi_i=f_{i}(g)^{-2}\xi_{i}=\xi_i=\varepsilon(g)\xi_i,
\]
so that every $\xi_i$ is invariant and $\mathrm{H}^{2}_{\mathrm{Hoch}}(R,\Bbbk)=\mathrm{H}^{2}_{\mathrm{Hoch}}(R,\Bbbk)^{\Bbbk G}$. Thus, since %$\mathrm{H}^m(R,\Bbbk)$ has dimension ${n+m-1 \choose n-1}$ (see \cite[p. 393]{MPSW}), so, in particular, 
$\mathrm{H}^{2}_{\mathrm{Hoch}}(R,\Bbbk)$ has dimension $\frac{n(n+1)}{2}$, we deduce by \eqref{eq:iso-En} that $\mathrm{H}^{2}_{\mathrm{Hoch}}(E(n),\Bbbk)$ has dimension $\frac{n(n+1)}{2}$. We observe that, since $E(n)$ is a finite-dimensional self-dual Hopf algebra, we obtain
\[
\mathrm{H}^{m}(E(n),\Bbbk)\cong\mathrm{H}^{m}_{\mathrm{Hoch}}(E(n)^{*},\Bbbk)^{*}\cong\mathrm{H}^{m}_{\mathrm{Hoch}}(E(n),\Bbbk)^{*}
\]
and then also $\mathrm{H}^{2}(E(n),\Bbbk)$ has dimension $\frac{n(n+1)}{2}$.

We now explicitly give a basis of $\mathrm{H}^{2}(E(n),\Bbbk)$. For all $i,l=1,\ldots, n$, let us denote by $\overline{gx_{i}\otimes x_{l}}$ the image of $gx_{i}\otimes x_{l} \in E(n) \otimes E(n)$ in the quotient $\mathrm{H}^2(E(n),\Bbbk)$.
\begin{lemma}\label{lem:cocyclenotcob}
     The elements $\overline{gx_{i}\otimes x_{l}}$, with $i=1, \ldots, n$ and $l\geq i$ are linearly independent in $\mathrm{H}^{2}(E(n),\Bbbk)$, for all $n\geq1$. In particular, for all $i,l=1,\ldots,n$, the element $gx_{i}\otimes x_{l}$ is a 2-cocycle but not a 2-coboundary.
\end{lemma}

\begin{proof}
Fix an arbitrary $n\geq1$. Let us suppose that $\Gamma:=\sum_{l \geq i=1}^n \gamma_{il}gx_{i}\otimes x_{l}\in\mathrm{B}^{2}(E(n))$, so that there exists $\alpha=\sum_{P,j=0,1}{\alpha^{j}_{P}g^{j}x_{P}}\in E(n)$ such that $b^{1}(\alpha)=1\otimes\alpha-\Delta(\alpha)+\alpha\otimes1=\Gamma$. We compute 
    \[
    \begin{split}
    b^{1}(\sum_{P,j=0,1}{\alpha^{j}_{P}g^{j}x_{P}})&=\sum_{P,j=0,1}{\alpha^{j}_{P}1\otimes g^{j}x_{P}}-\sum_{P,j=0,1}{\alpha^{j}_{P}\Delta(g^{j}x_{P})}+\sum_{P,j=0,1}{\alpha^{j}_{P}g^{j}x_{P}\otimes1}\\&=\sum_{P,j=0,1}{\alpha^{j}_{P}1\otimes g^{j}x_{P}}-\sum_{P,j=0,1}{\alpha^{j}_{P}}\sum_{F\subseteq P}{(-1)^{S(F,P)}g^{|F|+j}x_{P\setminus F}\otimes g^{j}x_{F}}\\&+\sum_{P,j=0,1}{\alpha^{j}_{P}g^{j}x_{P}\otimes1}.
    \end{split}
    \]
    When we force $b^{1}(\alpha)=\Gamma$, we immediately get by linear independence that $\alpha^{j}_{\emptyset}=0$. Notice that, for every $P \subseteq \lbrace 1, \ldots, n \rbrace$, the only term in $b^1(\alpha)$ proportional to $gx_P \otimes g$ is $-\alpha^1_P(gx_P \otimes g)$, while there is no term of this form appearing in $\Gamma$. This yields $\alpha^1_P=0$ for every $P$. Similarly, for every $P \subseteq \lbrace 1, \ldots, n \rbrace$ with $|P|$ odd, the only term in $b^1(\alpha)$ proportional to $1 \otimes x_P$ is $\alpha^0_P(1 \otimes x_P)$, while there is no term of this form appearing in $\Gamma$. This yields $\alpha^0_P=0$ for every $P$ with odd cardinality. As a consequence we have that
    \[
    b^1(\alpha)=\Gamma \iff -\sum_{|P|\textrm{ even}}{\alpha^0_{P}}\sum_{\emptyset \neq F\subsetneq P}{(-1)^{S(F,P)}g^{|F|}x_{P\setminus F}\otimes x_{F}}=\sum_{l \geq i=1}^n \gamma_{il}gx_{i}\otimes x_{l}.
    \]
    This implies that $\alpha^0_P=0$ for every $P$ which is not of the form $\lbrace i, l \rbrace$ with $i,l \in \lbrace 1, \ldots, n \rbrace$ and $l > i$. It follows that $\alpha=\sum_{i=1}^n\sum_{l>i}\alpha_{il}x_ix_l$ and that
    \[
    b^1(\alpha)=\Gamma \iff-\sum_{i=1}^n\sum_{l>i}\alpha_{il}[-gx_i\otimes x_l +gx_l \otimes x_i]=\sum_{l \geq i=1}^n \gamma_{il}gx_{i}\otimes x_{l}.
    \]
    Notice that for each pair $(i,l)$ with $l>i$ there is only one term on the LHS of the latter equality that is proportional to $x_l \otimes x_i$ and no term on the RHS that respects this property. This forces $\alpha_{il}=0$ for every $i$ and $l>i$, i.e. $\alpha=0$, and subsequently that $\Gamma=0$, that is, $\gamma_{il}=0$ for every $i=1, \ldots, n$ and $l\geq i$, by linear independence. This shows that the elements $\overline{gx_{i}\otimes x_{l}}$ with $i=1, \ldots, n$ and $l\geq i$ are linearly independent in $\mathrm{H}^{2}(E(n),\Bbbk)$ and, in particular, that each corresponding representative is not a $2$-coboundary. To conclude one can easily verify that $gx_{i}\otimes x_{l}\in \mathrm{Z}^{2}(E(n),\Bbbk)$ and $gx_{i}\otimes x_{l}-gx_{l}\otimes x_{i}=b^{1}(x_{i}x_{l})$, for all $i,l\in\{1,\ldots, n\}$. Since for every $i=1, \ldots, n$ and $l\geq i$ the elements $gx_i \otimes x_l$ are not $2$-coboundaries, we get that also the elements $gx_l \otimes x_i$ are not.
\end{proof}

\begin{remark}
Notice that, e.g. by \cite[Lemma 2.13]{Da}, the second cohomology group $\mathrm{H}^{2}(E(1), \Bbbk)$ of Sweedler’s Hopf algebra $E(1)$ is one-dimensional and it is generated by $\overline{gx_{1}\otimes x_{1}}$ (in loc. cit. the generator is $x_{1}\otimes gx_{1}$ because the given comultiplication is the opposite of the one we are considering on $x_{1}$). In particular, $gx_{1}\otimes x_{1}$ is not a 2-coboundary.
\end{remark}

Let us observe that, since the cardinality of the set $\mathfrak{B}= \lbrace \overline{gx_{i}\otimes x_{l}}\ |\ l\geq i \rbrace$ is $\frac{n(n+1)}{2}$, then $\mathfrak{B}$ is a basis of $\mathrm{H}^{2}(E(n),\Bbbk)$. If we define $\mathrm{I}(E(n),\Bbbk):=\langle gx_{i}\otimes x_{l}\ |\ l\geq i \rangle$, we obtain the following result:

\begin{theorem}\label{prop:2cocycle}
    For any $n\geq1$, we have
\[
\mathrm{Z}^{2}(E(n),\Bbbk)=\mathrm{B}^{2}(E(n),\Bbbk)\oplus \mathrm{I}(E(n),\Bbbk).
\]
\end{theorem}
%\mg{In Subsection \ref{sub:Hoch} abbiamo usato la notazione $Z^2(E(n))$, $B^2(E(n))$} \rd{[Cambiata in Subsection \ref{sub:Hoch} aggiungendo $\Bbbk$, visto che abbiamo usato sempre quella nella discussione sopra]}
\begin{remark}
Let us observe that $\mathrm{H}^{2}(E(n),\Bbbk)$, for all $n\geq1$, is substantially known in literature in case the base field $\Bbbk$ is $\mathbb{C}$. Indeed, recall that, given a finite-dimensional cocommutative Hopf superalgebra $\mathcal{A}$ over $\mathbb{C}$, we have $\mathcal{A}=\mathbb{C}[G]\ltimes\wedge W$, where $W$ is the (purely odd) space of primitive elements and $G$ is the group of group-like elements acting on $W$, hence on $\wedge W$, by conjugation, see \cite[Theorem 3.3]{Ko}. Moreover, in \cite[Lemma 3.2]{EG}, it is computed the Hochschild cohomology of the group algebra $\mathbb{C}[G\ltimes W]$ of the supergroup $G\ltimes W$, i.e., the Hopf superalgebra $\mathbb{C}[G]\ltimes\wedge W$. In \cite[p. 659 and Lemma 8.1]{BC} it is stated that $E(n)$ can be realized as a commutative Hopf superalgebra $(\mathbb{C}[G]\ltimes\wedge W)^{*}$ and that $S^{2}(W)^{G}\cong H^{2}_{L}(E(n))$, the \textit{lazy 2-cohomology group} in the sense of \cite[Definition 1.7]{BC}. Moreover, at the end of \cite[p. 637]{BC}, it is said that $H^{2}_{L}(E(n))\cong S^{2}(\mathbb{C}^{n})$, the group on $n\times n$ symmetric matrices with entries in $\mathbb{C}$. Hence, we obtain
\[
\mathrm{H}^{2}_{\mathrm{Hoch}}(E(n),\mathbb{C})\cong \mathrm{H}^{2}_{\mathrm{Hoch}}(E(n)^{*},\mathbb{C})\cong \mathrm{H}^{2}_{\mathrm{Hoch}}(\mathbb{C}[G\ltimes W],\mathbb{C})\overset{(!)}{=}S^{2}(W^{*})^{G}\cong S^{2}(W)^{G}\cong H^{2}_{L}(E(n))\cong S^{2}(\mathbb{C}^{n}),
\]
where $(!)$ follows using \cite[Lemma 3.2]{EG}. Thus, $\mathrm{H}^{2}_{\mathrm{Hoch}}(E(n),\mathbb{C})$ is isomorphic to the space of $n\times n$ symmetric matrices with entries in $\mathbb{C}$, so it has dimension $\frac{n(n+1)}{2}$. Therefore, also $\mathrm{H}^{2}(E(n),\mathbb{C})$ has the same dimension.
\end{remark}

\subsubsection{Infinitesimal $\Rr$-matrices for $E(n)$}

Since we have the classification of the $2$-cocycles of $E(n)$, we can impose further necessary conditions on $\chi$ in order to obtain the classification of infinitesimal $\Rr$-matrices of $E(n)$. %First recall the following result: 
%\begin{proposition}\label{prop:proj}\cite[Proposition 2.24]{ABSW}
%Let $(H, \mathcal{R}, \chi)$ be a pre-Cartier quasitriangular Hopf algebra endowed with a Hopf algebra projection $\pi : H \rightarrow L$ onto a commutative Hopf algebra $L$. Then, we have $(\pi \otimes \pi)(\chi) \in P (L) \otimes P (L)$, $(\mathrm{Id} \otimes \pi)(\chi) \in \mathcal{Z}(H) \otimes L$ and $(\pi \otimes\mathrm{Id})(\chi) \in L \otimes \mathcal{Z} (H)$.
%\end{proposition}
By Theorem \ref{prop:2cocycle} an arbitrary element $\chi\in \mathrm{Z}^{2}(E(n),\Bbbk)$ is given by
$$\chi=b^1(\xi)+\sum_{1\leq i\leq l\leq n}\alpha_{i,l}gx_i\otimes x_l,$$
where $\xi=\sum_{P,j=0,1}
\xi^j_P g^jx_P\in E(n)$. 
We compute $b^1(\xi)$:
\begin{equation}\label{eq:b1}
\begin{split}
b^{1}(\xi)&=
\sum_{P,j=0,1}\xi^j_P1\otimes g^jx_P+\sum_{P,j=0,1}\xi^j_Pg^jx_P\otimes 1-\sum_{P,j=0,1}\xi^j_P\sum_{F\subseteq P}(-1)^{S(F,P)} g^{|F|+j}x_{P\setminus F}\otimes g^jx_F\\&=
\sum_{|P|\ \text{even}}{\xi^{1}_{P}1\otimes gx_{P}}+\sum_{|P|\ \text{odd}}{\xi^{0}_{P}1\otimes x_{P}}+\sum_{P}{\xi^{1}_{P}gx_{P}\otimes 1}\\&-\sum_{\emptyset\not=P,j=0,1}\xi^j_P\sum_{\emptyset\not=F\subsetneq P}{(-1)^{S(F,P)}g^{|F|+j}x_{P\setminus F}\otimes g^{j}x_{F}}\\&-\sum_{|P|\ \text{even}}{\xi^{1}_{P}g\otimes gx_{P}}-\sum_{|P|\ \text{odd}}{\xi^{0}_{P}g\otimes x_{P}}-\sum_{P}{\xi^{1}_{P}gx_{P}\otimes g}.
\end{split}
\end{equation}
%\begin{invisible}
%By \cite[Proposition 2.24]{ABSW}, considering the Hopf algebra projection $\pi:E(n) \rightarrow \Bbbk \langle g \rangle$, $g^j x_P\mapsto\delta_{P,\emptyset} g^j$ and applying $\pi \otimes \pi$ on $\chi$, since $P(E(n))=0$ (cf. \cite[Proposition 1.16 (ii)]{BDG}),  we deduce that $\xi^0_\emptyset=0=\xi^1_\emptyset$.
%\end{invisible}

%$\alpha_{j,k}^{\emptyset, \emptyset}=0$ for $j,k=0,1$. 
By \cite[Proposition 2.24]{ABSW}, by considering the Hopf algebra projection $\pi:E(n) \rightarrow \Bbbk C_{2}$, $g^j x_P\mapsto\delta_{P,\emptyset} g^j$ and applying $\pi \otimes \pi$ on $b^{1}(\xi)$ as in \eqref{eq:b1}, since $P(\Bbbk C_2)=0$,  we deduce that $\xi^1_\emptyset=0$. Hence we can suppose $P\not=\emptyset$ in every term appearing in \eqref{eq:b1}. 

Let us set $\chi':=\sum_{1\leq i\leq l\leq n}\alpha_{i,l}gx_i\otimes x_l$ so that $\chi=b^{1}(\xi)+\chi'$.

\begin{remark}\label{rmk:chi'}
We observe that, since $\Delta(g)\chi'=\chi'\Delta(g)$, condition \eqref{cqtr1} on $g$ becomes equivalent to $(g\otimes g)b^{1}(\xi)=b^{1}(\xi)(g\otimes g)$. Analogously, since $\Delta(x_{i})\chi'=\chi'\Delta(x_{i})$, condition \eqref{cqtr1} on $x_{i}$ becomes equivalent to $(x_{i}\otimes 1+g\otimes x_{i})b^{1}(\xi)=b^{1}(\xi)(x_{i}\otimes 1+g\otimes x_{i})$.
\end{remark}
We observe the following facts.

\begin{lemma}\label{Centg}
Let $a\in E(n) \otimes E(n)$. Then, % $a \in C_{E(n)\otimes E(n)}(\Delta(g))$ (i.e., 
$a\Delta(g)=\Delta(g)a$ if, and only if,
\begin{equation}\label{eq:a}a=\sum_{\underset{|P|+|Q| \textrm{ even}}{j,k=0,1}} \alpha_{j,k}^{P,Q} g^jx_P \otimes g^kx_Q.
\end{equation}
\end{lemma}

\begin{proof}
Write $a=\sum_{\underset{P,Q}{j,k=0}}^1 \alpha_{j,k}^{P,Q} g^jx_P \otimes g^kx_Q\in E(n)\otimes E(n)$. The condition $a\Delta(g)=\Delta(g)a$ reads
\[\sum_{\underset{P,Q}{j,k=0}}^1 (-1)^{|P|+|Q|}\alpha_{j,k}^{P,Q} g^{j+1}x_P \otimes g^{k+1}x_Q=\sum_{\underset{P,Q}{j,k=0}}^1 \alpha_{j,k}^{P,Q} g^{j+1}x_P \otimes g^{k+1}x_Q,\]
and, since $\text{char}(\Bbbk)\not=2$, the latter is satisfied if, and only if, $\alpha_{j,k}^{P,Q}=0$ for any $P,Q$ such that $|P|+|Q|$ is odd.
\end{proof}

\begin{lemma}\label{centx}
Fix a generator  $x_i\in E(n)$, and suppose that $$a=\sum_{\underset{P,Q}{j,k=0}}^1 \alpha_{j,k}^{P,Q} g^jx_P \otimes g^kx_Q\in E(n)\otimes E(n)$$
is an element such that $a \Delta(x_i)=\Delta(x_i)a$. 

For every $P, Q\subseteq \{1,\ldots, n\}$, with $i\notin P\cup Q$ and $|P|, |Q|$ even, we have that $\alpha^{P,Q}_{0,1}=0$.
\end{lemma}

\begin{proof}
The condition $\Delta(x_i)a=a\Delta(x_i)$ reads
\[\sum_{\underset{P,Q}{j,k=0}}^1 \alpha_{j,k}^{P ,Q} [(-1)^{j+|P|}-1](g^jx_Px_i \otimes g^kx_Q)+\sum_{\underset{P,Q}{j,k=0}}^1 \alpha_{j,k}^{P,Q}[(-1)^{k+|Q|}-(-1)^{|P|}](g^{j+1}x_P \otimes g^kx_Qx_i)=0,\]
i.e.
\[\sum_{|P| \textrm{ odd},Q} \big[-\alpha_{0,0}^{P ,Q}(x_Px_i \otimes x_Q)-\alpha_{0,1}^{P ,Q}(x_Px_i \otimes gx_Q)\big]+\sum_{|P| \textrm{ even}, Q}\big[-\alpha_{1,0}^{P ,Q} (gx_Px_i \otimes x_Q)-\alpha_{1,1}^{P ,Q} (gx_Px_i \otimes gx_Q)\big]+\]
\[+\sum_{ |P|+|Q| \textrm{ odd}} (-1)^{|Q|}\big[\alpha_{0,0}^{P,Q}(gx_P \otimes x_Qx_i)+\alpha_{1,0}^{P,Q}(x_P \otimes x_Qx_i)\big]+\]
\[+\sum_{|P|+|Q| \textrm{ even}}(-1)^{|Q|+1}\big[\alpha_{0,1}^{P,Q}(gx_P \otimes gx_Qx_i)+\alpha_{1,1}^{P,Q}(x_P \otimes gx_Qx_i)\big]=0.\]

To better understand what is happening, let us rearrange everything in terms of two properties: the parity of each cardinality $|P|, |Q|$ and the inclusion $i \in P, i \in Q$. Recall that $x_Px_i=0$, whenever $i\in P$.
\[\sum_{\underset{|Q| \textrm{ odd}, Q \not \ni i}{|P| \textrm{ odd}, P \not \ni i}} \big[-\alpha_{0,0}^{P ,Q}(x_Px_i \otimes x_Q)-\alpha_{0,1}^{P ,Q}(x_Px_i \otimes gx_Q)\big]+\sum_{\underset{|Q| \textrm{ odd}, Q \ni i}{|P| \textrm{ odd}, P \not \ni i}}\big[ -\alpha_{0,0}^{P ,Q}(x_Px_i \otimes x_Q)-\alpha_{0,1}^{P ,Q}(x_Px_i \otimes gx_Q)\big]+\]
\[+\sum_{\underset{|Q| \textrm{ even}, Q \not \ni i}{|P| \textrm{ odd}, P \not \ni i}} \big[-\alpha_{0,0}^{P ,Q}(x_Px_i \otimes x_Q)-\alpha_{0,1}^{P ,Q}(x_Px_i \otimes gx_Q)\big]+\sum_{\underset{|Q| \textrm{ even}, Q \ni i}{|P| \textrm{ odd}, P \not \ni i}}\big[ -\alpha_{0,0}^{P ,Q}(x_Px_i \otimes x_Q)-\alpha_{0,1}^{P ,Q}(x_Px_i \otimes gx_Q)\big]+\]
\[+\sum_{\underset{|Q| \textrm{ odd}, Q \not \ni i}{|P| \textrm{ even}, P \not \ni i}}\big[-\alpha_{1,0}^{P ,Q} (gx_Px_i \otimes x_Q)-\alpha_{1,1}^{P ,Q} (gx_Px_i \otimes gx_Q)\big]+\sum_{\underset{|Q| \textrm{ odd}, Q \ni i}{|P| \textrm{ even}, P \not \ni i}}\big[-\alpha_{1,0}^{P ,Q} (gx_Px_i \otimes x_Q)-\alpha_{1,1}^{P ,Q} (gx_Px_i \otimes gx_Q)\big]+\]
\[+\sum_{\underset{|Q| \textrm{ even}, Q \not \ni i}{|P| \textrm{ even}, P \not \ni i}}\big[-\alpha_{1,0}^{P ,Q} (gx_Px_i \otimes x_Q)-\alpha_{1,1}^{P ,Q} (gx_Px_i \otimes gx_Q)\big]+\sum_{\underset{|Q| \textrm{ even}, Q \ni i}{|P| \textrm{ even}, P \not \ni i}}\big[-\alpha_{1,0}^{P ,Q} (gx_Px_i \otimes x_Q)-\alpha_{1,1}^{P ,Q} (gx_Px_i \otimes gx_Q)\big]+\]

\[+\sum_{\underset{|Q| \textrm{ even},Q \not \ni i}{|P| \textrm{\ odd}, P \not \ni i}} \big[\alpha_{0,0}^{P,Q}(gx_P \otimes x_Qx_i)+\alpha_{1,0}^{P,Q}(x_P \otimes x_Qx_i)\big]+\sum_{\underset{|Q| \textrm{ even},Q \not \ni i}{|P| \textrm{\ odd}, P \ni i}} \big[\alpha_{0,0}^{P,Q}(gx_P \otimes x_Qx_i)+\alpha_{1,0}^{P,Q}(x_P \otimes x_Qx_i)\big]+\]
\[+\sum_{\underset{|Q| \textrm{ odd},Q \not \ni i}{|P| \textrm{\ even}, P \not \ni i}} \big[-\alpha_{0,0}^{P,Q}(gx_P \otimes x_Qx_i)-\alpha_{1,0}^{P,Q}(x_P \otimes x_Qx_i)\big]+\sum_{\underset{|Q| \textrm{ odd},Q \not \ni i}{|P| \textrm{\ even}, P \ni i}} \big[-\alpha_{0,0}^{P,Q}(gx_P \otimes x_Qx_i)-\alpha_{1,0}^{P,Q}(x_P \otimes x_Qx_i)\big]+\]

\[+\sum_{\underset{|Q| \textrm{ even}, Q \not \ni i}{|P| \textrm{ even}, P \not \ni i}}\big[-\alpha_{0,1}^{P,Q}(gx_P \otimes gx_Qx_i)-\alpha_{1,1}^{P,Q}(x_P \otimes gx_Qx_i)\big]+\sum_{\underset{|Q| \textrm{ even}, Q \not \ni i}{|P| \textrm{ even}, P \ni i}}\big[-\alpha_{0,1}^{P,Q}(gx_P \otimes gx_Qx_i)-\alpha_{1,1}^{P,Q}(x_P \otimes gx_Qx_i)\big]+\]
\[+\sum_{\underset{|Q| \textrm{ odd}, Q \not \ni i}{|P| \textrm{ odd}, P \not \ni i}}\big[\alpha_{0,1}^{P,Q}(gx_P \otimes gx_Qx_i)+\alpha_{1,1}^{P,Q}(x_P \otimes gx_Qx_i)\big]+\sum_{\underset{|Q| \textrm{ odd}, Q \not \ni i}{|P| \textrm{ odd}, P \ni i}}\big[\alpha_{0,1}^{P,Q}(gx_P \otimes gx_Qx_i)+\alpha_{1,1}^{P,Q}(x_P \otimes gx_Qx_i)\big]=0.\]

Consider two subsets $P, Q\subseteq \{1,\ldots, n\}$, with $i\notin P\cup Q$ and $|P|, |Q|$ even. Looking at the term $gx_P\otimes gx_Qx_i$, by linear independence, we get $\alpha_{0,1}^{P,Q}=0$. 
% Looking at the terms $x_Px_i\otimes x_Q$ and $x_Px_i\otimes gx_Q$ respectively, by linear independence, we have that $\alpha_{0,0}^{P,Q}=0=\alpha_{0,1}^{P,Q}$, for all $P,Q$ with $|P|$ odd. Looking at the terms $gx_Px_i\otimes x_Q$ and $gx_Px_i\otimes gx_Q$ respectively, by linear independence, we get  that $\alpha_{1,0}^{P,Q}=0=\alpha_{1,1}^{P,Q}$, for all $P,Q$ with $|P|$ even. %and $P \cup Q \not \ni i$. 
% Similarly, %when $i\notin P \cup Q$, 
% looking at the term $x_P\otimes x_Qx_i$, by linear independence, we have  $\alpha_{1,0}^{P,Q}=0$ for $|P|$ odd and $|Q|$ even; looking at the term $gx_P\otimes x_Qx_i$, by linear independence, we have $\alpha_{0,0}^{P,Q}=0$ for $|P|$ even and $|Q|$ odd; looking at the term $gx_P\otimes gx_Qx_i$, by linear independence, we have $\alpha_{0,1}^{P,Q}=0$ for $|P|$ even and $|Q|$ even; looking at the term $x_P\otimes gx_Qx_i$, by linear independence, we have $\alpha_{1,1}^{P,Q}=0$ for $|P|$ odd and $|Q|$ odd.
\end{proof}

\begin{invisible}

\begin{lemma}\label{centx}
Suppose $a \in C_{E(n)\otimes E(n)}(\Delta(x_i))$, i.e. $a \Delta(x_i)=\Delta(x_i)a$.
If we write
\[a=\sum_{\underset{P,Q}{j,k=0}}^1 \alpha_{j,k}^{P,Q} g^jx_P \otimes g^kx_Q,\]
then the following equalities must hold:
\begin{center}
\def\arraystretch{1.5}
\begin{tabular}{| c | c | c |}
\hline
\textcolor{red}{$i \not \in P \cup Q$} & $|P|$ even  & $|P|$ odd \\
\hline 
$|Q|$ even & $\alpha^{P,Q}_{0,1}=\alpha^{P,Q}_{1,0}=\alpha^{P,Q}_{1,1}=0$ & $\alpha^{P,Q}_{0,0}=\alpha^{P,Q}_{0,1}=\alpha^{P,Q}_{1,0}=0$\\ 
&  $\alpha_{1,0}^{P ,Q\cup \lbrace i \rbrace} =(-1)^{s(P \cup Q  \cup \lbrace i \rbrace,i)}\alpha_{0,0}^{P \cup \lbrace i \rbrace,Q}$ & $\alpha_{0,1}^{P ,Q\cup \lbrace i \rbrace}=(-1)^{s(P\cup Q \cup \lbrace i \rbrace,i)+1}\alpha_{1,1}^{P \cup \lbrace i \rbrace ,Q}$\\
\hline
$|Q|$ odd & $\alpha^{P,Q}_{0,0}=\alpha^{P,Q}_{1,0}=\alpha^{P,Q}_{1,1}=0$ & $\alpha^{P,Q}_{0,0}=\alpha^{P,Q}_{0,1}=\alpha^{P,Q}_{1,1}=0$\\
& $\alpha_{1,1}^{P ,Q \cup \lbrace i \rbrace} =(-1)^{s(P \cup Q \cup  \lbrace i \rbrace,i)}\alpha_{0,1}^{P \cup \lbrace i \rbrace,Q}$ & $\alpha_{0,0}^{P ,Q\cup \lbrace i \rbrace}=(-1)^{s(P \cup Q \cup  \lbrace i \rbrace,i)+1}\alpha_{1,0}^{P \cup \lbrace i \rbrace,Q}$\\
\hline
\end{tabular}

\medskip
\begin{tabular}{| c | c | c |}
\hline
\textcolor{red}{$i \not \in P,\ i \in Q$} & $|P|$ even  & $|P|$ odd \\
\hline 
$|Q|$ even & $\alpha^{P,Q}_{1,0}=0$ & $\alpha^{P,Q}_{0,1}=0$\\
\hline
$|Q|$ odd & $\alpha^{P,Q}_{1,1}=0$ & $\alpha^{P,Q}_{0,0}=0$ \\
\hline
\end{tabular}
\hspace{0.5cm}
\begin{tabular}{| c | c | c |}
\hline
\textcolor{red}{$i \in P,\ i  \not \in Q$} & $|P|$ even  & $|P|$ odd \\
\hline 
$|Q|$ even & $\alpha^{P,Q}_{0,1}=0$ & $\alpha^{P,Q}_{1,0}=0$\\
\hline
$|Q|$ odd & $\alpha^{P,Q}_{0,0}=0$ & $\alpha^{P,Q}_{1,1}=0$ \\
\hline
\end{tabular}
\end{center}
where $(-1)^{s(P,i)}$ is the sign obtained when transforming $x_P$ into $x_{P\setminus \lbrace i \rbrace}x_i$, i.e. such that $x_P=(-1)^{s(P,i)}x_{P\setminus\lbrace i \rbrace}x_i$. 
\end{lemma}

\begin{proof}
Fix an $x_i$. Condition $\Delta(x_i)a=a\Delta(x_i)$ reads
\[\sum_{\underset{P \not \ni i,Q}{j,k=0}}^1 \alpha_{j,k}^{P ,Q} [(-1)^{j+|P|}-1](g^jx_Px_i \otimes g^kx_Q)+\sum_{\underset{P,Q \not \ni i}{j,k=0}}^1 \alpha_{j,k}^{P,Q}[(-1)^{k+|Q|}-(-1)^{|P|}](g^{j+1}x_P \otimes g^kx_Qx_i)=0,\]
i.e.
\[\sum_{|P| \textrm{ odd}, P \not \ni i,Q} \big[-\alpha_{0,0}^{P ,Q}(x_Px_i \otimes x_Q)-\alpha_{0,1}^{P ,Q}(x_Px_i \otimes gx_Q)\big]+\sum_{|P| \textrm{ even}, P \not \ni i,Q}\big[-\alpha_{1,0}^{P ,Q} (gx_Px_i \otimes x_Q)-\alpha_{1,1}^{P ,Q} (gx_Px_i \otimes gx_Q)\big]+\]
\[+\sum_{P,Q \not \ni i, \ |P|+|Q| \textrm{ odd}} (-1)^{|Q|}\big[\alpha_{0,0}^{P,Q}(gx_P \otimes x_Qx_i)+\alpha_{1,0}^{P,Q}(x_P \otimes x_Qx_i)\big]+\]
\[+\sum_{P,Q \not \ni i, \ |P|+|Q| \textrm{ even}}(-1)^{|Q|+1}\big[\alpha_{0,1}^{P,Q}(gx_P \otimes gx_Qx_i)+\alpha_{1,1}^{P,Q}(x_P \otimes gx_Qx_i)\big]=0.\]

To better understand what is happening, let us rearrange everything in terms of two properties: the parity of each cardinality $|P|, |Q|$ and the inclusion $i \in P, i \in Q$.

\[\sum_{\underset{|Q| \textrm{ odd}, Q \not \ni i}{|P| \textrm{ odd}, P \not \ni i}} \big[-\alpha_{0,0}^{P ,Q}(x_Px_i \otimes x_Q)-\alpha_{0,1}^{P ,Q}(x_Px_i \otimes gx_Q)\big]+\sum_{\underset{|Q| \textrm{ odd}, Q \ni i}{|P| \textrm{ odd}, P \not \ni i}}\big[ -\alpha_{0,0}^{P ,Q}(x_Px_i \otimes x_Q)-\alpha_{0,1}^{P ,Q}(x_Px_i \otimes gx_Q)\big]+\]
\[+\sum_{\underset{|Q| \textrm{ even}, Q \not \ni i}{|P| \textrm{ odd}, P \not \ni i}} \big[-\alpha_{0,0}^{P ,Q}(x_Px_i \otimes x_Q)-\alpha_{0,1}^{P ,Q}(x_Px_i \otimes gx_Q)\big]+\sum_{\underset{|Q| \textrm{ even}, Q \ni i}{|P| \textrm{ odd}, P \not \ni i}}\big[ -\alpha_{0,0}^{P ,Q}(x_Px_i \otimes x_Q)-\alpha_{0,1}^{P ,Q}(x_Px_i \otimes gx_Q)\big]+\]
\[+\sum_{\underset{|Q| \textrm{ odd}, Q \not \ni i}{|P| \textrm{ even}, P \not \ni i}}\big[-\alpha_{1,0}^{P ,Q} (gx_Px_i \otimes x_Q)-\alpha_{1,1}^{P ,Q} (gx_Px_i \otimes gx_Q)\big]+\sum_{\underset{|Q| \textrm{ odd}, Q \ni i}{|P| \textrm{ even}, P \not \ni i}}\big[-\alpha_{1,0}^{P ,Q} (gx_Px_i \otimes x_Q)-\alpha_{1,1}^{P ,Q} (gx_Px_i \otimes gx_Q)\big]+\]
\[+\sum_{\underset{|Q| \textrm{ even}, Q \not \ni i}{|P| \textrm{ even}, P \not \ni i}}\big[-\alpha_{1,0}^{P ,Q} (gx_Px_i \otimes x_Q)-\alpha_{1,1}^{P ,Q} (gx_Px_i \otimes gx_Q)\big]+\sum_{\underset{|Q| \textrm{ even}, Q \ni i}{|P| \textrm{ even}, P \not \ni i}}\big[-\alpha_{1,0}^{P ,Q} (gx_Px_i \otimes x_Q)-\alpha_{1,1}^{P ,Q} (gx_Px_i \otimes gx_Q)\big]+\]

\[+\sum_{\underset{|Q| \textrm{ even},Q \not \ni i}{|P| \textrm{\ odd}, P \not \ni i}} \big[\alpha_{0,0}^{P,Q}(gx_P \otimes x_Qx_i)+\alpha_{1,0}^{P,Q}(x_P \otimes x_Qx_i)\big]+\sum_{\underset{|Q| \textrm{ even},Q \not \ni i}{|P| \textrm{\ odd}, P \ni i}} \big[\alpha_{0,0}^{P,Q}(gx_P \otimes x_Qx_i)+\alpha_{1,0}^{P,Q}(x_P \otimes x_Qx_i)\big]+\]
\[+\sum_{\underset{|Q| \textrm{ odd},Q \not \ni i}{|P| \textrm{\ even}, P \not \ni i}} \big[-\alpha_{0,0}^{P,Q}(gx_P \otimes x_Qx_i)-\alpha_{1,0}^{P,Q}(x_P \otimes x_Qx_i)\big]+\sum_{\underset{|Q| \textrm{ odd},Q \not \ni i}{|P| \textrm{\ even}, P \ni i}} \big[-\alpha_{0,0}^{P,Q}(gx_P \otimes x_Qx_i)-\alpha_{1,0}^{P,Q}(x_P \otimes x_Qx_i)\big]+\]

\[+\sum_{\underset{|Q| \textrm{ even}, Q \not \ni i}{|P| \textrm{ even}, P \not \ni i}}\big[-\alpha_{0,1}^{P,Q}(gx_P \otimes gx_Qx_i)-\alpha_{1,1}^{P,Q}(x_P \otimes gx_Qx_i)\big]+\sum_{\underset{|Q| \textrm{ even}, Q \not \ni i}{|P| \textrm{ even}, P \ni i}}\big[-\alpha_{0,1}^{P,Q}(gx_P \otimes gx_Qx_i)-\alpha_{1,1}^{P,Q}(x_P \otimes gx_Qx_i)\big]+\]
\[+\sum_{\underset{|Q| \textrm{ odd}, Q \not \ni i}{|P| \textrm{ odd}, P \not \ni i}}\big[\alpha_{0,1}^{P,Q}(gx_P \otimes gx_Qx_i)+\alpha_{1,1}^{P,Q}(x_P \otimes gx_Qx_i)\big]+\sum_{\underset{|Q| \textrm{ odd}, Q \not \ni i}{|P| \textrm{ odd}, P \ni i}}\big[\alpha_{0,1}^{P,Q}(gx_P \otimes gx_Qx_i)+\alpha_{1,1}^{P,Q}(x_P \otimes gx_Qx_i)\big]=0.\]

By thorough inspection we see that $\alpha_{0,0}^{P,Q}=\alpha_{0,1}^{P,Q}=0$ for all $P,Q$ with $|P|$ odd and $P \cup Q \not \ni i$ and also that $\alpha_{1,0}^{P,Q}=\alpha_{1,1}^{P,Q}=0$ for all $P,Q$ with $|P|$ even and $P \cup Q \not \ni i$. Similarly, when $P \cup Q \not \ni i$, also $\alpha_{1,0}^{P,Q}=0$ for $|P|$ odd and $|Q|$ even, $\alpha_{0,0}^{P,Q}=0$ for $|P|$ even and $|Q|$ odd, $\alpha_{0,1}^{P,Q}=0$ for $|P|$ even and $|Q|$ even, $\alpha_{1,1}^{P,Q}=0$ for $|P|$ odd and $|Q|$ odd.

We are left with

\[\sum_{\underset{|Q| \textrm{ odd}, Q \ni i}{|P| \textrm{ odd}, P \not \ni i}} \big[-\alpha_{0,0}^{P ,Q}(x_Px_i \otimes x_Q)-\alpha_{0,1}^{P ,Q}(x_Px_i \otimes gx_Q)\big]+\sum_{\underset{|Q| \textrm{ even}, Q \ni i}{|P| \textrm{ odd}, P \not \ni i}}\big[ -\alpha_{0,0}^{P ,Q}(x_Px_i \otimes x_Q)-\alpha_{0,1}^{P ,Q}(x_Px_i \otimes gx_Q)\big]+\]
\[+\sum_{\underset{|Q| \textrm{ odd}, Q \ni i}{|P| \textrm{ even}, P \not \ni i}}\big[-\alpha_{1,0}^{P ,Q} (gx_Px_i \otimes x_Q)-\alpha_{1,1}^{P ,Q} (gx_Px_i \otimes gx_Q)\big]+\sum_{\underset{|Q| \textrm{ even}, Q \ni i}{|P| \textrm{ even}, P \not \ni i}}\big[-\alpha_{1,0}^{P ,Q} (gx_Px_i \otimes x_Q)-\alpha_{1,1}^{P ,Q} (gx_Px_i \otimes gx_Q)\big]+\]

\[+\sum_{\underset{|Q| \textrm{ even},Q \not \ni i}{|P| \textrm{\ odd}, P \ni i}}\big[ \alpha_{0,0}^{P,Q}(gx_P \otimes x_Qx_i)+\alpha_{1,0}^{P,Q}(x_P \otimes x_Qx_i)\big]+\sum_{\underset{|Q| \textrm{ odd},Q \not \ni i}{|P| \textrm{\ even}, P \ni i}}\big[ -\alpha_{0,0}^{P,Q}(gx_P \otimes x_Qx_i)-\alpha_{1,0}^{P,Q}(x_P \otimes x_Qx_i)\big]+\]
\[+\sum_{\underset{|Q| \textrm{ even}, Q \not \ni i}{|P| \textrm{ even}, P \ni i}}\big[-\alpha_{0,1}^{P,Q}(gx_P \otimes gx_Qx_i)-\alpha_{1,1}^{P,Q}(x_P \otimes gx_Qx_i)\big]+\sum_{\underset{|Q| \textrm{ odd}, Q \not \ni i}{|P| \textrm{ odd}, P \ni i}}\big[\alpha_{0,1}^{P,Q}(gx_P \otimes gx_Qx_i)+\alpha_{1,1}^{P,Q}(x_P \otimes gx_Qx_i)\big]=0.\]

Further inspection shows that when $P \not \ni i$, $Q \ni i$, we must have  $\alpha_{0,0}^{P,Q}=0$ for $|P|$ and $|Q|$ odd, $\alpha_{0,1}^{P,Q}=0$, when $|P|$ is odd and $|Q|$ is even, $\alpha_{1,1}^{P,Q}=0$, when $|P|$ is even and $|Q|$ is odd, $\alpha_{1,0}^{P,Q}=0$, when $|P|$ and $|Q|$ are even.
Also when $P \ni i$ and $Q \not \ni i$, we have $\alpha_{1,0}^{P,Q}=0$ when $|P|$ is odd and $|Q|$ is even, $\alpha_{0,0}^{P,Q}=0$ when $|P|$ is even and $|Q|$ is odd, $\alpha_{0,1}^{P,Q}=0$ when $|P|$ and $|Q|$ are even and $\alpha_{1,1}^{P,Q}=0$ when both $|P|$ and $|Q|$ are odd. By comparing the remaining terms, we find four equations:

\[\sum_{\underset{|Q| \textrm{ odd}, Q \ni i}{|P| \textrm{ odd}, P \not \ni i}}\alpha_{0,1}^{P ,Q}(x_Px_i \otimes gx_Q)+\sum_{\underset{|Q| \textrm{ even}, Q \not \ni i}{|P| \textrm{ even}, P \ni i}}\alpha_{1,1}^{P,Q}(x_P \otimes gx_Qx_i)=0\]
\[\sum_{\underset{|Q| \textrm{ even}, Q \ni i}{|P| \textrm{ odd}, P \not \ni i}} \alpha_{0,0}^{P ,Q}(x_Px_i \otimes x_Q)+\sum_{\underset{|Q| \textrm{ odd},Q \not \ni i}{|P| \textrm{\ even}, P \ni i}} \alpha_{1,0}^{P,Q}(x_P \otimes x_Qx_i)=0\]
\[\sum_{\underset{|Q| \textrm{ odd}, Q \ni i}{|P| \textrm{ even}, P \not \ni i}}\alpha_{1,0}^{P ,Q} (gx_Px_i \otimes x_Q)-\sum_{\underset{|Q| \textrm{ even},Q \not \ni i}{|P| \textrm{\ odd}, P \ni i}} \alpha_{0,0}^{P,Q}(gx_P \otimes x_Qx_i)=0\]
\[\sum_{\underset{|Q| \textrm{ even}, Q \ni i}{|P| \textrm{ even}, P \not \ni i}}\alpha_{1,1}^{P ,Q} (gx_Px_i \otimes gx_Q)-\sum_{\underset{|Q| \textrm{ odd}, Q \not \ni i}{|P| \textrm{ odd}, P \ni i}}\alpha_{0,1}^{P,Q}(gx_P \otimes gx_Qx_i)=0.\]

Let $j \in \lbrace 1, \ldots, n \rbrace$ and $j \in P \subseteq \lbrace 1, \ldots ,n \rbrace$. Define $s(P,j)$ to be the number of consecutive swaps to transform $x_P$ in the form $x_{P\setminus\lbrace j \rbrace}x_j$. It is clear that $x_P=(-1)^{s(P,j)}x_{P\setminus\lbrace j \rbrace}x_j$. %and $(-1)^{s(P,j)}(-1)^{s(Q,j)}=(-1)^{s(P \cup Q,j)}$. 
Then, the previous four equations become

\[\sum_{\underset{|Q| \textrm{ even}, Q \not \ni i}{|P| \textrm{ odd}, P \not \ni i}}[(-1)^{s(Q\cup \lbrace i \rbrace,i)}\alpha_{0,1}^{P ,Q\cup \lbrace i \rbrace}+(-1)^{s(P\cup \lbrace i \rbrace,i)}\alpha_{1,1}^{P \cup \lbrace i \rbrace ,Q}](x_Px_i \otimes gx_Qx_i)=0\]
\[\sum_{\underset{|Q| \textrm{ odd}, Q \not \ni i}{|P| \textrm{ odd}, P \not \ni i}} [(-1)^{s(Q\cup \lbrace i \rbrace,i)}\alpha_{0,0}^{P ,Q\cup \lbrace i \rbrace}+(-1)^{s(P\cup \lbrace i \rbrace,i)}\alpha_{1,0}^{P \cup \lbrace i \rbrace,Q}](x_Px_i \otimes x_Qx_i)=0\]
\[\sum_{\underset{|Q| \textrm{ even}, Q \not \ni i}{|P| \textrm{ even}, P \not \ni i}} [(-1)^{s(Q \cup \lbrace i \rbrace ,i)}\alpha_{1,0}^{P ,Q\cup \lbrace i \rbrace} -(-1)^{s(P \cup \lbrace i \rbrace,i)}\alpha_{0,0}^{P \cup \lbrace i \rbrace,Q}](gx_Px_i \otimes x_Qx_i)=0\]
\[\sum_{\underset{|Q| \textrm{ odd}, Q \not \ni i}{|P| \textrm{ even}, P \not \ni i}}[(-1)^{s(Q \cup \lbrace i \rbrace, i)}\alpha_{1,1}^{P ,Q \cup \lbrace i \rbrace} -(-1)^{s(P \cup \lbrace i \rbrace,i)}\alpha_{0,1}^{P \cup \lbrace i \rbrace,Q}](gx_Px_i \otimes gx_Qx_i)=0.\qedhere\]
\end{proof}
\end{invisible}

\begin{corollary} 
    Let $b^1(\xi)$ as in \eqref{eq:b1}, with $\xi^{1}_{\emptyset}=0$. Then, $b^1(\xi)$ satisfies \eqref{cqtr1} if, and only if, \begin{equation}\label{eq:b1xii}
b^1(\xi)=-\sum_{0\neq |P|\ \text{even}}\xi^0_P\sum_{\emptyset\not=F\subsetneq P}{(-1)^{S(F,P)}g^{|F|}x_{P\setminus F}\otimes x_{F}}.
\end{equation}
As a consequence, the element
\begin{equation}\label{eq:chiax1}
\chi=-\sum_{0\neq |P|\ \text{even}}\xi^0_P\sum_{\emptyset\not=F\subsetneq P}{(-1)^{S(F,P)}g^{|F|}x_{P\setminus F}\otimes x_{F}}+\chi',
\end{equation}
for $\xi^0_P\in\Bbbk$, satisfies \eqref{cqtr1}.
\end{corollary}
\begin{proof}
By Lemma \ref{Centg}, we have that $b^{1}(\xi)$ as in \eqref{eq:b1}, with $\xi^{1}_{\emptyset}=0$, satisfies \eqref{cqtr1} on $g$ if and only if $\xi^{j}_{P}=0$ with $|P|$ odd, i.e. 
\begin{equation}\label{eq:b1xi}
\begin{split}
b^1(\xi)&=
\sum_{0\neq |P|\ \text{even}}{\xi^{1}_{P}1\otimes gx_{P}}+\sum_{0\neq |P|\ \text{even}}{\xi^{1}_{P}gx_{P}\otimes 1}\\&-\sum_{\underset{0\neq |P|\ \text{even}}{j=0,1}}\xi^j_P\sum_{\emptyset\not=F\subsetneq P}{(-1)^{S(F,P)}g^{|F|+j}x_{P\setminus F}\otimes g^{j}x_{F}}\\&-\sum_{0\neq |P|\ \text{even}}{\xi^{1}_{P}g\otimes gx_{P}}-\sum_{0\neq |P|\ \text{even}}{\xi^{1}_{P}gx_{P}\otimes g}.
\end{split}
\end{equation}

We observe that $b^{1}(\xi)$ as in \eqref{eq:b1xi} is a summation of elements $\sum_{\underset{P,Q}{j,k=0}}^1 \alpha_{j,k}^{P,Q} g^jx_P \otimes g^kx_Q$ such that $0\not=|P\cup Q|$ is even. Let us consider two different cases, namely $n$ odd and $n$ even. If $n$ is odd, for every $P$ with $0\not=|P|$ even there exists an $i\notin P$. If we consider the first addendum in \eqref{eq:b1xi}, we have $\xi^1_P=\alpha^{\emptyset, P}_{0,1}$ with $|P|$ even, and thus we get $\xi^1_P=0$ by Lemma \ref{centx}. %\alpha^{\emptyset\ P}_{01}=\alpha^{\emptyset \ P}_{10}=\alpha^{\emptyset\  P}_{11}
Then, we obtain \eqref{eq:b1xii}.
% i.e.
% \[
% b^1(\xi)=-\sum_{0\neq |P|\ \text{even}}\xi^0_P\sum_{\emptyset\not=F\subsetneq P}{(-1)^{S(F,P)}g^{|F|}x_{P\setminus F}\otimes x_{F}}.
% \]

Suppose that $n$ is even. For every $0\neq |P|<n$ even, we can apply the same arguments as in the previous case. Thus, denoting $\textbf{n}:=\lbrace 1, \ldots, n \rbrace$, we get
\[
\begin{split}
b^1(\xi)&=\xi^1_{\textbf{n}}1\otimes gx_{{\textbf{n}}}+\xi^1_{\textbf{n}}gx_{\textbf{n}}\otimes 1- {\xi^{1}_{\textbf{n}}g\otimes gx_{\textbf{n}}}-\xi^{1}_{\textbf{n}}gx_{\textbf{n}}\otimes g\\&-\sum_{j=0,1}\xi^j_\textbf{n}\sum_{\emptyset\not=F\subsetneq \textbf{n}}{(-1)^{S(F,\textbf{n})}g^{|F|+j}x_{\textbf{n}\setminus F}\otimes g^{j}x_{F}}\\&-\sum_{0\neq |P|<n\ \text{even}}\xi^0_P\sum_{\emptyset\not=F\subsetneq P}{(-1)^{S(F,P)}g^{|F|}x_{P\setminus F}\otimes x_{F}}.
\end{split}
\]
In order to satisfy \eqref{cqtr1} for any $x_{i}$ (clearly $i\in\textbf{n}$), by linear independence, we must have $\xi^{1}_{\textbf{n}}=0$ (look at the second or third addendum), so  we get again \eqref{eq:b1xii}.
One can check that \eqref{cqtr1} is satisfied by $b^{1}(\xi)$ of the form \eqref{eq:b1xii}. By Remark \ref{rmk:chi'}, it follows that $\chi$ as in \eqref{eq:chiax1} satifies \eqref{cqtr1}.
\end{proof}

\begin{invisible}OLD VERSION
\begin{remark} \rd{Modicare in corollario+proof.}
We have that $b^{1}(\xi)$ as in \eqref{eq:b1}, with $\xi^{1}_{\emptyset}=0$, satisfies \eqref{cqtr1} on $g$ if and only if $\xi^{j}_{P}=0$ with $|P|$ odd, i.e. 
\begin{equation}%\label{eq:b1xi}
\begin{split}
b^1(\xi)&=
\sum_{0\neq |P|\ \text{even}}{\xi^{1}_{P}1\otimes gx_{P}}+\sum_{0\neq |P|\ \text{even}}{\xi^{1}_{P}gx_{P}\otimes 1}\\&-\sum_{\underset{0\neq |P|\ \text{even}}{j=0,1}}\xi^j_P\sum_{\emptyset\not=F\subsetneq P}{(-1)^{S(F,P)}g^{|F|+j}x_{P\setminus F}\otimes g^{j}x_{F}}\\&-\sum_{0\neq |P|\ \text{even}}{\xi^{1}_{P}g\otimes gx_{P}}-\sum_{0\neq |P|\ \text{even}}{\xi^{1}_{P}gx_{P}\otimes g}.
\end{split}
\end{equation}
\end{remark}
%Moreover, by applying \eqref{eq:eps-chi2} we obtain $\alpha^{\emptyset, Q}_{j,k}=0$ when $0\neq|Q|$ even. Thus, $b^1(\xi)=\sum_{\underset{\emptyset\neq P, \emptyset\neq Q, |P|+|Q| \textrm{ even}}{j,k=0,1}} \alpha_{j,k}^{P,Q} g^jx_P \otimes g^kx_Q$.

We observe that $b^{1}(\xi)$ as in \eqref{eq:b1xi} is a summation of elements $\sum_{\underset{P,Q}{j,k=0}}^1 \alpha_{j,k}^{P,Q} g^jx_P \otimes g^kx_Q$ such that $0\not=|P\cup Q|$ is even. Let us consider two different cases, namely $n$ odd and $n$ even. If $n$ is odd, for every $P$ with $0\not=|P|$ even there exists an $i\notin P$. 
By Lemma \ref{centx}, \rd{\sout{we look at the first table where $|P|$ is even and $|Q|$ is even;}} if we consider the first addendum in \eqref{eq:b1xi}, we have that 
$\xi^1_P=\alpha^{\emptyset, P}_{0,1}$ with $|P|$ even, and thus we get $\xi^1_P=0$. %\alpha^{\emptyset\ P}_{01}=\alpha^{\emptyset \ P}_{10}=\alpha^{\emptyset\  P}_{11}
Then, we obtain
\begin{equation}%\label{eq:b1xii}
b^1(\xi)=-\sum_{0\neq |P|\ \text{even}}\xi^0_P\sum_{\emptyset\not=F\subsetneq P}{(-1)^{S(F,P)}g^{|F|}x_{P\setminus F}\otimes x_{F}}.
\end{equation}

Suppose that $n$ is even. For every $0\neq |P|<n$ even, we can apply the same arguments as in the previous case. Thus, denoting $\textbf{n}:=\lbrace 1, \ldots, n \rbrace$, we get
\[
\begin{split}
b^1(\xi)&=\xi^1_{\textbf{n}}1\otimes gx_{{\textbf{n}}}+\xi^1_{\textbf{n}}gx_{\textbf{n}}\otimes 1- {\xi^{1}_{\textbf{n}}g\otimes gx_{\textbf{n}}}-\xi^{1}_{\textbf{n}}gx_{\textbf{n}}\otimes g\\&-\sum_{j=0,1}\xi^j_\textbf{n}\sum_{\emptyset\not=F\subsetneq \textbf{n}}{(-1)^{S(F,\textbf{n})}g^{|F|+j}x_{\textbf{n}\setminus F}\otimes g^{j}x_{F}}\\&-\sum_{0\neq |P|<n\ \text{even}}\xi^0_P\sum_{\emptyset\not=F\subsetneq P}{(-1)^{S(F,P)}g^{|F|}x_{P\setminus F}\otimes x_{F}}.
\end{split}
\]
In order to satisfy \eqref{cqtr1} for any $x_{i}$ (clearly $i\in\textbf{n}$), by linear independence, we must have $\xi^{1}_{\textbf{n}}=0$ (look at the second or third addendum) so  we get again \eqref{eq:b1xii}.

By observing that \eqref{cqtr1} is satisfied by $b^{1}(\xi)$ of the form \eqref{eq:b1xii} we obtain the following result.

\begin{corollary}%label{cor:ax1}
    The element
\begin{equation}\label{eq:chiax1}
\chi=-\sum_{0\neq |P|\ \text{even}}\xi^0_P\sum_{\emptyset\not=F\subsetneq P}{(-1)^{S(F,P)}g^{|F|}x_{P\setminus F}\otimes x_{F}}+\sum_{1\leq i\leq l\leq n}\alpha_{i,l}gx_i\otimes x_l,
\end{equation}
for $\xi^0_P, \alpha_{i,l}\in\Bbbk$ satisfies \eqref{cqtr1}. 
\end{corollary}
\end{invisible}
   Hereafter, we rewrite $\chi$ as $\chi=\tilde{\chi}+\hat{\chi}$, where 
\begin{align}\label{eq:tilde}\tilde{\chi}&=-\sum_{ |P|\geq 4 \ \text{even}}\xi^0_P\sum_{\emptyset\not=F\subsetneq P}{(-1)^{S(F,P)}g^{|F|}x_{P\setminus F}\otimes x_{F}},\\ 
\hat{\chi}&=\sum_{p,q=1}^n\gamma_{pq}gx_p\otimes x_q,\label{eq:hat}\end{align}
    with $\xi^0_P,\gamma_{pq}\in\Bbbk$.

Now we look at axiom \eqref{cqtr2}. First, we show a preliminary result.

\begin{lemma}\label{Rswap}
For any quasitriangular structure $\Rr$ of $E(n)$, the following equalities hold
\begin{eqnarray}
\mathcal{R}(g\otimes g)& =&(g\otimes g)\Rr\label{Rswap0}\\
\mathcal{R}(x_p \otimes 1)& =& (x_p \otimes g)\mathcal{R} \label{Rswap1}\\
\mathcal{R}(g \otimes x_q) &=& (1 \otimes x_q)\mathcal{R}, \label{Rswap2}
\end{eqnarray}
for $p,q=1,\ldots, n$.
\end{lemma}

\begin{proof}
Let $\Rr$ be as in \eqref{RstructureE(n)}. First \eqref{Rswap0} follows trivially. Moreover,  \eqref{Rswap1} follows as
\begin{eqnarray*}
2[\mathcal{R}(x_p \otimes 1) - (x_p \otimes g)\mathcal{R}]&=&(x_p \otimes 1 + x_p \otimes g + gx_p \otimes 1-gx_p \otimes g)+\sum_{|F|=|P|}(-1)^{\frac{|P|(|P|-1)}{2}}\det(P,F) \times\\ 
&\times& (g^{|P|}x_Fx_p \otimes x_P + g^{|P|}x_Fx_p \otimes gx_P +g^{|P|+1}x_Fx_p \otimes x_P -  g^{|P|+1}x_Fx_p \otimes gx_P)+\\
&-&(x_p \otimes g + x_p \otimes 1 +x_pg \otimes g-x_pg \otimes 1)-[\sum_{|F|=|P|}(-1)^{\frac{|P|(|P|-1)}{2}}\det(P,F) \times\\ 
&\times& (x_pg^{|P|}x_F \otimes gx_P + x_pg^{|P|}x_F \otimes x_P +x_pg^{|P|+1}x_F     \otimes gx_P -  x_pg^{|P|+1}x_F \otimes x_P)]\\
&=&(x_p \otimes 1 + x_p \otimes g + gx_p \otimes 1-gx_p \otimes g)+\sum_{|F|=|P|}(-1)^{\frac{|P|(|P|-1)}{2}}\det(P,F) \times\\ 
&\times& (g^{|P|}x_Fx_p \otimes x_P + g^{|P|}x_Fx_p \otimes gx_P +g^{|P|+1}x_Fx_p \otimes x_P -  g^{|P|+1}x_Fx_p \otimes gx_P)+\\
&-&(x_p \otimes g + x_p \otimes 1 -gx_p \otimes g+gx_p \otimes 1)-[\sum_{|F|=|P|}(-1)^{\frac{|P|(|P|-1)}{2}}\det(P,F) \times\\ 
&\times& (g^{|P|}x_Fx_p \otimes gx_P + g^{|P|}x_Fx_p \otimes x_P -g^{|P|+1}x_Fx_p \otimes gx_P +g^{|P|+1}x_Fx_p \otimes x_P)]\\
&=&0.
\end{eqnarray*}

\begin{invisible}
\begin{eqnarray*}
2[\mathcal{R}(g \otimes x_q) - (1_H \otimes x_q)\mathcal{R}]&=&(g \otimes x_q + g \otimes gx_q + 1_H \otimes x_q-1_H \otimes gx_q)+[\sum_{|F|=|P|}(-1)^{\frac{|P|(|P|-1)}{2}}\det(P,F) \times \\ 
&\times& (g^{|P|}x_Fg \otimes x_Px_q + g^{|P|}x_Fg \otimes gx_Px_q +g^{|P|+1}x_Fg     \otimes x_Px_q -  g^{|P|+1}x_Fg \otimes gx_Px_q)]+\\
&-&(1 \otimes x_q - 1 \otimes gx_q + g \otimes x_q+g \otimes gx_q)-[\sum_{|F|=|P|}(-1)^{\frac{|P|(|P|-1)}{2}}\det(P,F) \times\\ 
&\times & (g^{|P|}x_F \otimes x_qx_P + g^{|P|}x_F \otimes x_qgx_P +g^{|P|+1}x_F \otimes x_qx_P -  g^{|P|+1}x_F \otimes x_qgx_P)]\\
&=&\sum_{|F|=|P|}(-1)^{\frac{|P|(|P|-1)}{2}}\det(P,F)((-1)^{|F|}g^{|P|+1}x_F \otimes x_Px_q +(-1)^{|F|} g^{|P|+1}x_F \otimes gx_Px_q+\\
&+&(-1)^{|F|}g^{|P|}x_F \otimes x_Px_q +(-1)^{|F|+1}g^{|P|}x_F \otimes gx_Px_q+(-1)^{|P|+1}g^{|P|}x_F \otimes x_Px_q+\\
&+&(-1)^{|P|} g^{|P|}x_F \otimes gx_Px_q-(-1)^{|P|}g^{|P|+1}x_F \otimes x_Px_q +(-1)^{|P|+1} g^{|P|+1}x_F \otimes gx_Px_q)\\
&=&0\\
\end{eqnarray*}
\end{invisible}
Notice that (\ref{Rswap2}) can be immediately deduced from (\ref{Rswap1}) and (\ref{qtr1}).
In fact
\[\mathcal{R}\Delta(x_q)\mathcal{R}^{-1}=\Delta^{\mathrm{op}}(x_q) \iff \mathcal{R}(x_q \otimes 1 +g \otimes x_q)\mathcal{R}^{-1}=x_q \otimes g + 1 \otimes x_q \overset{(\ref{Rswap1})}{\iff}  \mathcal{R}(g \otimes x_q)\mathcal{R}^{-1}=1 \otimes x_q.\]
\end{proof}

\begin{lemma}\label{lemma3}
  The element $\hat{\chi}$ as in \eqref{eq:hat} satisfies \eqref{cqtr2} and \eqref{cqtr3}.
\end{lemma}

\begin{proof}
Let us write $\Rr^{-1}=\sum \overline{\Rr}^i \otimes \overline{\Rr}_i$. An element $\hat{\chi}=\sum_{p,q=1}^n \gamma_{pq}gx_p \otimes x_q$ satisfies \eqref{cqtr2} if and only if
\[
\sum_{p,q=1}^n \gamma_{pq}[gx_p \otimes g \otimes x_q+ gx_p \otimes x_q\otimes 1]=\sum_{p,q=1}^n \gamma_{pq}  gx_p \otimes x_q\otimes 1+\sum_{p,q=1}^n \gamma_{pq} \overline{\mathcal{R}}^{i}gx_p\mathcal{R}^j \otimes\overline{\mathcal{R}}_i\mathcal{R}_j\otimes x_q,
\]
i.e., by linear independence, if and only if
\[
\sum_{p=1}^n \gamma_{pq}[gx_p \otimes g-\overline{\mathcal{R}}^{i}gx_p\mathcal{R}^j \otimes\overline{\mathcal{R}}_{i}\mathcal{R}_j]=0,\ \text{for each}\ q=1, \ldots, n,
\]
which is equivalent, by multiplying by $\Rr$ on the left, to 
\[
\sum_{p=1}^n \gamma_{pq} [\mathcal{R}(gx_{p} \otimes g) - (gx_{p} \otimes 1)\mathcal{R}]=0,\ \text{for each}\ q=1, \ldots, n.
\]
Thanks to \eqref{Rswap0} and \eqref{Rswap1} of Lemma~\ref{Rswap} we can conclude that \eqref{cqtr2} is always satisfied. Similarly, one can show that $\hat{\chi}$ satisfies \eqref{cqtr3}, using \eqref{Rswap2} of Lemma~\ref{Rswap}. 
\begin{invisible}
    An element $\hat{\chi}=\sum_{p,q=1}^n \gamma_{pq}gx_p \otimes x_q$ satisfies \eqref{cqtr3} if and only if
\[
\sum_{p,q=1}^n \gamma_{pq}[gx_p \otimes g \otimes x_q+1_H \otimes gx_p \otimes x_q]=\sum_{p,q=1}^n \gamma_{pq} 1_H \otimes gx_p \otimes x_q+\sum_{p,q=1}^n \gamma_{pq} gx_p \otimes\overline{\mathcal{R}}^{i}\mathcal{R}^j \otimes\overline{\mathcal{R}}_{i}x_q\mathcal{R}_j,
\]
i.e., by linear independence, if and only if
\[
\sum_{q=1}^n \gamma_{pq}[g \otimes x_q-\overline{\mathcal{R}}^{i}\mathcal{R}^j \otimes\overline{\mathcal{R}}_{i}x_q\mathcal{R}_j]=0,\ \text{for each}\ p=1, \ldots, n,
\]
which is equivalent, by multiplying by $\Rr$ on the left, to 
\[
\sum_{q=1}^n \gamma_{pq} [\mathcal{R}(g \otimes x_q) - (1_H \otimes x_q)\mathcal{R}]=0,\ \text{for each}\ p=1, \ldots, n.
\]
Thanks to \eqref{Rswap2} of Lemma~\ref{Rswap} we can conclude that \eqref{cqtr3} is always satisfied.
\end{invisible}\qedhere
\end{proof}

We also have the following result.

\begin{lemma}
    The element $\tilde{\chi}$ defined as in \eqref{eq:tilde}  satisfies \eqref{cqtr2} (and \eqref{cqtr3}) if and only if it is zero.
\end{lemma}
\begin{proof}
We compute
\[
\begin{split}(\mathrm{Id}\otimes\Delta )(\tilde{\chi})&=
-\sum_{|P|\geq 4\  \text{even}}\xi^0_P\sum_{\emptyset \neq F\subsetneq P}(-1)^{S(F,P)}g^{|F|}x_{P\setminus F}\otimes \sum_{F'\subseteq F}(-1)^{S(F',F)}g^{|F'|}x_{F\setminus F'}\otimes x_{F'}.
\end{split}
\]
If we consider $F'=\emptyset$ in the previous term we obtain $\tilde{\chi}_{12}$. We observe that $\Rr^{-1}_{12}\tilde{\chi}_{13}\Rr_{12}$ can be compared just with elements in the previous summation such that $F'=F$. Since $|P|\geq 4$ we are able to find an $F'$ such that $\emptyset\not=F'\subsetneq F\subsetneq P$, hence by linear independence $\xi^{0}_{P}=0$ for all $|P|\geq 4$, so $\tilde{\chi}=0$. 
\end{proof}

Since, by Lemma \ref{lemma3}, $\hat{\chi}$ satisfies \eqref{cqtr2} and \eqref{cqtr3}, if $\chi$ satisfies \eqref{cqtr2} and \eqref{cqtr3} we must have $\chi=\hat{\chi}$. Thus, $\chi=\sum_{p,q=1}^{n}{\gamma_{pq}gx_{p}\otimes x_{q}}$ makes $(E(n),R_{A})$ a pre-Cartier quasitriangular Hopf algebra.

Finally, we study axiom \eqref{eq:ctr2}.

\begin{lemma}\label{lemmaCartier}
    An element $\chi=\sum_{p,q=1}^{n}{\gamma_{pq}gx_{p}\otimes x_{q}}$ satisfies \eqref{eq:ctr2} %\rd{\sout{, i.e., $\Rr\chi=\chi^{\mathrm{op}}\Rr$}} 
    if and only if $\gamma_{pq}=-\gamma_{qp}$ for any $p,q=1,\ldots,n$.
\end{lemma}

\begin{proof}
By Proposition \ref{cor:SSRchi}, $\chi$ satisfies \eqref{eq:ctr2} if and only if \eqref{eq:HopfCartier} holds.
We compute
\[
\tau(S \otimes S)(\chi)=-\sum_{p,q=1}^{n}{\gamma_{pq}gx_{q}\otimes x_{p}}.
\]
\begin{comment}
\begin{split}

\mathcal{R}\chi&=\sum_{p,q=1}^{n}\gamma_{pq}\mathcal{R}(gx_p\otimes x_q)\overset{\eqref{Rswap2}}{=}\sum_{p,q=1}^{n}\gamma_{pq}(1_H \otimes x_q)\mathcal{R}(x_p\otimes 1)\overset{\eqref{Rswap1}}{=}-\sum_{p,q=1}^n \gamma_{pq}(x_p \otimes gx_q)\mathcal{R}\\&=(S\otimes S)(\chi)\Rr.

%Indeed, $S(gx_p)=g^2x_p=x_p$ and $S(x_q)=-gx_q$.
\end{split}
\end{comment}
We obtain that \eqref{eq:HopfCartier} is equivalent to
\[
-\sum_{p,q=1}^n \gamma_{pq}gx_q\otimes x_{p}=\sum_{p,q=1}^n \gamma_{qp}gx_q\otimes x_{p},
\]
%\rd{\sout{
%which is equivalent to $\sum_{p,q=1}^n \gamma_{pq}(x_p \otimes gx_q)=-\sum_{p,q=1}^n \gamma_{qp}(x_p \otimes gx_q)$}} 
hence to $\gamma_{pq}=-\gamma_{qp}$ for any $p,q=1, \ldots, n$, since the set $\lbrace gx_q\otimes x_{p} \ | \ p,q=1, \ldots, n \rbrace$ is linearly independent. 
\end{proof}

%Therefore, we have obtained the classification of the infinitesimal $\Rr$-matrices of $E(n)$, summarized in the following result. 
The next theorem summarizes the classification of infinitesimal $\Rr$-matrices of $E(n)$ that we have obtained. Note that each infinitesimal $\Rr$-matrix for $(E(n),R_A)$ is independent of the choice of the quasitriangular structure $R_A$ (given as in \eqref{tristruct}).

\begin{theorem}\label{prop:class-chi}
The Hopf algebra $E(n)$, endowed with any quasitriangular structure $R_A$, admits an exhaustive $n^2$-dimensional space of infinitesimal $\Rr$-matrices given by
\begin{equation}\label{eq:class-chi}
\chi=\sum_{p,q=1}^n \gamma_{pq}gx_p \otimes x_q.\end{equation}
Each $\chi$ can be associated to a square matrix $\Gamma=(\gamma_{pq})_{p,q=1, \ldots, n} \in M_n(\Bbbk)$.
The $E(n)$'s are Cartier only via infinitesimal $\Rr$-matrices $\chi$ whose associated matrices $\Gamma$ are anti-symmetric.
\end{theorem}

\begin{remark}
    In case of the Sweedler's Hopf algebra $E(1)$, the only infinitesimal $\Rr$-matrices are $\gamma gx_{1}\otimes x_{1}$, for $\gamma\in\Bbbk$, and there is no non-zero Cartier structure, cf. \cite[Proposition 2.7 and Remark 2.8]{ABSW}.
\end{remark}

\begin{remark}
    By Corollary \ref{cor:infinitesimalcop}, the fact that $(E(n),R_{A},\chi)$ is a (pre-)Cartier quasitriangular Hopf algebra is equivalent to the fact that $(E(n)^{\mathrm{cop}},R_{A}^{\mathrm{op}},\chi')$ is a (pre-)Cartier quasitriangular Hopf algebra, where $\chi':=-\sum_{p,q=1}^{n}{\gamma_{pq}x_{p}\ot gx_{q}}$.
    We also observe that, given an arbitrary quasitriangular bialgebra $(H,\Rr)$, the bialgebra $E(n)\ot H$ is pre-Cartier with quasitriangular structure $(\mathrm{Id}_{E(n)}\ot\tau_{E(n),H}\ot\mathrm{Id}_{H})(R_{A}\ot\Rr)$ and pre-Cartier structure $\chi^{i}\ot1_{H}\ot\chi_{i}\ot1_{H}$ by \cite[Proposition 2.15]{ABSW}.
\end{remark}

\begin{remark}
The following observation was made by M. Faitg, A. Gainutdinov and C. Schweigert, after our paper appeared online. In case $\Bbbk=\mathbb{C}$, the dimension of the space of infinitesimal braidings tangent to the braiding $\sigma_{E(n)}^{R_{A}}$ is computed in \cite[Proposition 5.5 and Remark 5.8]{FGS}, and the basis is obtained in \cite[Proposition 5.7]{FGS} only in case $A=\mathrm{Id}$ (and $\Bbbk=\mathbb{C}$). The results obtained are coherent with ours, remembering \cite[Remark 5.2]{FGS}. We point out that, in Theorem \ref{prop:class-chi}, we obtained the dimension and the explicit basis of the space of infinitesimal braidings for any quasitriangular structure $R_{A}$ and, moreover, the formula \eqref{eq:class-chi} is independent from $A$. Furthermore, we do not require $\Bbbk$ to have any particular property, apart from $\mathrm{char}\left( \Bbbk \right) \neq 2$.
\end{remark}

In Section \ref{sec:theory} we have observed that for any quasitriangular Cartier Hopf algebra $\left( H,\mathcal{R},\chi \right)$, if $\mathrm{char} (\Bbbk) \neq 2$, the infinitesimal $\Rr$-matrix $\chi$ is a $2$-coboundary in the  cohomology for coalgebras recalled in Section \ref{sub:Hoch}. Now, for $E(n)$ we also prove that any infinitesimal $\Rr$-matrix which is a $2$-coboundary satisfies the Cartier condition.

\begin{proposition}\label{prop:2cob}
	Let $\chi$ be an infinitesimal $\Rr$-matrix for $(E(n),R_{A})$. Then, $(E(n), R_A, \chi)$ is Cartier if and only if $\chi$ is a $2$-coboundary.    
\end{proposition}
\begin{proof}
Assume that $(E(n),R_{A})$ is Cartier with infinitesimal $\Rr$-matrix $\chi$. By Proposition \ref{cor:SSRchi}, we have that $\tau(S\otimes S)(\chi)=\chi$, hence $\chi$ is a 2-coboundary by \cite[Proposition 2.29]{ABSW}, since $\mathrm{char}(\Bbbk)\not=2$.

\begin{comment}
\begin{invisible}
DA ELIMINARE

By Theorem \ref{prop:class-chi} $\chi$ must be of the form 
%\eqref{eq:class-chi}, i.e. 
$\chi=\sum_{p,q=1}^n \gamma_{pq}gx_p \otimes x_q$, with $\gamma_{pq}\in \Bbbk$ such that $\gamma_{pq}=-\gamma_{qp}$, for any $p,q=1,\ldots,n$. Recall that the $2$-coboundaries for $E(n)$ are all the elements in $E(n)\otimes E(n)$ of the form $b^1(\alpha)=1\otimes\alpha-\Delta(\alpha)+\alpha\otimes 1$, for an element $\alpha\in E(n)$. Note that 
\[
\begin{split}
b^1(\gamma_{pq}x_px_q)&=\gamma_{pq}(x_px_q\otimes 1_H-\Delta(x_px_q)+1_H\otimes x_px_q)\\&=\gamma_{pq}(x_{p}x_{q}\otimes1_{H}-x_{p}x_{q}\otimes1_{H}-x_{p}g\otimes x_{q}-gx_{q}\otimes x_{p}-1_{H}\otimes x_{p}x_{q}+1_{H}\otimes x_{p}x_{q})\\&=
%-\gamma_{pq}\sum_{\emptyset\neq F\subsetneq \{p,q\} }(-1)^{S(F,P)}g^{\vert F\vert}x_{P\setminus F}\otimes x_F=
\gamma_{pq}gx_p\otimes x_q-\gamma_{pq}gx_q\otimes x_p.
\end{split}
\]
Thus, we obtain that 
	\begin{equation*}
		\begin{split}
b^1(\sum_{p,q=1}^n\gamma_{pq}x_px_q)&=\sum_{p,q=1}^n\gamma_{pq}(gx_p\otimes x_q)-\sum_{p,q=1}^n\gamma_{pq}(gx_q\otimes x_p)\\&\overset{(*)}{=}\sum_{p,q=1}^n\gamma_{pq}(gx_p\otimes x_q)+\sum_{p,q=1}^n\gamma_{qp}(gx_q\otimes x_p)\\&=2\chi,
\end{split}	
\end{equation*}
where $(*)$ holds true since $\gamma_{pq}=-\gamma_{qp}$ as $E(n)$ is Cartier. Since $\mathrm{char}\left( \Bbbk \right) \neq 2$, we get $\chi=b^1(\frac{1}{2}\sum_{p,q=1}^n\gamma_{pq}x_px_q)$ and so $\chi$ is a 2-coboundary.
\end{invisible}
\end{comment}

On the other hand, suppose that the infinitesimal $\Rr$-matrix $\chi=\sum_{p,q=1}^n \gamma_{pq}gx_p \otimes x_q$ is a $2$-coboundary, i.e., $\chi=b^1(\alpha)=1\otimes\alpha-\Delta(\alpha)+\alpha\otimes 1$, for an element $\alpha=\sum_{P,j=0}^1\alpha^j_Pg^jx_P$ in $E(n)$ where $P\subseteq \lbrace 1, \ldots, n \rbrace$. Note that, since  $$\Delta(\alpha)=\sum_{
P,j=0}^1\alpha^j_P\Delta(g^jx_P)=\sum_{P,j=0}^1\alpha^j_P\sum_{F\subseteq P}(-1)^{S(F,P)}g^{\vert F\vert+j} x_{P\setminus F}\otimes g^jx_F$$ while $\chi$ is sum of elements with only an $x_{i}$ for each tensorand, we must have $\alpha=\sum_{\underset{|P|=2}{j=0}}^1\alpha^j_Pg^jx_P$. Moreover, $j$ must be 0 in order to recover $1\otimes\alpha$ and $\alpha\otimes1$, so $\alpha=\sum_{|P|=2}{\alpha^0_Px_P}$.

Hence, using \eqref{antipode}, we obtain
%\[
%S(\alpha)=\sum_{\underset{|P|=2}{j=0}}^1{\alpha^j_PS(g^jx_P)}=\sum_{\underset{|P|=2}{j=0}}^1{\alpha^j_P(-1)^{|P|(j+1)}g^{|P|+j}x_P}=\sum_{\underset{|P|=2}{j=0}}^1\alpha^j_Pg^jx_P=\alpha.
%\]
\[
S(\alpha)=\sum_{|P|=2}{\alpha^0_PS(x_P)}=\sum_{|P|=2}{\alpha^0_P(-1)^{|P|}g^{|P|}x_P}=\sum_{|P|=2}{\alpha^0_Px_P=\alpha},
\]
so that we can conclude by Corollary \ref{corCartiercob}.
%\[
%\gamma:=m(S\otimes\mathrm{Id})(\chi)=m(S\otimes\mathrm{Id})%(1\otimes\alpha)-%m(S\otimes\mathrm{Id})\Delta(\alpha)+m(S\otimes\mathrm{Id})%(\alpha\otimes1)=\alpha-%\varepsilon(\alpha)1+\alpha=2\alpha.
%\]
\qedhere
\end{proof}

\begin{remark}
By \eqref{antipode} the Casimir element for the pre-Cartier quasitriangular Hopf algebra $(E(n),R_{A},\chi)$ is given by $\gamma=\sum_{p,q=1}^n\gamma_{pq}S(gx_p)x_q=\sum_{p,q=1}^n\gamma_{pq}x_px_q$.
\end{remark}

\subsubsection{An explicit example: the 8-dimensional unimodular ribbon Hopf algebra $E(2)$}

Here we consider the so-called 8-dimensional unimodular ribbon Hopf algebra $E(2)$, introduced in \cite{Rad}. Recall that $E(2)$ is a $8$-dimensional Hopf algebra given by generators $g,x_1,x_2$, and relations 
\[
g^2=1,\quad x_i^2=0,\quad gx_i=-x_ig,\quad x_ix_j=-x_jx_i,\quad\text{for}\ i,j=1,2.
\]
The Hopf algebra structure is defined by
\[\Delta(g)=g \otimes g, \quad  \Delta(x_i)=x_i \otimes 1 +g \otimes x_i,\quad \varepsilon(g)=1_{\Bbbk},\quad \varepsilon(x_i)=0
\]
for $i=1,2$ and the antipode is given by $S(g)=g$, $S(x_i)=-gx_i$, for $i=1,2$. 
The set $\lbrace g^jx_P \ | \ P\subseteq \lbrace 1, 2\rbrace, j \in \lbrace 0,1 \rbrace \rbrace$ is a basis of $E(2)$. %Moreover, we have
%$\Delta(gx_i)=gx_i\otimes g+1\otimes gx_i$, for $i=1,2$, $\Delta(x_1x_2)=x_1x_2\otimes 1-gx_1\otimes x_2+gx_2\otimes x_1+1\otimes x_1x_2$, and $\Delta(gx_1x_2)=gx_1x_2\otimes g-x_1\otimes gx_2+x_2\otimes gx_1+g\otimes gx_1x_2$.
%\par
The quasitriangular structure \eqref{tristruct} for $E(2)$ is	\begin{equation}\label{eq:R-E2}
		\begin{split}
		\Rr:=&\frac{1}{2}[(1 \otimes 1 + 1 \otimes g + g \otimes 1-g \otimes g)+\alpha(x_1\otimes x_1-x_1\otimes gx_1+gx_1\otimes x_1+gx_1\otimes gx_1)\\
		&+\beta (x_1\otimes x_2-x_1\otimes gx_2+gx_1\otimes x_2+gx_1\otimes gx_2)\\
		&+\gamma (x_2\otimes x_1-x_2\otimes gx_1+gx_2\otimes x_1+gx_2\otimes gx_1)\\
		&+\delta (x_2\otimes x_2-x_2\otimes gx_2+gx_2\otimes x_2+gx_2\otimes gx_2)\\
		&+(\beta\gamma -\alpha\delta) (x_1x_2\otimes x_1x_2+x_1x_2\otimes gx_1x_2+gx_1x_2\otimes x_1x_2 -gx_1x_2\otimes gx_1x_2)],
\end{split}
	\end{equation}
for $\alpha, \beta,\gamma,\delta\in \Bbbk$. This formula was obtained by S. Gelaki in \cite{Ge}.

We report the classification of the $2$-cocycles and infinitesimal $\Rr$-matrices for $E(2)$. \medskip

Let $\chi$ be an arbitrary element in $E(2)\otimes E(2)$. 
By 
Theorem \ref{prop:2cocycle}, 
we have that $\chi$ is a 2-cocycle if and only if 
\[
\begin{split}
   \chi&= b^{1}(\alpha 1+\beta g+\gamma x_{1}+\delta x_{2}+\eta gx_{1}+\zeta gx_{2}+\theta gx_{1}x_{2}+\lambda x_{1}x_{2})\\&+\mu(gx_{1}\otimes x_{1})+\nu(gx_{1}\otimes x_{2})+\xi(gx_{2}\otimes x_{2}) 
\end{split}
\]
for all $\alpha,\beta,\gamma,\delta,\eta,\zeta,\theta,\lambda,\mu,\nu,\xi\in\Bbbk$, i.e., if and only if 
\begin{equation}\label{eq:E(2)-2cocycle'}
\begin{split}
\chi&=\alpha(1\otimes 1)+\beta (1\otimes g+g\otimes 1-g\otimes g)+\gamma (1\otimes x_1-g\otimes x_1)+\delta (1\otimes x_2-g\otimes x_2) 
\\&+\eta (gx_1\otimes 1- gx_1\otimes g)+\zeta (gx_2\otimes 1- gx_2\otimes g)\\&
+\theta(1\otimes gx_1x_2-g\otimes gx_1x_2+x_1\otimes gx_2-x_2\otimes gx_1 +gx_1x_2\otimes 1-gx_1x_2\otimes g)\\&
+\mu(gx_1\otimes x_1)+\nu (gx_1\otimes x_2) 
+\xi(gx_2\otimes x_2). 
\end{split}
\end{equation}
By Theorem \ref{prop:class-chi} the infinitesimal $\Rr$-matrices for $E(2)$ are given by 
\begin{equation}\label{eq:chiE2}
\begin{split}
    \chi=
    \alpha gx_1\otimes x_1+\beta gx_1\otimes x_2+\gamma gx_2\otimes x_1+\delta gx_2\otimes x_2,
\end{split}
\end{equation}
with $\alpha,\beta,\gamma,\delta\in \Bbbk$. Moreover, the infinitesimal $\Rr$-matrices which also satisfy \eqref{eq:ctr2} are of the form $\chi=\kappa(gx_{1}\otimes x_{2}-gx_{2}\otimes x_{1})$, for $\kappa\in\Bbbk$. As shown in Proposition \ref{prop:2cob}, the infinitesimal $\Rr$-matrices which are also 2-coboundaries coincide exactly with those which give $E(2)$ a Cartier structure.

\section{The quantization problem}\label{sect:quant}
  
\noindent In this final section we solve the quantization problem given in \cite[Question 2.10]{ABSW} for $A_{C_{2}^n}$ and for $E(n)$. Let us recall that, given a bialgebra $H$, one can consider the corresponding \textit{trivial topological bialgebra} $H[[\hbar]]$ of formal power series with formal parameter $\hbar$, see e.g. \cite[Section XVI.4, Example 3]{Ka}. The bialgebra structure of $H[[\hbar]]$ is obtained by $\hbar$-linearly extending the bialgebra structure of $H$, where one has to replace the tensor product $\otimes$ with its topological completion $\hat{\otimes}$, so that $(H\otimes H)[[\hbar]]\cong H[[\hbar]]\hat{\otimes}H[[\hbar]]$. Throughout this section the tensor product of $H[[\hbar]]$ is $\hat{\otimes}$, which we just denote as $\otimes$ by abuse of notation.

\noindent Recall that in \cite[Proposition 2.12]{ABSW} it is shown that, given a pre-Cartier structure $(H,\Rr,\chi)$ on the Sweedler's Hopf algebra $H$, then $\tilde{\Rr}:=\Rr\mathrm{exp}(\hbar\chi) $ is a quasitriangular structure on the trivial topological bialgebra $H[[\hbar]]$. 

First, we prove the following general result.

\begin{proposition}\label{prop:quant}
Let $(H,\Rr,\chi)$ be a pre-Cartier quasitriangular bialgebra where $\chi^{n+1}=0$, for some $1\leq n\in\mathbb{N}$, and the following equations are satisfied:
\begin{enumerate}
    \item[$1)$] $\chi_{12}(\Rr^{-1}_{12}\chi_{13}\Rr_{12})=(\Rr^{-1}_{12}\chi_{13}\Rr_{12})\chi_{12}$, \medskip
    \item[$2)$] $\chi_{23}(\Rr^{-1}_{23}\chi_{13}\Rr_{23})=(\Rr^{-1}_{23}\chi_{13}\Rr_{23})\chi_{23}$.
\end{enumerate}
Then, there is a quasitriangular structure on $H[[\hbar]]$ given by $\tilde{\Rr}:=\Rr\mathrm{exp}(\hbar\chi)$.  
\end{proposition}
    
\begin{proof}
We show that $\tilde{\Rr}$ is a quasitriangular structure for $H[[\hbar]]$ checking \eqref{qtr1}, \eqref{qtr2} and \eqref{qtr3}, by using \eqref{qtr1}, \eqref{qtr2}, \eqref{qtr3} for $\Rr$ and \eqref{cqtr1}, \eqref{cqtr2}, \eqref{cqtr3} for $\chi$. First observe that $\tilde{\Rr}^{-1}=\mathrm{exp}(-\hbar\chi)\Rr^{-1}$. Indeed, $\tilde{\Rr}\tilde{\Rr}^{-1}=\Rr\mathrm{exp}(\hbar\chi)\mathrm{exp}(-\hbar\chi)\Rr^{-1}=\Rr\Rr^{-1}=1$ and, similarly, $\tilde{\Rr}^{-1}\tilde{\Rr}=1$. Moreover, we can compute
\[
\begin{split}
\tilde{\Rr}\Delta(\cdot)&=\Rr(1\otimes1+\hbar\chi+\cdots+\frac{\hbar^{n}}{n!}\chi^n)\Delta(\cdot)\overset{\eqref{cqtr1}}{=}\Rr(\Delta(\cdot)+\hbar\Delta(\cdot)\chi+\cdots+\frac{\hbar^{n}}{n!}\Delta(\cdot)\chi^n)\\&=\Rr\Delta(\cdot)(1\otimes1+\hbar\chi+\cdots+\frac{\hbar^{n}}{n!}\chi^n)\overset{\eqref{qtr1}}{=}\Delta^{\mathrm{op}}(\cdot)\Rr(1\otimes1+\hbar\chi+\cdots+\frac{\hbar^n}{n!}\chi^n)\\&=\Delta^{\mathrm{op}}(\cdot)\tilde{\Rr}
\end{split}
\]
and so \eqref{qtr1} is verified. Furthermore, we obtain
\[
\begin{split}
(\mathrm{Id}\otimes\Delta)(\tilde{\Rr})&=(\mathrm{Id}\otimes\Delta)(\Rr)(\mathrm{Id}\otimes\Delta)(1\otimes1+\hbar\chi+\cdots+\frac{\hbar^n}{n!}\chi^n)\\&\overset{\eqref{qtr2},\eqref{cqtr2}}{=}\Rr_{13}\Rr_{12}(1\otimes1\otimes1+\hbar(\chi_{12}+\Rr^{-1}_{12}\chi_{13}\Rr_{12})+\cdots+\frac{\hbar^n}{n!}(\chi_{12}+\Rr^{-1}_{12}\chi_{13}\Rr_{12})^{n})\\&=\Rr_{13}\Rr_{12}\mathrm{exp}(\hbar(\chi_{12}+\Rr^{-1}_{12}\chi_{13}\Rr_{12}))\overset{(*)}{=}\Rr_{13}\Rr_{12}\exp(\hbar\Rr^{-1}_{12}\chi_{13}\Rr_{12})\mathrm{exp}(\hbar\chi_{12})\\&=\Rr_{13}\Rr_{12}\Rr^{-1}_{12}\mathrm{exp}(\hbar\chi_{13})\Rr_{12}\mathrm{exp}(\hbar\chi_{12})=\Rr_{13}\mathrm{exp}(\hbar\chi_{13})\Rr_{12}\mathrm{exp}(\hbar\chi_{12})\\&=\Rr_{13}(1\otimes1\otimes1+\hbar\chi_{13}+\cdots+\frac{\hbar^n}{n!}(\chi_{13})^{n})\Rr_{12}(1\otimes1\otimes1+\hbar\chi_{12}+\cdots+\frac{\hbar^n}{n!}(\chi_{12})^{n})\\&=\Rr_{13}(1\otimes1\otimes1+\hbar\chi_{13}+\cdots+\frac{\hbar^n}{n!}(\chi^n)_{13})\Rr_{12}(1\otimes1\otimes1+\hbar\chi_{12}+\cdots+\frac{\hbar^n}{n!}(\chi^n)_{12})\\&=\tilde{\Rr}_{13}\tilde{\Rr}_{12},
\end{split}
\]
where $(*)$ follows by 1) and \cite[XVI.4 (4.6)]{Ka}. Thus, also \eqref{qtr2} is satisfied. Finally, we also have
\[
\begin{split}
    (\Delta\otimes\mathrm{Id})(\tilde{\Rr})&=(\Delta\otimes\mathrm{Id})(\Rr)(\Delta\otimes\mathrm{Id})(1\otimes1+\hbar\chi+\cdots+\frac{\hbar^n}{n!}\chi^n)\\&\overset{\eqref{qtr3},\eqref{cqtr3}}{=}\Rr_{13}\Rr_{23}(1\otimes1\otimes1+\hbar(\chi_{23}+\Rr^{-1}_{23}\chi_{13}\Rr_{23})+\cdots+\frac{\hbar^n}{n!}(\chi_{23}+\Rr^{-1}_{23}\chi_{13}\Rr_{23})^{n})\\&=\Rr_{13}\Rr_{23}\mathrm{exp}(\hbar(\chi_{23}+\Rr^{-1}_{23}\chi_{13}\Rr_{23}))\overset{(*')}{=}\Rr_{13}\Rr_{23}\exp(\hbar\Rr^{-1}_{23}\chi_{13}\Rr_{23})\mathrm{exp}(\hbar\chi_{23})\\&=\Rr_{13}\Rr_{23}\Rr^{-1}_{23}\mathrm{exp}(\hbar\chi_{13})\Rr_{23}\mathrm{exp}(\hbar\chi_{23})=\Rr_{13}\mathrm{exp}(\hbar\chi_{13})\Rr_{23}\mathrm{exp}(\hbar\chi_{23})\\&=\Rr_{13}(1\otimes1\otimes1+\hbar\chi_{13}+\cdots+\frac{\hbar^n}{n!}(\chi_{13})^{n})\Rr_{23}(1\otimes1\otimes1+\hbar\chi_{23}+\cdots+\frac{\hbar^n}{n!}(\chi_{23})^{n})\\&=\Rr_{13}(1\otimes1\otimes1+\hbar\chi_{13}+\cdots+\frac{\hbar^n}{n!}(\chi^n)_{13})\Rr_{23}(1\otimes1\otimes1+\hbar\chi_{23}+\cdots+\frac{\hbar^n}{n!}(\chi^n)_{23})\\&=\tilde{\Rr}_{13}\tilde{\Rr}_{23}
\end{split}
\]
where $(*')$ follows by 2) and \cite[XVI.4 (4.6)]{Ka}. Hence \eqref{qtr3} holds true.
\end{proof}

Now we apply Proposition \ref{prop:quant} to our examples.

\begin{corollary}
    Given the pre-Cartier structure $\chi_\alpha=\alpha x\otimes xg$, with $\alpha\in\Bbbk$, for the triangular Hopf algebra $(A_{C_{2}^n},\Rr_{a})$, where $\Rr_{a}$ is defined as in \eqref{eq:Ra}, we obtain a quasitriangular structure 
\[   
    \Rr_{a}\mathrm{exp}(\hbar\chi_{\alpha})=\Rr_{a}(1\otimes1+\hbar\chi_{\alpha})
\]    
on $A_{C_{2}^n}[[\hbar]]$.
\end{corollary}

\begin{proof}
Since $x^{2}=0$ then $\chi_{\alpha}^2=0$. Moreover, we trivially have
\[
\chi_{12}(\Rr^{-1}_{12}\chi_{13}\Rr_{12})=(\Rr^{-1}_{12}\chi_{13}\Rr_{12})\chi_{12}=\chi_{23}(\Rr^{-1}_{23}\chi_{13}\Rr_{23})=(\Rr^{-1}_{23}\chi_{13}\Rr_{23})\chi_{23}=0,
\] 
so Proposition \ref{prop:quant} applies.    
\end{proof}

\begin{corollary}
    Given the pre-Cartier structure $\chi=\sum_{p,q=1}^{n}{\gamma_{pq}gx_{p}\otimes x_{q}}$, with $\gamma_{pq}\in\Bbbk$, on the quasitriangular Hopf algebra $(E(n),\Rr)$, where $\Rr$ is defined as in \eqref{RstructureE(n)}, we obtain a quasitriangular structure 
\[    
    \Rr\mathrm{exp}(\hbar\chi)=\Rr(1\otimes1+\hbar\chi+\cdots+\hbar\chi^n)
\]    
    on $E(n)[[\hbar]]$.
\end{corollary}

\begin{proof}
    We just have to verify 1) and 2) of Proposition \ref{prop:quant}.
    As it is shown in the proof of Lemma \ref{lemmaCartier}, we have that
    \begin{align}\label{eq:ggchiR1}
    \Rr\chi&=-(\sum_{p,q=1}^{n}{\gamma_{p,q}x_{p}\otimes gx_{q}})\Rr=-(g\otimes g)\chi\Rr,\,\, \text{ and so }\\
\label{eq:ggchiR2}\chi\Rr^{-1}&=-\Rr^{-1}(g\otimes g)\chi.
    \end{align}
Moreover, we have
\begin{align}\label{eqchi13}
\chi_{13}(g\otimes g\otimes1)\chi_{12}&=-(g\otimes g\otimes1)\chi_{13}\chi_{12}=(g\otimes g\otimes 1)\chi_{12}\chi_{13},\,\, \text{ and }\\
\label{eqchi13'}
    \chi_{13}(1\otimes g\otimes g)\chi_{23}&=-(1\otimes g\otimes g)\chi_{13}\chi_{23}=(1\otimes g\otimes g)\chi_{23}\chi_{13}.
\end{align}
Then, we obtain
\[
(\Rr^{-1}_{12}\chi_{13}\Rr_{12})\chi_{12}\overset{\eqref{eq:ggchiR1}}{=}-\Rr^{-1}_{12}\chi_{13}(g\otimes g\otimes1)\chi_{12}\Rr_{12}\overset{\eqref{eqchi13}}{=}-\Rr^{-1}_{12}(g\otimes g\otimes1)\chi_{12}\chi_{13}\Rr_{12}\overset{\eqref{eq:ggchiR2}}{=}\chi_{12}(\Rr^{-1}_{12}\chi_{13}\Rr_{12})
\]
and 
\[
(\Rr^{-1}_{23}\chi_{13}\Rr_{23})\chi_{23}\overset{\eqref{eq:ggchiR1}}{=}-\Rr^{-1}_{23}\chi_{13}(1\otimes g\otimes g)\chi_{23}\Rr_{23}\overset{\eqref{eqchi13'}}{=}-\Rr^{-1}_{23}(1\otimes g\otimes g)\chi_{23}\chi_{13}\Rr_{23}\overset{\eqref{eq:ggchiR2}}{=}\chi_{23}(\Rr^{-1}_{23}\chi_{13}\Rr_{23}).
\]
\end{proof}

\noindent\textbf{Acknowledgements.} The authors would like to thank A.\@ Ardizzoni for meaningful suggestions regarding the classification of the infinitesimal $\Rr$-matrices for the Hopf algebras $E(n)$, M. Faitg, A. Gainutdinov and C. Schweigert for useful observations on the paper and A. Chirvăsitu, D. Ştefan and T.\@ Weber for nice discussions and valuable comments. %They also thank the referee for meaningful suggestions. 
This paper was written while LB and AS were members of the ``National Group for Algebraic and Geometric Structures and their Applications'' (GNSAGA-INdAM). FR would also like to extend his gratitude to the Department of Mathematics ``G. Peano'' of the University of Turin for the warm hospitality. All the authors were partially supported by the project funded by the European Union -NextGenerationEU under NRRP, Mission 4 Component 2 CUP D53D23005960006 - Call PRIN 2022 No.\, 104 of February 2, 2022 of Italian Ministry of University and Research; Project 2022S97PMY \textit{Structures for Quivers, Algebras and Representations (SQUARE).}

\end{document}